\documentclass[11pt]{amsart}
\usepackage[utf8]{inputenc}

\usepackage[abbrev,lite,alphabetic]{amsrefs}
\usepackage[colorlinks,plainpages,urlcolor=blue,linkcolor=reference,citecolor=citation,colorlinks = true, hyperfootnotes=false]{hyperref}
\hypersetup{citecolor=gray, linkcolor = darkgray, urlcolor= gray}

\usepackage[T1]{fontenc}
\usepackage[shortlabels]{enumitem}
\usepackage{moreenum}
\usepackage{subfiles}
 \usepackage{lmodern}
\usepackage{wrapfig}
\usepackage{csquotes}
\usepackage{comment}
\usepackage{booktabs}
\usepackage{bm}
\usepackage{tikz-cd}
\usepackage{mathrsfs}
\usepackage{sseq}
\usepackage{mathdots}
\usepackage{thm-restate}
\usepackage{amsmath}
\usepackage{amsxtra}
\usepackage{amscd}
\usepackage{amsthm}
\usepackage{amsfonts}
\usepackage{amssymb}
\usepackage{mathtools}
\usepackage[scr=rsfs]{mathalfa}
\usepackage{eucal}
\usepackage{bbm}
\usepackage[all,cmtip]{xy}

\usepackage{graphicx}
\usepackage{tikz}
\usepackage{morefloats}
\usepackage{pdfpages}
\usetikzlibrary{matrix,arrows,decorations.pathmorphing,quotes,angles,decorations.markings,shapes.misc}
\RequirePackage{color}

\usepackage[colorinlistoftodos, textsize=tiny]{todonotes}
\usepackage[margin=1in,marginparwidth=0.8in, marginparsep=0.1in]{geometry}
\usepackage[framemethod=TikZ]{mdframed}
\usepackage{cleveref}

\definecolor{reference}{rgb}{0.20,0.36,0.74}
\definecolor{citation}{rgb}{0,.40,.80}

\usepackage{lipsum}
\usepackage{cleveref}
\crefname{section}{\S \!\!}{\S\S \!\!}
\crefname{equation}{}{}
\crefname{enumi}{}{}
\crefname{appendix}{\S \!\!}{\S\S \!\!}  % fix appendix section names

\setenumerate[1]{label={(\arabic*)}}

\allowdisplaybreaks

\theoremstyle{plain}
\newtheorem{theorem}{Theorem}[section]
\newtheorem{mainthm}{Theorem}

\newtheorem{proposition}[theorem]{Proposition}
\newtheorem{lemma}[theorem]{Lemma}
\newtheorem{corollary}[theorem]{Corollary}
\newtheorem*{corollary*}{Corollary}

\newtheorem*{conjecture*}{Conjecture}
\theoremstyle{definition}

\newtheorem{definition}[theorem]{Definition}
\newtheorem*{definition*}{Definition}

\newtheorem{notation}[theorem]{Notation}

\newtheorem{remark}[theorem]{Remark}

\newtheorem{example}[theorem]{Example}
\newtheorem{warn}[theorem]{Warning}
\newtheorem*{example*}{Example}
\newtheorem*{remark*}{Remark}

\numberwithin{equation}{subsection}

\def\cA{\mathcal A}\def\cB{\mathcal B}\def\cC{\mathcal C}\def\cD{\mathcal D}
\def\cE{\mathcal E}\def\cF{\mathcal F}\def\cH{\mathcal H}
\def\cI{\mathcal I}\def\cJ{\mathcal J}
\def\cM{\mathcal M}\def\cN{\mathcal N}\def\cO{\mathcal O}\def\cP{\mathcal P}
\def\cQ{\mathcal Q}\def\cR{\mathcal R}\def\cS{\mathcal S}\def\cT{\mathcal T}
\def\cU{\mathcal U}\def\cW{\mathcal W}

\newcommand{\fX}{\mathfrak{X}}
\newcommand{\frM}{\mathfrak{M}}

\newcommand{\fg}{\mathfrak{g}}
\newcommand{\fgv}{\mathfrak{g}^\vee}
\newcommand{\frf}{\mathfrak{f}}
\newcommand{\ffv}{\mathfrak{f}^\vee}
\newcommand{\fd}{\mathfrak{d}}
\newcommand{\fdv}{\mathfrak{d}^\vee}

\makeatletter
\providecommand{\leftsquigarrow}{%
  \mathrel{\mathpalette\reflect@squig\relax}%
}
\newcommand{\reflect@squig}[2]{%
  \reflectbox{$\m@th#1\rightsquigarrow$}%
}
\makeatother

\newcommand{\Fun}{{\sf Fun}}
\newcommand{\Hom}{{\sf Hom}}

\newcommand{\Cat}{{\sf Cat}}

\renewcommand{\Pr}{{\sf Pr}}

\newcommand{\st}{{\sf st}}

\newcommand{\Bet}{{\sf Bet}}

\newcommand{\Mnd}{{\sf Mnd}}

\newcommand{\St}{{\sf St}}
\newcommand{\Mod}{{\sf Mod}}

\newcommand{\Perf}{{\sf Perf}}

\newcommand{\pt}{{\sf pt}}

\newcommand{\id}{{\sf id}}

\newcommand{\Maps}{{\sf Map}}

\newcommand{\bit}[1]{\textbf{\textit{#1}}}

\newcommand{\cofib}{{\sf cofib}}
\newcommand{\fib}{{\sf fib}}

\renewcommand{\AA}{\mathbb{A}}

\newcommand{\CC}{\mathbb{C}}

\newcommand{\EE}{\mathbb{E}}
\newcommand{\GG}{\mathbb{G}}
\newcommand{\LL}{\mathbb{L}}

\newcommand{\PP}{\mathbb{P}}
\newcommand{\RR}{\mathbb{R}}
\renewcommand{\SS}{\mathbb{S}}

\newcommand{\ZZ}{\mathbb{Z}}

\newcommand{\kk}{\Bbbk}

\newcommand{\Aut}{\mathsf{aut}}
\newcommand{\Auts}{{\mathsf{Aut}}}
\newcommand{\End}{\mathsf{end}}

\newcommand{\Coh}{\mathsf{Coh}}

\newcommand{\IndCoh}{\mathsf{IndCoh}}

\newcommand{\Sph}{\mathsf{Sph}}

\newcommand{\Adj}{\mathsf{Adj}}

\newcommand{\bT}{\mathbf{T}}

\newcommand{\Perv}{\mathsf{Perv}}
\newcommand{\PS}{\mathsf{PervCat}}

\newcommand{\Sh}{\mathsf{Sh}}
\newcommand{\Sing}{\mathsf{Sing}}

\newcommand{\muSh}{\mu\mathsf{Sh}}
\newcommand{\muPerv}{\mu\mathsf{Perv}}
\newcommand{\Loc}{\mathsf{Loc}}

\newcommand{\Ind}{{\sf Ind}}

\newcommand{\LocCat}{{\sf LocCat}}
\newcommand{\Rep}{{\sf Rep}}

\newcommand{\redu}{/\!\!/}

\newcommand{\HP}{{\sf HP}}
\newcommand{\HH}{{\sf HH}}

\newcommand{\heart}{\heartsuit}

\newcommand{\FI}{t}
\newcommand{\mass}{m}

\newcommand{\DQ}{{\sf DQ}}
\newcommand{\dR}{{\sf dR}}
\newcommand{\Dmod}{{\sf Dmod}}

\newcommand{\BM}{{\sf BM}}
\newcommand{\Oskel}{\mathbb{L}_G(\FI,\mass)}

\newcommand{\StAdj}{\overline{\Adj}}

\newcommand{\PervCat}{\mathsf{PervCat}}
\newcommand{\CohCat}{\mathsf{CohCat}}

\newcommand{\LBet}{\mathcal{L}_\Bet}

\newcommand{\sgn}{\alpha}
\newcommand{\sgnsub}{I}
\newcommand{\sgvect}{2^{[n]}}

\newcommand{\bdd}{\text{-bdd}}
\newcommand{\unbdd}{\text{-unbdd}}
\newcommand{\unstab}{\text{-unstab}}
\newcommand{\htvar}{\mathfrak{M}}
\newcommand{\twoO}{2\mathcal{O}}
\newcommand{\sfA}{\mathsf{A}}
\newcommand{\sfB}{\mathsf{B}}

\newcommand{\Cotan}{\mathbf{L}}
\newcommand{\rT}{\mathrm{T}}
\newcommand{\rN}{\mathrm{N}}
\newcommand{\Conorm}{\mathrm{N}^*}

\newcommand{\sC}{{\sf C}}
\newcommand{\Good}{{\sf Good}}
\newcommand{\bP}{\mathbf{P}}

\title{Hypertoric 2-categories $\mathcal{O}$ and symplectic duality}
\author{Benjamin Gammage and Justin Hilburn}
\begin{document}

\maketitle

\begin{abstract}
%For each hypertoric variety, we def use microlocal perverse (resp. coherent) 
We define 2-categories of microlocal perverse (resp. coherent) sheaves of categories on the skeleton of a hypertoric variety and show that the generators of these 2-categories
%sheaves of categories to produce a pair of 2-categories whose generators 
lift the projectives (resp. simples) in hypertoric category $\cO.$
We then establish equivalences of 2-categories categorifying the Koszul duality between Gale dual hypertoric categories $\cO.$ 
These constructions give a prototype for understanding symplectic duality via the fully extended 3d mirror symmetry conjecture.
%These constructions are meant to inaugurate a program of understanding Koszul dualities appearing in representation theory as arising from equivalences of 2-categories.

%Our 2-categories are defined as 2-categories of microlocal perverse and coherent sheaves of categories, respectively.

%The main technical innovations of the paper consist in beginning the development of a microlocal theory for such sheaves of categories.
%This paper is a sequel to \cite{GMH}, in which equivalences were constructed between coherent and perverse sheaves of categories on Gale dual toric-equivariant vector spaces. The 2-categories of this paper, by contrast, are to be understood as microlocal analogues of these. This paper is thus an advertisement for the as yet undeveloped theory of microlocalization for coherent and perverse sheaves of categories.
%The main technical result of this paper is a description of coherent-categorical microlocalization, which we combine with a new proposal for perverse-categorical microlocalization.
%In this paper, we begin to develop a theory of microlocalization for perverse schobers and for Ind-coherent sheaves of categories. We use this theory to produce a pair of 2-categories from the symplectic and algebraic geometry of Gale dual hypertoric varieties. We prove that these 2-categories are equivalent and that this equivalence categorifies Koszul duality for hypertoric categories $\cO.$ This is part of a series of works on recovering enumerative and representation-theoretic invariants in holomorphic symplectic geometry from the study of 2-categories.
\end{abstract}

\setcounter{tocdepth}{2}
\tableofcontents
 
\setcounter{section}{-1}

\section{Introduction}\label{sec:intro}
This paper inaugurates a program to understand features of categories $\cO$ in terms of a pair of categorifications arising from symplectic and algebraic geometry, realizing predictions from 3-dimensional supersymmetric gauge theory.
%{\color{blue}(Something about phyiscal motivation)}
\subsection{Overview}\label{sec:overview}
%The classical category $\cO$ was introduced in \cite{BGG} to study the representation theory of a semisimple Lie algebra $\fg.$ It addresses the difficulty that the category $\Rep(\fg)$ of all $\fg$-representations is too large to be manageable, whereas the category $\Rep^{\sf f.d.}(\fg)$ of finite-dimensional representations is too constrained to be useful. The BGG category $\cO$ contains all the finite-dimensional $\fg$-representations, but also contains enough projectives; as a result, it turns out to be a finite-length abelian category, and it possesses many other desirable properties.

The classical category $\cO$ was introduced in \cite{BGG} to study the representation theory of a semisimple Lie algebra $\fg.$ It addresses the difficulty that the category $\Rep(\fg)$ of all $\fg$-representations is too large to be manageable, whereas the category $\Rep^{\sf f.d.}(\fg)$ of finite-dimensional representations is too constrained to be useful. The BGG category $\cO$ contains the finite-dimensional $\fg$-representations and has enough projective objects but is still a finite-length abelian category. 

%by Bernstein-Gel'fand-Gel'fand 

%What we now call the BGG category $\cO$ for a semisimple Lie algebra $\fg$ was introduced in \cite{BGG} as a category of $\fg$-representations with several convenient properties. The Beilinson-Bernstein theorem \cite{BB81} gives a geometric interpretation of this category in terms of D-modules on the flag variety $G/B$ with characteristic variety contained in conormals to Schubert strata.

One of the many surprising properties of category $\cO$ is
%that it is controlled by a Koszul algebra $A=\End_{\cO}(P)$
the phenomenon of Koszul duality. For each category $\cO$ there is a dual category $\cO^\vee$ (which in this case is a category $\cO$ for the Langlands dual Lie algebra $\fgv$), such that the endomorphism algebra of the projectives in $\cO$ is equivalent to the Ext algebra of the simples in $\cO^\vee$. Moreover, within each category $\cO,$ the (derived) endomorphism algebras of projectives and simples are Koszul dual to each other. These relationships are summarized in the following diagram, where
%and these are related by a diagram of the following form, where 
the horizontal arrows are equivalences 
%{\color{red}(Do I need to mention a shear/regrading?)} 
(with gradings understood appropriately)
and the vertical arrows are Koszul dualities:
\begin{equation}\label{eq:Koszul-diagram1-s01}
\begin{tikzcd}
    \End_{\cO}(P)\arrow[r, leftrightarrow, "\sim"]\arrow[d, dashed, leftrightarrow, "\mathsf{KD}"'] & 
    \End_{\cO^\vee}(L^\vee)\arrow[d, dashed, leftrightarrow,"\mathsf{KD}"]\\
    \End_{\cO^\vee}(P^\vee)\arrow[r,leftrightarrow, "\sim"] &
    \End_{\cO}(L).
\end{tikzcd}
\end{equation}

%In \cite{BGS-Koszul}, this Koszul duality was explained as originating from mixed geometry: the BGG category $\cO$ can be understood geometrically as a category of perverse sheaves on the Schubert-stratified flag variety, and Koszul duality uses in an essential way the extra grading on this category coming from the graded lift to mixed Hodge modules.
%In these terms, the duality between $\cO$ and $\cO^\vee$ can be understood in terms of a hidden symmetry relating dual Schubert stratifications; as is usual in mirror symmetry and Langlands duality, this symmetry interchanges geometric properties in an interesting way, and playing both sides off of each other can yield many new insights about each.
Category $\cO$ may be understood geometrically as a category of perverse sheaves or regular D-modules on the Schubert-stratified flag variety.
In \cite{BPW}, the cotangent bundle to the flag variety (where the conormal to the Schubert stratification lives) was reconceptualized as but one (quite distinguished) example in the broader class of 
{\em conical symplectic resolutions} (CSRs),
including ADE surface singularities, Nakajima quiver varieties, and slices in the affine Grassmannian. In \cite{BLPW16} it was proposed that a category $\cO$, in analogy with the construction of \cite{BGG}, be associated to each of these spaces. 
More specifically, a category $\cO$ is associated not to a CSR $X$, but to a {\em sectorial} CSR $(X\supset\LL),$ containing the data of $X$ together with a Lagrangian subspace $\LL$ determined as the stable set for a Hamiltonian vector field.
The Koszul duality of \Cref{eq:Koszul-diagram1-s01} now manifests as a relationship, called \bit{symplectic duality} in \cite{BLPW16}, relating invariants of seemingly unrelated 
pairs $(X\supset\LL)$ and $(X^\vee\supset\LL^\vee).$
%symplectic resolutions in a surprising way.
%The BGG category $\cO$ was reconceptualized in \cite{BLPW16} as one example of a broader phenomenon involving symplectic resolutions

Subsequently, symplectic duality was reframed as a phenomenon in supersymmetric gauge theory. This process began with an observation of Gukov and Witten:
%that sectorial CSRs arise naturally as moduli spaces: 
given a massive 3-dimensional $\cN=4$ supersymmetric quantum field theory, 
the Higgs and Coulomb branches of its moduli space of vacua form a dual pair
of sectorial CSRs.
%namely, as the Higgs or Coulomb branches of the moduli of vacua in a massive 3-dimensional $\cN=4$ supersymmetric field theory.
From this perspective, symplectic duality is a consequence of \bit{3d mirror symmetry,} an involution on the space of 3d $\cN=4$ theories first observed in \cite{IS96}. As in the more familiar 2d mirror symmetry, there is a non-obvious correspondence between the invariants of dual theories; in particular, the Higgs branch of one theory is exchanged with the Coulomb branch of its dual.
%
%Thanks to an observation by Gukov and Witten, it became clear that symplectic duality should be closely related to \bit{3d mirror symmetry}, an involution on the space of $3d$ $\cN=4$ supersymmetric quantum field theories first observed in \cite{IS96}. Just as in the more famous 2d mirror symmetry, there is a non-obvious correspondence between invariants of one theory and invariants of its mirror. In particular, from each massive $3$d $\cN=4$ QFT one can extract a pair of sectorial CSRs -- the Higgs and Coulomb branches of the moduli space of vacua -- and these are exchanged by 3d mirror symmetry.

The physical meaning of category $\cO$ was clarified in \cite{BDGH}: Each 3d $\cN=4$ theory determines a pair of 3-dimensional topological field theories (TFTs) --- the 3d A and B-models\footnote{These theories are also known as 3d Donaldson-Seiberg-Witten theory and Rozansky-Witten theory, respectively.} --- which are exchanged by 3d mirror symmetry. 
After dimensionally reducing along a circle in the $\Omega$-background,\footnote{This may be understood mathematically as a periodic cyclic, i.e., Tate $S^1$-invariant, trace decategorification.}
the 3d A and B-models become equivalent. The category of boundary conditions for the resulting 2-dimensional TFT can be identified with the $\ZZ/2$-graded category $\cO$ for the Higgs branch --- or, by 3d mirror symmetry, for the Coulomb branch.\footnote{Other 2d reductions are discussed in \cite{BDGH}*{Section 7}, including one giving the graded lift of category $\cO$.} 

The process of dimensional reduction is lossy, and we propose instead to study the 3d A- and B-models directly, in terms of their 2-categories of boundary conditions.
%By the cobordism hypothesis \cites{Baez-Dolan,Lurie-cobordism} on expects that the 3d A and B-model are themselves determined by their 2-categories of boundary conditions. 
%In our case, each of these 2-categories should have basic objects given by holomorphic Lagrangians. 
The basic objects in each of these 2-categories are given by holomorphic Lagrangians, but otherwise the two 2-categories have a very different flavor.
Detailed expectations for the B-model 2-category have been given in \cites{KRS,KR,OR} and subsequently developed in many places; for our purposes, the most important are \cites{BZNP, Stefanich-QCoh,diFiore,Arinkin-notes}. On the other hand, the A-model 2-category is less well-understood. It should be a ``Fukaya-type'' 2-category, defined by counting $J$-holomorphic 2-disks and Fueter 3-disks, with Hom categories as studied in \cite{Doan-Rezchikov}, but serious analytic and conceptual difficulties have obstructed a full definition. This still hypothetical 2-category has also been studied by Bousseau \cite{Bousseau-holomorphicfloer} via its relation to 4d $\cN=2$ theories, and by Khan \cite{khan-thesis} from the perspective of the $\zeta$-instanton equation for the holomorphic area functional.

Inspired by similar theorems \cite{NZ,GPS3} for Fukaya categories,
we have conjectured in \cite{GMH} that the A-model 2-category is equivalent to one defined topologically via the perverse schobers of \cite{KS-schobers}. 
With 2-categories of boundary conditions for both the A- and B-models in hand, and guided by \cite{BDGH}, we are now able to state our prediction for the 2-categorical origins of category $\cO:$
\begin{conjecture*}
    To a sectorial conical symplectic resolution $(X\supset\LL)$, there is associated a pair of 2-categories
    \[
    \twoO_\sfA:=\mu\PervCat(\LL),\qquad
    \twoO_\sfB:=\mu\CohCat(\LL)
    \]
    of \bit{microlocal perverse (resp. coherent) schobers} on $\LL,$ with canonical generators $\cP$ and $\cS,$ respectively. After taking periodic cylic homology, these become the projectives and simples in 
    %such that periodic cyclic homology recovers 
    a $\ZZ/2$-graded version of category $\cO$:
    \[
    \HP(\End_{\twoO_\sfA}(\cP))\simeq \End_{\cO_{\ZZ/2}}(P),\qquad
    \HP(\End_{\twoO_\sfB}(\cS))\simeq \End_{\cO_{\ZZ/2}}(L).\qquad
    \]
\end{conjecture*}
The 2-categories $\twoO_{(\sfA/\sfB)}$ are much richer than the usual 1-category $\cO$, which sees only the geometry of the Lagrangian $\LL.$
%{\color{red}Explain why here.}
For instance, the $\EE_2$ (i.e., Drinfeld) center of $\twoO_\sfB$ is equivalent to the whole category $\Coh(X)$ of coherent sheaves on $X$; dually, the $\EE_3$ center of $\twoO_\sfA$ recovers the Coulomb branch construction from \cite{BFN1}. %{\color{red} Add here?}

In terms of the above 2-categories, we can now interpret symplectic duality %can now be interpreted 
as one consequence of a more fundamental statement about boundary conditions in dual theories.
%Following \cite{K94}, the following statement ought to be called ``Homological 3d mirror symmetry:''
\begin{conjecture*}[Fully extended 3d mirror symmetry]
    Let $(X\supset\LL)$ and $(X^\vee\supset\LL^\vee)$ be a dual pair of sectorial conical symplectic resolutions. Then there are equivalences of 2-categories
    \[
    \twoO_\sfA\simeq \twoO_\sfB^\vee,\qquad
    \twoO_\sfB\simeq \twoO_\sfA^\vee
    \]
    exchanging the canonical generators.
\end{conjecture*}
As a result, application of periodic cylic homology to the diagram
\begin{equation}\label{eq:2cat-Koszul-diagram-s01}
    \begin{tikzcd}
        \End_{\twoO_\sfA}(\cP)\arrow[r, leftrightarrow, "\sim"] &
        \End_{\twoO_\sfB^\vee}(\cS^\vee)\\
        \End_{\twoO_\sfA^\vee}(\cP)\arrow[r, leftrightarrow, "\sim"] &
        \End_{\twoO_\sfB}(\cS).
    \end{tikzcd}
\end{equation}
recovers the horizontal equivalences in \Cref{eq:Koszul-diagram1-s01}. These vertical Koszul dualities thus appear only after decategorifying from $2\cO$ to 1-category $\cO.$

In this paper, we prove the above conjectures in the case where $X,X^\vee$ are hypertoric varieties.

\subsection{Dual toric stacks}
Our story begins with the 3d mirror symmetry theorem formulated and proved in \cite{GMH}:
\begin{theorem}\label{thm:gmh-main}
    Let $\cS$ be the stratification of $\CC$ as $\CC=\CC^\times\sqcup 0.$
    Then there is an equivalence of stable 2-categories\footnote{To simplify notation here, we index our 2-categories here by the stratification $\cS$ rather than the union of Lagrangian conormals $\LL$ to $\cS,$ as we do elsewhere in the paper. For clarity, we also write ``sheaves of categories'' here for what elsewhere is often denoted by the more economical word ``schobers.'' More broadly, our notation throughout this paper differs from that in \cite{GMH}: We write $\PervCat$ (resp. $\CohCat$) for the 2-categories which were written there as $\Perv^{(2)}$ (resp. $\IndCoh^{(2)}$).}
    \begin{equation}\label{eq:gmh-basic}
    \PervCat_\cS(\CC)\simeq \CohCat^{\CC^\times}_\cS(\CC)
    \end{equation}
    between the 2-category of ``perverse sheaves of categories'' \cite{KS-schobers} on $\CC$ with a singularity at 0, and the 2-category of $\CC^\times$-equivariant ``coherent sheaves of categories'' \cite{Ari-talk} on $\CC$ with a singularity at $0.$

    More generally, let 
    \begin{equation}\label{eq:basic-exact-seq-intro}
\begin{tikzcd}
1 \arrow[r] & G \arrow[r]& (\CC^\times)^n \arrow[r] & F \arrow[r] & 1
\end{tikzcd}
    \end{equation}
    be an exact sequence of tori, and let $\cS$ denote the stratification of $\CC^n$ by coordinate hyperplanes. Then there is an equivalence of stable 2-categories
    \begin{equation}\label{eq:gmh-general}
       \PervCat^G_\cS(\CC^n)\simeq\CohCat^{F^\vee}_\cS(\CC^n).\footnote{Despite the apparent symmetry of the statement, note that the copies of $\CC^n$ on the two sides of \Cref{eq:gmh-general} are actually dual to each other, so that the right-hand $\CC^n$ is naturally a representation of $((\CC^\times)^n)^\vee$ and hence also $F^\vee.$}
    \end{equation}
\end{theorem}

The stratifications involved in the statements of \Cref{thm:gmh-main} should be understood as singular-support conditions inside the cotangent bundle of $\rT^*\CC^n$ --- or, more properly, inside the stacky cotangent bundles $\rT^*(\CC^n/G)$ and $\rT^*(\CC^n/F^\vee).$ A first goal of this paper is to describe the equivalence \Cref{eq:gmh-general} microlocally, so that we may restrict to stable loci inside the stacky cotangent bundles.

\subsection{Hypertoric varieties and Lagrangian skeleta}

Throughout the paper, we fix the choice of exact sequence \Cref{eq:basic-exact-seq-intro}.
The stack $\rT^*(\CC^n/G)$ can also be written as a stacky Hamiltonian reduction,
\begin{equation}\label{eq:cotangent-stack}
\rT^*(\CC^n/G) = \mu^{-1}(0)/G,
\end{equation}
where $\mu:\rT^*\CC^n\to \fgv$ is the moment map for the Hamiltonian $G$-action on $\rT^*\CC^n.$ We can pass from the stack \Cref{eq:cotangent-stack} to a smooth variety by GIT.
Let $\FI\in \fgv_\ZZ$ be a character for the group $G$. 
%(One may ultimately want to allow $\FI\in\fgv_\RR$ more generally, but here we restrict to $\fgv_\ZZ$ for simplicity.) 
The character $\FI$ specifies a $G$-equivariant line bundle $\cO(t)$ on $\rT^*(\CC^n/G)$ and hence on $\mu^{-1}(0),$ which allows us to perform a GIT quotient
\begin{equation}
    \htvar_G(\FI) := \mu^{-1}(0)/\!/_\FI G = \mu^{-1}(0)^{\FI-ss}/G.
\end{equation}

\begin{example} We highlight two special cases of the above construction:
    \begin{enumerate}
        \item For $G=\CC^\times$ embedded diagonally in $(\CC^\times)^n,$ at generic parameter $\FI$ the variety $\htvar_G(\FI)$ is $\rT^*\PP^{n-1}.$ 
        \item If $G$ is the kernel of the multiplication map $(\CC^\times)^n\to \CC^\times,$ then at generic parameter $\FI,$ the the variety $\htvar_G(\FI)$ is a resolution of the $A_{n-1}$ surface singularity $\CC^2/\ZZ/(n-1).$
    \end{enumerate}
\end{example}

The spaces $\htvar_G(\FI)$ were first introduced in \cite{Goto}, where they were called \bit{toric hyperk\"ahler manifolds}, and studied further in \cite{Konno} and in \cites{Biel-Dan,Haus-Sturm}; in \cite{Harada-Proudfoot} they were called \bit{hypertoric varieties}, to avoid confusion with toric varieties. These spaces are the simplest examples of symplectic resolutions, and often serve as a general testing ground for aspects of the general theory, including categories $\cO$ \cite{BLPW12} and Koszul duality \cite{BLPW10}, the quantum differential equation \cite{McBreen-Shenfeld}, and K-theoretic quasimap counts \cite{Smirnov-Zhou}.
 Following the program begun in \cite{GMH}, we seek to derive as much of this theory as possible from an analysis of 2-categories. In this paper, we will recover category $\cO,$ and the statement of Koszul duality for dual categories $\cO.$ 
 %in the hypertoric case. As we will explain below, this 2-categorical perspective leads to both computational and conceptual simplifications in the theory of Koszul duality.

In order to describe category $\cO$ for these varieties, it will be helpful to reframe the discussion in terms of a certain Lagrangian subspace.
\begin{definition}\label{def:Lagrangians-00}
    Let $\LL\subset \rT^*\CC^n$ be the singular Lagrangian given by the union of conormals to intersections of coordinate hyperplanes, and let $\LL_G$ be its image in $\rT^*(\CC^n/G).$
    %(To reduce clutter, we will also sometimes denote $\LL_G(0,0)$ simply by $\LL_G.$)
\end{definition}
The Lagrangian $\LL_G$ may be understood as a union of conormals to stacky hyperplane intersections in the stack $\CC^n/G$; alternatively, one may note that $\LL$ is contained in $\mu^{-1}(0)\subset \rT^*\CC^n,$ so that we can treat $\LL_G$ as its projection to $\mu^{-1}(0)/G.$
As a singular-support condition, the conic Lagrangian $\LL_G$ imposes the condition of local constancy along toric strata of $\CC^n/G$; we may therefore rewrite \Cref{eq:gmh-general} as the equivalence
\begin{equation}\label{eq:gmh-with-singularsupport}
    \PervCat_{\LL_G}(\CC^n/G) \simeq \CohCat_{\LL_{F^\vee}}(\CC^n/F^\vee)
\end{equation}
between perverse (resp. coherent) sheaves of categories with singular support along the Lagrangian $\LL_G$ (resp. $\LL_{F^\vee}$).

Our goal in this paper is to move from the Lagrangian $\LL_G,$ inside the stack $\rT^*(\CC^n/G),$ to a Lagrangian contained in the stable locus $\htvar_G(\FI).$ To this end, we will introduce a family of Lagrangians $\LL_G(\FI,\mass),$ depending on a pair of parameters. The first of these is the stability condition $\FI\in\fgv_\ZZ,$ and the latter is a cocharacter $\mass\in \frf_\ZZ$,
which determines an inclusion $\mass:\CC^\times\hookrightarrow F$ and therefore a Hamiltonian action of $\CC^\times$ on the symplectic stack $\rT^*(\CC^n/G)$. 
%More generally (treating $\FI$ as a real Hamiltonian reduction parameter and $\mass$ as specifying a Hamiltonian vector field which does not necessarily come from a $\CC^\times$ action), we may allow $\FI\in \fg_\RR$ and $\mass\in \frf_\RR$ to be chosen continuously, which we will do from now on.
%The two parameters in \Cref{def:Lagrangians-00} (currently set to 0) will control two modifications of this Lagrangian. The first of these we have already discussed: it is the stability parameter $\FI.$
%\begin{definition}
%    For nonzero $\FI\in \fgv_\ZZ,$ we will write 
%    \[
%    \LL_G(\FI,0):=\LL_G^{\FI-ss}=\LL_G\cap \mu^{-1}(0)^{\FI-ss}/G
%    \]
    %for the image of $\LL_G(0,0)\cap \mu^{-1}(0)^{\FI-ss}$ in $\mu^{-1}(0)^{\FI-ss}/G.$
%    for the set of \bit{$\FI$-semistable} points in $\LL_G.$
%\end{definition}
%The second parameter is a cocharacter $\mass\in X_\bullet(\frf)=\frf_\ZZ,$ 
%which determines an inclusion $\mass:\CC^\times\hookrightarrow F$ and therefore a Hamiltonian action of $\CC^\times$ on the stack $\rT^*(\CC^n/G)$ or variety $\htvar_G(\FI).$
\begin{definition}
    For $(\FI,\mass)\in \fgv_\ZZ\times \frf_\ZZ,$ we write 
    \[
    \LL(\FI,\mass):=\{x\in\LL_G^{\FI-ss}\mid \lim_{\lambda\to \infty} \mass(\lambda)\cdot x\text{ exists}\}
    \]
   for the subspace of $\LL_G$ consisting of those points which are \bit{$\FI$-semistable} and \bit{$\mass$-bounded}.
\end{definition}
When $\FI,\mass=0,$ we recover our original Lagrangian $\LL_G(0,0)=\LL_G.$ In general, the Lagrangian $\LL_G(\FI,\mass)$ is a locally closed subset of $\LL_G$: turning on $\FI$ deletes the closed subset of $\FI$-unstable points, while turning on $\mass$ deletes the open subset of $\mass$-unbounded points.
%\begin{remark}
%    Note that the calculation $\lim_{t\to \infty}\mass(t)\cdot x$ occurs in the stack $\rT^*(\CC^n/G),$ so that the bounded points may be larger than is naively expected. For instance, if $G=(\CC^\times)^n,$ then every point in $\LL_G$ is bounded for any choice of $\mass.$
%\end{remark}
\begin{remark}
    In the physics literature, $\FI$ is called the Fayet-Iliopolous (FI) parameter and $\mass$ the mass parameter. We prefer the names \bit{stability parameter} and \bit{attraction parameter}.
\end{remark}
%\begin{remark}
%    The parameters $\FI$ and $\mass$ play roles analogous to the K\"ahler and complex-structure parameters familiar from 2d mirror symmetry. Often, those parameters are allowed to take complex values, and the family of categories associated to choices of such parameters forms a local system away from the discriminant locus in moduli space, or (allowing singular parameters) a perverse schober over the whole moduli space.
%\end{remark}

Observe that the parameters $\FI$ and $\mass$ for $\LL_G$ play the opposite roles for $\LL_{F^\vee}.$ This is part of a general paradigm in mirror symmetry that parameters exchange roles under mirror duality (which in this case is implemented by dualizing the exact sequence \Cref{eq:basic-exact-seq-intro}). In order to lift this duality to the level of 2-categories, we will have to understand how the 2-categories in \Cref{eq:gmh-with-singularsupport} change when we replace $\LL_G$ (resp. $\LL_{F^\vee}$) with $\LL_G(\FI,\mass)$ (resp. $\LL_{F^\vee}(\mass,\FI)$). In other words, we need a good theory of {\em microlocalization} for perverse (resp. coherent) sheaves of categories.

On the A-side, we lack even a general non-microlocal theory for perverse sheaves of categories, so a full theory of microlocalization is out of reach. Nevertheless, inspired by the microlocal behavior of the 1-category of perverse sheaves, we can define directly our expectation for microlocal perverse sheaves of categories on the Lagrangian $\LL_G$. 
The key fact is categorical Kirwan surjectivity, which has been established for DQ-modules in \cite{BLPW16} and will be proven, in the hypertoric setting, for microlocal perverse sheaves in \cite{CGH}. This theorem implies that the category of microlocal perverse sheaves on $\LL_G(\FI,\mass)$ admits a description purely in terms of (non-microlocal) perverse sheaves.

We therefore imitate the construction of microlocal perverse sheaves on $\LL_G,$ which we recall in \Cref{sec:betti-cato}. For each component $\LL_G^\sgn$ of $\LL_G,$
we can define an object $\cP_G^\sgn$ in the 2-category $\PervCat_{\LL_G}(\CC^n/G)$ which corepresents the functor taking a perverse sheaf of categories to its microstalk at a smooth point of $\LL_G^\sgn.$ In addition,
if $\LL'_G\subset \LL_G$ is the closed embedding of a union of the components of $\LL_G,$ we have an embedding $\PervCat_{\LL'_G}(\CC^n/G)\hookrightarrow \PervCat_{\LL_G}(\CC^n/G).$
%we can define both a corresponding ``simple'' and ``projective'' object in the 2-category $\PervCat_{\LL_G}(\CC^n/G).$ (In geometric terms, these should correspond to the Lagrangian component itself, and its holomorphic cocore, respectively.)
\begin{definition}[\Cref{defn:betti-2cat-o}]
    The 2-category $\mu\PervCat(\LL_G(\FI,\mass))$ of \bit{microlocal perverse sheaves of categories} on the Lagrangian $\LL_G(\FI,\mass)$ is defined as a subquotient of the 2-category $\PervCat_{\LL_G}(\CC^n/G),$
%    the quotient
    \begin{equation}\label{eq:mupervcat-defn-intro}
    \mu\PervCat(\LL_G(\FI,\mass)):=
%    \frac{\PervCat_{\LL_G}(\CC^n/G)}{
%    \PervCat_{\LL_G^{\FI\unstab}}(\CC^n/G), 
%    \langle \cP_\sgn \rangle_{\sgn\in\mass\unbdd}
    \frac{\langle \cP_G^\sgn\rangle_{\sgn\in \FI\text{-ss}}}{
    \langle\cP_G^\sgn\rangle_{\sgn\in \mass\unbdd}
    }
    %{\langle\text{$\FI$-unstable simples, $\mass$-unbounded projectives}\rangle}
    \end{equation}
    %of $\PervCat_{\LL_G}(\CC^n/G)$ by the sub-2-category generated by the perverse sheaves of categories with $\FI$-unstable microsupport 
    %simple objects corresponding to $\FI$-unstable components of $\LL_G$ and 
    obtained by starting with the full sub-2-category on the objects $\cP^\sgn_G$ corresponding to $\FI$-semistable components of $\LL_G$ and quotienting by the objects $\cP^\sgn_G$ corresponding to $\mass$-unbounded components.
%    and the projective objects $\cP_\sgn$ corresponding to $\mass$-unbounded components $\LL_G^\sgn$ of $\LL_G.$
\end{definition}
\begin{remark}
    The 2-category $\mu\PervCat(\LL_G(\FI,\mass))$ defined in \Cref{eq:mupervcat-defn-intro} depends not only on the Lagrangian $\LL_G(\FI,\mass)$ but also on the ambient cotangent stack $\rT^*(\CC^n/G).$ 
    See \Cref{sec:Fueter} for a conjectural justification for this dependence in terms of generalized Seiberg-Witten theory.
    %conjectural explanation for this dependence in terms of symplectic vortices is mentioned in (??).
\end{remark}

%On the B-side, the theory is better developed, although unfortunately a full account is not yet available in the literature. 
%The center of the 2-category $\CohCat_{\LL_{F^\vee}}(\CC^n/F^\vee)$ has been computed in \cite{BZNP} to be a category of coherent sheaves on the (Betti) loop space of $\CC^n/F^\vee.$ By pulling back Lagrangians in $\rT^*(\CC^n/F^\vee)$ to coherent singular-support conditions on the loop space $\LBet(\CC^n/F^\vee),$ we explain in \Cref{sec:bside-muloc} how to microlocalize the 2-category $\CohCat_{\LL_{F^\vee}}(\CC^n/F^\vee)$ over open subsets of the skeleton $\LL_{F^\vee}$ and thus to define a 2-category $\mu\CohCat(\LL_{F^\vee}(\mass,\FI))$ of \bit{microlocal coherent sheaves of categories} on $\LL_{F^\vee}(\mass,\FI).$
On the B-side, more is known, thanks to works in progress of Dima Arinkin and German Stefanich. Suppose the Lagrangian $\LL_{F^\vee}$ is the image in $\rT^*Y$ of the conormal $\Conorm_XY$ of a map $X\to Y.$ Then we can present $\CohCat_{\LL_{F^\vee}}(\CC^n/F^\vee)$ as the 2-category with a single object, whose endomorphisms are given by the monoidal category $\Coh(X\times_Y X).$ 
In other words,  we have a presentation 
\[
\CohCat_{\LL_{F^\vee}}(\CC^n/F^\vee)\simeq \Mod_{\Coh(X\times_Y X)}(\St)
\] 
as the 2-category of module categories over $\Coh(X\times_Y X).$ 

Given a conic open subset $U\subset \LL_{F^\vee},$ using \Cref{defn:map-F} we produce a singular-support condition $\tilde{U},$\footnote{This is $F^{-1}(U)$ in the notation of \Cref{sec:bside}.}
in the sense of \cite{AG}, for coherent sheaves on $X\times_Y X.$
%is an open subset, then we will use $U$ to produce a singular-support condition $\tilde{U}$ for coherent sheaves on $X\times_Y X,$ in the sense of \cite{AG}.
\begin{definition}[\Cref{defn:cohcats-microlocal}]\label{defn:cohcats-microlocal-intro}
With $X\to Y$ as above, we define the 2-category of \bit{microlocal coherent sheaves of categories on $U$} as the 2-category
\[
\mu\CohCat(U):=\Mod_{\Coh_{\tilde{U}}(X\times_Y X)}(\St)
\]
of module categories for the monoidal category 
$\Coh_{\tilde{U}}(X\times_Y X)$ 
of coherent sheaves on the fiber product with singular support contained in $U$.
\end{definition}

To justify our notation, we prove that the 2-category so defined is invariant of the choice of map $X\to Y.$
\begin{proposition}[\Cref{prop:invariance-for-cohcats} and \Cref{prop:invariance-for-mucohcats}]
Let $X\to Y$ and $X'\to Y$ two maps whose conormals $\Conorm_XY$ and $\Conorm_{X'}Y$ have the same image in $\rT^*Y,$ and let $U\subset \rT^*Y$ be a conic open subset. Then the monoidal categories $\Coh_{\tilde{U}}(X\times_YX)$ and $\Coh_{\tilde{U}}(X'\times_Y X')$ are Morita equivalent.
\end{proposition}

\begin{remark}
    As on the A-side, the above definition of microlocalization is very much provisional. Ultimately, microlocal coherent sheaves should be global sections of a sheaf of categories defined by stipulating that its stalks are calculated by the procedure of \Cref{defn:cohcats-microlocal-intro}. That \Cref{defn:cohcats-microlocal-intro} recovers the true microlocalization in the case of hypertoric varieties would then follow from a ``2-categorical spectral Kirwan surjectivity'' theorem. We leave the full development of this coherent microlocal theory to future work.
\end{remark}

Our first main theorem establishes our predicted mirror relationship between the two 2-categories we have defined.
\begin{mainthm}\label{mainthm:2cats-O}
    There is an equivalence of stable 2-categories
    \begin{equation}\label{eq:2catsO-main-intro}
        \mu\PervCat(\LL_G(\FI,\mass)) \simeq \mu\CohCat(\LL_{F^\vee}(\mass,\FI))
    \end{equation}
   between the 2-categories of microlocal perverse (resp. coherent) sheaves of categories on $\LL_G(\FI,\mass)$ (resp. $\LL_{F^\vee}(\mass,\FI)$), identifying the canonical generators of these 2-categories.
\end{mainthm}

\begin{remark}\label{rem:coefficients}
    Although in the introduction we work over $\CC$
    %Although we have been working over $\CC$ everywhere 
    for the sake of symmetry, 
    \Cref{mainthm:2cats-O} remains true for an arbitrary choice of coefficients. As we explain in \Cref{sec:hyperplane}, the spaces on the right-hand side of \Cref{eq:2catsO-main-intro} can be defined for a general coefficient ring $\kk,$ which should then be taken as the coefficient ring for the microlocal perverse schobers we consider (which will still be defined 
    %the symplectic stacks on the spectral side (the home of the Lagrangian subspace $\LL_{F^\vee}(\mass,\FI)$ in equivalence \Cref{eq:2catsO-main-intro}) may be defined over a general coefficient ring $\kk$. \Cref{mainthm:2cats-O} will remain true if on the automorphic side we also take $\kk$ to be the coefficient ring for the microlocal perverse sheaves of categories we consider (which must still be defined 
    on the {\em complex} algebraic stack $\rT^*(\CC^n/G)$) on the left-hand side.
    This is part of the usual paradigm in mirror symmetry and the Langlands program; cf. the discussion in \cite{Dev-Satake} regarding the geometric Satake equivalence. Later, when we discuss category $\cO,$ we will specialize to coefficients $\kk=\CC$ to make contact with the original calculations of hypertoric category $\cO.$
\end{remark}

\begin{example}[\Cref{ex:2catO-A} and \Cref{ex:2catO-B}]\label{ex:2catso-intro}
    Let $G\simeq \CC^\times\hookrightarrow (\CC^\times)^2,$ embedded as the kernel of the multiplication map $(\CC^\times)^2 \xrightarrow{m} \CC^\times\simeq F.$ Then the 2-categories in \Cref{eq:2catsO-main-intro} can be described as follows:
    \begin{itemize}
        \item An object of $\mu\PervCat(\LL_G(\FI,\mass))$ is given by a spherical adjunction
        \[
        \begin{tikzcd}[column sep=2cm]
        \cC
        \arrow[yshift=0.9ex]{r}{F}
        \arrow[leftarrow, yshift=-0.9ex]{r}[yshift=-0.2ex]{\bot}[swap]{F^R}
        &
        \cD
        \end{tikzcd}
        \]
        equipped with the following extra data:
        \begin{itemize}
            \item An identification of the cotwist $\cofib(FF^R\to\id_\cD)$ with the 2-shift automorphism $[2]$;
            \item An extra monad $M$ on $\cC$;
            \item An identification of twists $\fib(\id_\cC\to M)\simeq \fib(\id_\cC\to F^RF).$
        \end{itemize}
            The Hom category between two such objects is the category of functors commuting with all the above structure.
        \item $\mu\CohCat(\LL_{F^\vee}(\mass,\FI))$ is given by the 2-category $\Mod_\cA(\St)$ of stable module categories over the monoidal category
        \[
        \cA:=\Coh\left( (\PP^1\sqcup 0)\times_{\PP^1} (\PP^1 \sqcup 0) \right).
        \]
        Using Koszul duality to write $\Coh(0\times_\PP^1 0)\simeq \Perf_{\kk[\beta]}$ where $\beta$ is a generator of degree 2, 
        and Beilinson's exceptional collection to identify $\Coh(\PP^1)$ with modules over the Kronecker quiver,
        we can rewrite $\cA$ as the matrix monoidal category
        \begin{equation}\label{eq:intro-firstexample-quivers}
        \cA\simeq \left(
        \begin{array}{ll}
        \Perf_{\bullet\rightrightarrows \bullet} & \Perf_\kk\\
        \Perf_\kk & \Perf_{\kk[\beta]}
        \end{array}
        \right),
        \end{equation}
        giving an explicit presentation of the generating objects, 1-morphisms, and 2-morphisms in the 2-category $\Mod_\cA(\St)$.
    \end{itemize}
%   Derivations of these 2-categories from the above definitions are given in \Cref{ex:2catO-A} and \Cref{ex:2catO-B}, respectively. 
   
   Under the equivalence \Cref{eq:2catsO-main-intro}, the generators of the rings whose perfect modules form the components of $\cA$ in \Cref{eq:intro-firstexample-quivers} correspond on the A-side to the structural morphisms among the spherical adjunctions and monads on the categories $\cC$ and $\cD$: namely, the arrows in the Kronecker quiver $\bullet\rightrightarrows\bullet$ correspond to the maps $T\rightrightarrows \id_\cC$ whose respective cofibers are the monads $M$ and $F^RF$ on $\cC,$ and the degree-2 map $\beta$ corresponds to the map $\id_\cD\to \cofib(FF^R\to \id_\cD)\simeq [2].$
\end{example}

%{
%\color{red} I am not sure that Beilinson's theorem is the best way to think about this. In particular for $\PP^n$ the data we get is just a quiver with two vertices and $n+1$-arrows such that a certain Koszul complex vanishes. However we also know how everything behaves under tensor product which lets us recover Beilinson's quiver if we wanted to.
%}

\subsection{Category $\cO$ and decategorification}\label{sec:decat}
As mentioned in \Cref{mainthm:2cats-O}, each of our 2-categories comes equipped with a canonical generator. Considered as an object in the automorphic (resp. spectral) 2-category, we will denote this generator by $\cP$ (resp. $\cS$). By taking their endomorphisms, we obtain monoidal categories
$\End_{\mu\PervCat(\LL_G(\FI,\mass))}(\cP)$ and
$\End_{\mu\CohCat(\LL_{F^\vee}(\mass,\FI))}(\cS).$

For a dualizable category $\cC,$ we write $\HH(\cC)$ for the \bit{Hochschild homology} of $\cC,$ which may be understood as the trace of the identity functor $\id_\cC.$ The Hochschild homology comes equipped with an $S^1$-action, for which we may take invariants $\HH^{S^1}(\cC)$ or coinvariants $\HH_{S^1}(\cC)$ to obtain cyclic homologies, or Tate invariants $\HH^{t S^1}(\cC)$ to obtain the \bit{periodic cyclic homology}
\[
\HP(\cC):=\HH^{t S^1}(\cC)\simeq\cofib(\HH_{S^1}(\cC)\to \HH^{S^1}(\cC)).
\]

The functor $\HP(-)$ (like $\HH(-)$) is symmetric monoidal, so that for $\cC$ a monoidal category, $\HP(\cC)$ is an algebra. We will be interested in applying the functor $\HP(-)$ to the endomorphism categories of the generators $\cP,\cS$ discussed above, and relating the resulting algebras to hypertoric categories $\cO.$

\begin{remark}
    For $\cC$ a monoidal category, the Hochschild homology algebra $\HH(\cC)$ (referred to in \cite{BZCHN} as the ``na\"ive trace'') is a first approximation to a more sophisticated ``categorical trace,'' in a sense which is made precise by \cite{GKRV}*{Theorem 3.8.5}. 
    %Full categorical traces have been computed for the 2-categories \Cref{eq:gmh-basic} in \cite{GH-Betti}. For 2-categories $\cO$
    The naive trace is sufficient for our purposes here, but we will return to the calculation of categorical traces in future work, extending the categorical trace computation in \cite{GH-Betti} for the 2-categories \Cref{eq:gmh-basic}.
    %The name ``na\"ive trace'' is meant to disambiguate $\tr(\cC)$ from the categorical trace $\Tr(\cC):=\cA\otimes_{\cA\otimes \cA^{rev}}\cA,$ where $\cA^{rev}$ denotes $\cA$ with reversed monoidal structure. By \cite{GKRV}*{Theorem 3.8.5}, the na\"ive trace can be regarded as a first approximation to the categorical trace.
    %The tools developed in \cites{BZCHN,Chen-Koszul} should be powerful enough to calculate the full categorical trace and center of our B-side 2-category, but we leave these determinations to future work, as we are primarily interested here in relation to category $\cO.$
\end{remark}

As we have mentioned in \Cref{sec:overview}, 
%the natural origin of category $\cO$ is the world of 3d $\cN=4$ gauge theories.
categories $\cO$ are supposed to arise physically from reductions of 3d $\cN=4$ gauge theories.
From this perspective, one class of input data from which one may produce a category $\cO$ is a stack of the form $\rT^*(V/G),$ where $V$ is a representation of the reductive group $G$, together with a pair of parameters specifying, respectively, a GIT stable locus $\frM\subset\rT^*(V/G)$ and a Hamiltonian $\CC^\times$ action on this stable locus. The stable set for this $\CC^\times$ action, inside the GIT-stable locus, is a Lagrangian subspace $\LL.$ 

Given the data of $(\frM,\LL),$ one may define a pair of categories:
\begin{itemize}
\item The \bit{Betti category $\cO$} is the category $\cO^\Bet(\frM, \LL):=\muPerv_{\LL}(\LL)$ of microlocal perverse sheaves on $\LL.$ 
%and 
\item The \bit{de Rham category $\cO$} is a category $\cO^{\dR}(\frM,\LL):=\DQ^{{\sf reg}}(\LL)$ is a category of regular holonomic DQ-modules on $\LL.$
\end{itemize}
%\begin{definition}[\Cref{defn:betti-cato} and \Cref{defn:derham-cato}]
%    Let $\frM, \LL$ as above.
%    \begin{itemize}
%        \item The \bit{Betti category $\cO$} is the category
%        \[
%            \cO^\Bet(\frM, \LL):=\muPerv_{\LL}(\LL)
%        \]
%        of microlocal perverse sheaves on $\LL.$ (As explained in \cite{CGH}, this is the Fukaya category of the sector $(\frM,\LL).$)
%        \item The \bit{de Rham category $\cO$}
%        \[
%            \cO^\dR(\frM,\LL):=\DQ^{{\sf reg}}(\LL)
%        \]
%    \end{itemize}
%    of regular deformation-quantization modules on $\LL.$\footnote{One needs to be careful about }
%\end{definition}
%The \bit{category $\bm{\cO}$} associated to this data is a category of either microlocal sheaves or regular DQ-modules supported on this Lagrangian $\LL.$ These categories are related (over $\CC$) by a Riemann-Hilbert correspondence, but we will distinguish them by referring to them as the \bit{Betti} (resp. \bit{de Rham}) \bit{category $\bm{\cO}$}, respectively.

%We write $\cO^{(*)}(\htvar_G(\FI),\mass)$ for the category $\cO$ associated to the hypertoric variety $\htvar_G(\FI)$ and attraction parameter $\mass,$ where $(*)\in \{\Bet,\dR\}$ indicates the flavor (Betti or de Rham) of category $\cO$ we consider. 

Each of these categories $\cO^{(*)}(\htvar,\LL)$ has a collection of simple objects $L_\sgn,$ with projective covers $P_\sgn.$ We write $L:=\bigoplus L_\sgn$ and $P:=\bigoplus P_\sgn$ for the direct sum of the simples (resp. projectives). 
We will use periodic cyclic homology to recover their endomorphism algebras,
%Let $\cO_{\htvar,\LL}$ denote the hypertoric category $\cO$ first defined in \cite{BLPW12} (redefined for a wider class of spaces in \cite{BLPW16}). 
%We prove that the periodic cyclic traces of the 2-categories \Cref{eq:2catsO-main-intro} recover categories $\cO,$ 
up to collapsing the cohomological grading modulo 2.
In the following theorem, we write
%we will be interested in 2-periodic versions of these categories, which we denote by
$\cO^{(*)}_{\ZZ/2}(\htvar,\LL):=\cO^{(*)}(\htvar,\LL)\otimes_{\Mod_\kk}\Mod_{\kk((u))}$ for the 2-periodization of category $\cO.$
%Below, we write $\tilHP$ for the ``categorified periodic cylic homology'' functor (defined in \Cref{defs:categorifed-hochschild}) which takes as input a $k$-linear 2-category and outputs a $k((u))$-linear category, where $u$ is a variable of degree 2.
\begin{mainthm}\label{mainthm:hp-decategorification}
Let $\kk=\CC.$
There are equivalences of algebras
\begin{align}
\label{eq:mainthm-decat-a}
\HP(\End_{\mu\PervCat(\LL_G(\FI,\mass))}(\cP))&\simeq 
\End_{
\cO^\Bet_{\ZZ/2}(\htvar_G(\FI),\LL(\FI,\mass))
}
(P)
,\\
\label{eq:mainthm-decat-b}
\HP(\End_{\mu\CohCat(\LL_{F^\vee}(\mass,\FI))}(\cS))&\simeq 
\End_{
\cO^{\dR}_{\ZZ/2}(\htvar_{F^\vee}(\mass),\LL_{F^\vee}(\mass,\FI))
}
(L).
\end{align}
%where we write $\cO^{\ZZ/2}$ for a 2-periodization of category $\cO.$ 
%Moreover, the equivalences \Cref{eq:mainthm-decat-a} and \Cref{eq:mainthm-decat-b} send the canonical generators 
%of $\mu\PervCat(\LL_G(\FI,\mass))$ (resp. $\mu\CohCat(\LL_{F^\vee}(\mass,\FI))$) 
%of the 2-categories on the LHS 
%to the projective (resp. simple) objects in category $\cO.$
\end{mainthm}
%\begin{remark}
%    {\color{red}Add remark here explaining that statements involving the de Rham category may fail to carry over to general coefficients $\kk,$ since they involves Riemann-Hilbert / Chern character.}
%\end{remark}
\begin{example}\label{ex:decategorification-intro}
    We return to the 2-categories discussed in \Cref{ex:2catso-intro}. In this case:
    \begin{itemize}
        \item The category \Cref{eq:mainthm-decat-a} is the (2-periodized) category of diagrams of vector spaces
        \[
        \begin{tikzcd}
            C \arrow[r, shift left, "x"] & D\arrow[l, shift left, "y"]
        \end{tikzcd}
        \]
        with the following relations:
        \begin{itemize}
            \item The sphericality condition on $F$ imposes that the linear maps $1_D-xy$ and $1_C-yx$ are invertible.
            \item The exact triangle $FF^R\to \id_\cD\to [2]$ decategorifies to the relation $xy=0.$ (This implies the invertibility conditions above.)
        \end{itemize}
        (The extra monad $M$ decategorifies to $yx$ and thus does not contain any new information.) We thus recover the usual quiver description for the category $\cO$ associated to $\widetilde{\CC^2/\ZZ/2}$ (which coincides with the classical BGG category $\cO$ for $SL(2)$).
        \item The category \Cref{eq:mainthm-decat-b} is a 2-category of modules over the ring $\HP(\cA)$, where $\cA$ is as in \Cref{ex:2catso-intro}. Using the identification
        \cite{Preygel-loopspaces}*{Theorem 1.1.2}
        \[
        \HP(\cA)=\HP
        \Coh\left( (\PP^1\sqcup 0)\times_{\PP^1} (\PP^1 \sqcup 0) \right)\simeq
        C_*^{BM}\left( (\PP^1\sqcup 0)\times_{\PP^1} (\PP^1 \sqcup 0) \right)_{\ZZ/2},
        \]
        we find that this category is the subcategory of (2-periodized) D-modules on $\PP^1$ generated by the D-modules $\cO_{\PP^1}$ and $\delta_0.$ 
        These are the simple objects in category $\cO$ for $\rT^*\PP^1,$ which once again agrees with the classical category $\cO$ for $SL(2).$
    \end{itemize}
\end{example}
\begin{remark}
    Although we have just explained that our 2-category (in the case described above) categorifies category $\cO$ for $SL(2),$ our hypertoric 2-category in this case should not be understood as the 2-category $\cO$ associated to $G=SL(2).$ That 2-category on the B-side should be equivalent not to coherent sheaves of categories on $G/B=\PP^1$ but rather to coherent sheaves of categories on $U\backslash G/B = \GG_a\backslash \PP^1.$ The 1-category, being topological, is insensitive to the quotient by the unipotent group, but the 2-category will notice it.
\end{remark}

\begin{remark}
    The decategorification in \Cref{mainthm:hp-decategorification} only recovers 2-periodized versions of category $\cO$. One might hope for an enhanced version of \Cref{mainthm:hp-decategorification} which recovers categories $\cO$ without collapsing to a $\ZZ/2$-grading. Such a result could be achieved by constructing a graded lift of the 2-categories $\cO$ we define in this paper. 
    We leave this question to future work.
    %On the B-side, the extra grading is expected to arise from cotangent fiber scaling, as in the theory of filtered Koszul duality on stacks \cite{Chen-Koszul}. The more mysterious A-side grading may require a mixed Hodge-theoretic enhancement of the theory of perverse schobers.
\end{remark}

%\subsection{Koszul duality}
%{\color{red}Add more about Koszul duality here.}
For a fixed hypertoric variety $\htvar_G(\FI),$ the two flavors of 2-category $\cO$ we define, namely $\mu\PervCat(\LL_G(\FI,\mass))$ and $\mu\CohCat(\LL_G(\FI,\mass))$ are very different, and certainly very far from being equivalent.
However, one surprising feature of \Cref{mainthm:hp-decategorification} is 
that they agree after $\HP$-decategorification:
%an identification of the two flavors of category $\cO$ for the same hypertoric variety:
\begin{corollary}\label{mainthm:koszul}
    There is an equivalence of categories
    \begin{align*}\label{eq:decat-samespace}
    \Perf_{\HP(\End_{\mu\PervCat(\LL_G(\FI,\mass))}(\cP))}
    &\simeq
    \cO^\Bet_{\ZZ/2}(\htvar_G(\FI),\LL_G(\FI,\mass))
    \\
    &\simeq 
    \cO^{\dR}_{\ZZ/2}(\htvar_G(\FI), \LL_G(\FI,\mass))\simeq
    \Perf_{\HP(\End_{\mu\CohCat(\LL_G(\FI,\mass))}(\cS))}.
    \end{align*}
\end{corollary}
The above identification, given by the Riemann-Hilbert correspondence $\cO^\Bet\simeq \cO^{\dR}$, is a categorification of the fact that both matrix factorizations $\mathsf{MF}(X,f)$ and the Fukaya-Seidel category $\mathsf{FS}(X,f)$ categorify the vanishing cohomology $H^*(X,\varphi_f)$ of a function $f.$ 

\begin{remark}
\Cref{mainthm:koszul} shows that 
%\Cref{eq:decat-samespace} can be understood as supplying the missing diagonal arrows in \Cref{eq:2cat-Koszul-diagram-s01}.
the Koszul dualities among categories $\cO$ appear only after decategorification:
after applying $\HP,$ the equivalence of \Cref{mainthm:koszul}
supplies diagonal arrows to the diagram \Cref{eq:2cat-Koszul-diagram-s01}.
In other words, the miraculous fact underlying Koszul duality between dual categories $\cO$ is that the A- and B-type 2-categories $2\cO$ for the same space, though a priori unrelated, have the same decategorification --- the 1-category $\cO$ --- with canonical generators going to projectives (resp. simples) of this 1-category.
%Our equivalence of 2-categories thus categorifies the Koszul duality equivalence of \cite{BLPW10}. 
%We now give some more background on the 2-categories $\Athree$ and $\Bthree$ involved in this categorification, beginning with the more familiar $\Bthree.$
We expect that essentially all the instances of Koszul duality appearing in geometric representation theory arise via decategorifications of this form (at least after passing to a $\ZZ/2$ grading). 
\end{remark}
%We give a more thorough account of this phenomenon in the companion paper {\color{red}(I think we should write this -- a short paper doing the examples of periodic cyclic categorical traces of dualities for $\rT^*\CC^\times \leftrightarrow \rT^*(B\CC^\times)$ and $\rT^*\CC\leftrightarrow \rT^*(\CC/\CC^\times)$)}.
%In future work, we hope to give a more systematic account of such appearances of Koszul duality from fully extended 3d mirror symmetry.

%\subsection{Further decategorifications}
%\ {\color{red} (Fill this in with all the expected statements about decategorifications of our 2-categories, especially categorical traces and centers, related invariants, and possibly higher decategorifications on 2- or 3-manifolds.)}

\subsection{Predictions for Fueter theory}\label{sec:Fueter}
Conjecturally, $\mu\PervCat(\LL_G(\FI,\mass))$ calculates the 2-category of boundary conditions for gauged Fueter theory (i.e., generalized 3d Seiberg-Witten theory \cite{Taubes-generalizations}) associated to $\rT^*(\CC^n/G)$, with generic real FI and mass parameters $\FI,\mass$. 
Note that this theory is {\em not} the same as the A-twisted sigma model (i.e., ungauged Fueter theory) on the stable locus $\htvar(\FI)\subset \rT^*(\CC^n/G)$: the gauge theory (even with generic FI parameter) receives contributions from the unstable locus. This explains the dependence of the 2-category $\mu\PervCat(\LL_G(\FI,\mass))$ on the stack $\rT^*(\CC^n/G)$ and not just the variety $\htvar(\FI).$ Indeed, the factorization homology of this 2-category on a closed 2-manifold is predicted in \cite{BFK-surface} to recover information about the moduli space of quasimaps, which depends on the stacky presentation of $\htvar_G(\FI).$
%This is in some ways analogous to the discrepancy, familiar from 2d mirror symmetry, between the sigma model to a non-Fano toric variety and its incarnation as a gauged linear sigma model. 

So far, the Fueter 2-category associated to a symplectic variety or stack $\htvar,$ equipped with Lagrangian skeleton $\LL,$ is not well-understood or even defined (although preliminary work in the ungauged case is available in \cite{Doan-Rezchikov}). Heuristically, the objects of this 2-category should be holomorphic Lagrangians in $\htvar$; 1-morphisms should be intersection points between these (after a displacement by a wrapping Hamiltonian); and 2-morphisms should be $J_\theta$-holomorphic curves between these. 

Our 3d mirror symmetry comparison \Cref{mainthm:2cats-O} thus gives us detailed predictions for the behavior of the holomorphic Lagrangians in the gauged Fueter 2-category of $\rT^*(\CC^n/G)$. The canonical generators of the 2-category $\mu\PervCat(\LL_G(\FI,\mass))$ are expected to correspond to Lagrangian cocores or linking disks to the components of the skeleton $\LL_G(\FI,\mass)),$ and \Cref{mainthm:2cats-O} computes the Hom categories (and compositions) among these.

\begin{example}\label{ex:fueter}
    We return to the situation of \Cref{ex:2catso-intro}, where 
    \[
    \LL_G(\FI,\mass) = \PP^1 \cup \rT^*_\infty \PP^1 \subset \rT^*\PP^1\subset \rT^*(\CC^2/\CC^\times).
    \]
    The 2-category $\mu\PervCat(\LL_G(\FI,\mass))$ has two canonical generators, corepresenting microstalk functors at the two components of $\LL_G(\FI,\mass),$ and we computed their endomorphism category to be
    \[
    \cA\simeq \left(\begin{array}{ll}
    \Perf_{\bullet \rightrightarrows \bullet} & \Perf_\kk\\
    \Perf_\kk & \Perf_{\kk[\beta]}
    \end{array}\right).
    \]

    Up to remembering the degree shift in the bottom-right corner, we can write the matrix $\cA$ entirely in terms of representations over quivers: $\cA\simeq \Perf_\cQ,$ where
    \[
    \cQ\simeq \left(\begin{array}{cc}
    \bullet\rightrightarrows \bullet & \bullet\\
    \bullet & \bullet \circlearrowleft
    \end{array}\right).
    \]
    This presentation makes evident that each of the four Hom categories in $\cA$ has distinguished generating objects (the vertices) and generating 1-morphisms (the arrows).
    It is natural to conjecture that these correspond to the (wrapped) intersection points and the $J_s$-holomorphic disks, respectively, among Lagrangian cocores in the Fueter 2-category.
    %although we acknowledge that we have so far not given any jsutfi
\end{example}
%\begin{remark}
%    Note that \Cref{ex:fueter} gives a conjectural calculation not of the Fueter theory of $T^*\PP^1$ but rather of the gauged Fueter theory of $T^*\CC^2.$ In \Cref{ex:2catso-intro}, the embedding of $T^*\PP^1$ into $T^*\CC^2$ contributed the extra monad $M$, arising geometrically from the interaction of $T^*_\infty\PP^1$ with the unstable Lagrangian $T^*_{0/\CC^\times}(\CC^2/\CC^\times).$ That monad contributed one of the arrows in the Kronecker quiver appearing in the upper-left corner of $\cQ$ above. We therefore predict that the endomorphism category of the generators of the Fueter 2-category for $T^*\PP^1$ itself should be given by
%    \[
%    \cA':=\left(
%    \begin{array}{cc}
%        \Perf_{\bullet\to\bullet}&\Perf_\kk\\
%        \Perf_\kk & \Perf_{\kk[\beta]}
%    \end{array}
%    \right).
%    \]
%\end{remark}

\subsection{Notation and conventions}\label{sec:notation}
Our categorical conventions are largely the same as in \cite{GMH}, and we refer to \S A there for more details. 
We use the ``implicit $\infty$'' convention: all constructions are understood in the homotopy-coherent sense; in particular, $(n-)$category always means $(\infty,n)$-category. 
Stable categories are always assumed $\kk$-linear for a fixed coefficient ring $\kk.$ 
%(We state our results in the case $\kk=\CC$ to simplify the statements, but see \Cref{rem:coefficients} for the adaptation to a general coefficient ring $\kk.$)
%We write $\Cat_n$ (resp. $\St_n$, $\Pr^{L,\St}_n$) for the $(n+1)$-category of small (resp. small stable, presentable stable) $n$-categories; in the case $n=1$ we often omit the index and simply write $\Cat$ (resp. $\St,\Pr^{L,\St}$). 
We write $\St$ for the stable 2-category of small ($\kk$-linear) stable categories. 
%As with $\Coh$ and $\Perv$ above, unless otherwise specified our notation always refers to 
%We write $\Sigma_+^{(\infty,n)}:\Cat_n\to\St_n$ for the n-categorical stabilization functor, which is left adjoint to the forgetful functor $\St_n\to\Cat_n.$ 

By default, the categories we discuss are derived categories rather than the hearts of their t-structure. If $\cC$ is a category with t-structure, we write $\cC^\heart$ for its abelian heart.
So for instance, for $X$ an algebraic stack, we write $\Coh(X)$ for the derived category of coherent sheaves on $X$. Throughout the paper, ``stack'' will always mean a derived stack which is quasi-compact, almost of finite presentation, and with affine diagonal.

For $\Lambda\subset \rT^*X$ a Lagrangian, we write $\Perv_\Lambda(X)$ for the category of compact objects inside the derived category of the heart of the perverse t-structure on the presentable stable category of sheaves on $X$ with singular support in $\Lambda.$ 
(Thus, our usage of $\Perv_\Lambda(X)$ differs from the usual one in two ways: first, we treat this as a derived category rather than the heart of its t-structure, and second, we allow some objects with infinite-dimensional microstalks.) We retain the analogous convention for the category $\muPerv_\Lambda(\Lambda)$ of microlocal perverse sheaves.

%{\color{red} (Add something about $\Mod$ vs $\Perf$?)}

%{\color{red} (Conditions on ideals in monoidal categories.)}
Ideals in a monoidal category will always be assumed to be idempotent-complete full subcategories, closed under finite limits and colimits.

For a torus $G,$ we write $\fg$ for its Lie algebra (with dual $\fg^\vee)$, $\fg^\vee_\ZZ$ for the character lattice (and $\fg_\ZZ$ for the cocharacter lattice), and $\fg^\vee_\RR:=\fg^\vee_\ZZ\otimes \RR,$ $\fg_\RR:=\fg\otimes \RR.$ When $G$ has an action on a symplectic variety, the moment map is naturally valued in $\fg^\vee$. If $G$ is defined over $\CC,$ we write $G_c$ for the maximal compact subgroup; if $G_c$ has a Hamiltonian action on a (real) symplectic manifold, the moment map is valued in $\fg^\vee_\RR.$

%{\color{red} Possibly something about shifted cotangents, conormals, \&c.}

{\bf Notational index:}
In \Cref{table:notation} we collect some notation used frequently to describe the geometry of hypertoric varieties.
In the following list, we write
    $\sgn \in \sgvect$ for a sign vector, and $\sgnsub\subset \sgvect$ for a collection of sign vectors.
    
{ %\setlength{\tabcolsep}{9pt}
    \renewcommand{\arraystretch}{1.3}
    \begin{table}[h]
    \begin{tabular}{@{}ll@{}}\toprule
         Notation & Definition \\\midrule
    $X_G^\sgn$  &  $\{z_i=0\mid i\notin \sgn\}/G\subset \AA^n/G$ \\
    $X_G$  &  $\bigsqcup_{\sgn\in\sgvect} X_G^\sgn$ \\
    $X_G^\sgnsub$ & $\bigsqcup_{\sgn\in\sgnsub} X_G^\sgn$ \\
    $\LL_G^\sgn$  &  $\Conorm_{X_G^\sgn}(\AA^n/G)\subset \rT^*(\AA^n/G)$ \\
    $\LL_G$  &  $\bigcup_{\sgn\in\sgvect}\LL_G^\sgn$ \\
    $\LL_G^\sgnsub$ & $\bigcup_{\sgn\in\sgnsub}\LL_G^\sgn$ \\
%    $\mrLL_G^\sgn$  &  $\LL_G^\sgn\setminus\bigcup_{\beta\neq \sgn}\LL_G^\beta\cap\LL_G^\sgn$ \\
%   $\cA_G$  &  $\Coh\left(\left(\bigsqcup_{\sgn\in \sgvect} X_G^\sgn\right)\times_{\CC^n/G}\left(\bigsqcup_{\sgn\in \sgvect} X_G^\sgn\right)\right)$ \\ 
%   $\cA_G^\sgnsub$  &  $\Coh\left(\left(\bigsqcup_{\sgn\in \sgnsub} X_G^\sgn\right)\times_{\CC^n/G}\left(\bigsqcup_{\sgn\in\sgnsub } X_G^\sgn\right)\right)$ \\ 
%    $\CohCat_{\LL_G}(\CC^n/G)$  &  $\Mod_{\cA_G}(\St)$ \\
%    $\CohCat_{\LL_G^\sgnsub}(\CC^n/G)$ & $\Mod_{\cA_G^\sgnsub}(\St)$ \\
%    $\PervCat_{\LL_G}(\CC^n/G)$  &  $(\Sph^{\otimes n})^G$ \\
\bottomrule
\\
    \end{tabular}
    \caption{Index of notation for the basic subvarieties and conormal Lagrangians considered in this paper}
    %, monoidal categories, and 2-categories considered in this paper.}
    \label{table:notation}
    \end{table}
}

{\bf Acknowledgements}
The authors are grateful to Aaron Mazel-Gee for collaborating on early stages of this project. BG would like to thank Elden Elmanto and Adeel Khan for suggestions about periodic cyclic homology and Laurent C\^ot\'e for many long discussions about microlocal perverse sheaves and category $\cO,$ and for collaboration on the companion paper \cite{CGH}. JH would like to thank Mat Bullimore, Tudor Dimofte, and Davide Gaiotto for their collaborations on \cite{BDGH}.
He would also like to thank Kevin Costelo, Niklas Garner, Wenjun Niu, and Philsang Yoo for many interesting discussions about $3$d $\mathcal{N}=4$ theories and thank Aleksander Doan, Ahsan Khan, and Semon Rezchikov for interesting discussions about Fueter maps.

BG acknowledges the support of an NSF Postdoctoral Research Fellowship, DMS-2001897. 

JH is part of the Simons Collaboration on Homological Mirror Symmetry supported by Simons Grant 390287. This research was supported in part by Perimeter Institute for Theoretical Physics. Research at Perimeter Institute is supported by the Government of Canada through the Department of Innovation, Science and Economic Development Canada and by the Province of Ontario through the Ministry of Research, Innovation and Science.

\section{Hyperplanes and hyperk\"ahler manifolds}\label{sec:hyperplane}
Throughout this paper, we fix an exact sequence of tori
\begin{equation}\label{eq:basic-exact-sequence}
\begin{tikzcd}
1 \arrow[r] & G \arrow[r,"i"]& D \arrow[r,"p"] & F \arrow[r] & 1,
\end{tikzcd}
\end{equation}
as well as an equivalence $D\simeq (\GG_m)^n.$ 
%We will always write $\AA^n$ for the standard representation of either $D$ or its dual torus $D^\vee$. (Where the distinction matters, it will be clear from context which torus we consider.)
\subsection{Hyperplane arrangements}
%The combinatorics of our constructions will be encoded in hyperplane arrangements, as we now recall. 
%The data defining a hypertoric variety may be encoded in the combinatorial language of hyperplane arrangements. 
From the exact sequence \Cref{eq:basic-exact-sequence} and a choice of parameters $(\FI,\mass)\in\fgv_\RR\times\frf_\RR,$ we will define a polarized hyperplane arrangment in $\frf^\vee_\RR$ --- that is, a hyperplane arrangement in $\frf^\vee_\RR$ together with an affine-linear functional on $\frf^\vee_\RR.$ 
We begin with the basic case where the parameters are 0 and $G=\{e\}$ is trivial:
\begin{definition}
    We write $\cH$ be the cooriented hyperplane arrangement in $\fdv_\RR\simeq\RR^n$ given by the $n$ coordinate hyperplanes, with their natural coorientations.
\end{definition}
Observe that the exact sequence \Cref{eq:basic-exact-sequence} induces an exact sequence
%\[
%\begin{tikzcd}
%1 \arrow[r]& \mathfrak{g}_{\RR} \arrow[r,"i_{\RR}"]& \mathfrak{d}_{\RR} \arrow[r,"p_{\RR}"]& \mathfrak{f}_{\RR} \arrow[r]& 1
%\end{tikzcd}
%\]
%and
\[
\begin{tikzcd}
0 &\arrow[l] \mathfrak{g}^{\vee}_{\ZZ}& \arrow[l,"i^\vee_{\ZZ}"'] \mathfrak{d}^{\vee}_{\ZZ} &\arrow[l,"p^\vee_{\ZZ}"'] \mathfrak{f}^{\vee}_{\ZZ} &\arrow[l] 0.
\end{tikzcd}
\]
Given $\FI\in \fg^\vee_\RR,$ we write
\[
\ffv_\RR(\FI):=(i_\RR^\vee)^{-1}(t)\subset \fdv_\RR
\]
for the affine-linear subspace of $\fdv_\RR$ obtained as a translation of the subspace $\ffv_\RR$ by $\FI.$ We will assume that $\ffv_{\RR}(\FI)$ is not contained in any of the hyperplanes in the arrangement $\cH$.

%of cocharacter and character lattices, respectively.
%Choose an identification $D \cong (\CC^\times)^n$ and hence dual identifications $\mathfrak{d}_{\ZZ} \cong \oplus_{i=1}^n \ZZ e_i$ and $\mathfrak{d}^{\vee}_{\ZZ} \cong \oplus_{i=1}^n \ZZ e^i$. In the rest of this section we will abbreviate tensor products $(-) \otimes_{\ZZ} \RR$ by changing the subscript $\ZZ$ to $\RR$.
%
%Let $\FI \in \mathfrak{g}^{\vee}_{\RR}$ be a stability parameter. The coordinate arrangement on $\mathfrak{d}^{\vee}_{\RR} \simeq\RR^n$ restricts to a hyperplane arrangement on $(i^{\vee}_{\RR})^{-1}(\FI)$. 

\begin{definition}\label{defn:hyperplane-arrangement}
    Given $(\FI,\mass)\in \fg^\vee_\RR\times\frf_\RR,$ we define a polarized hyperplane arrangement $(\cH_G(\FI),\mass)$,
    where $\cH_G(\FI)$ is the hyperplane arrangement on $\frf_\RR^\vee(\FI)$ obtained by intersecting with the coordinate hyperplane arrangement $\cH,$ and
    the polarization $\mass$ is the affine-linear functional on
    $\frf_\RR^\vee(\FI)$ given by $\mass\in\frf_\RR$.
    %for the hyperplane arrangement on $\ffv_\RR(\FI)$ obtained by intersecting with the coordinate hyperplane arrangement $\cH.$ 
    %We write $\cH_G^\sgn(\FI)$ for the restriction of the chamber $\cH^\sgn$ to $\ffv_\RR(\FI).$
\end{definition}

%To a hypertoric variety $\htvar_G(\FI)$ is associated a hyperplane arrangment in the real vector space $\ffv_\RR.$
%In the case where $G=\{e\}$ is trival, we have $\ffv_\RR = \fdv_\RR \simeq \RR^n$, and the hyperplane arrangment will be the arrangment $\cH$ defined above. 
We will frequently need to refer to the chambers in these hyperplane arrangements. 
Observe that one of the $2^n$ chambers in the cooriented hyperplane arrangement $\cH$
%Such a chamber 
may be specified by labeling which of the $n$ coordinates on $\RR^n$ is positive and which is negative.
\begin{definition}
    A \bit{sign vector} is an element of the set $\sgvect.$ Writing $2=\{+,-\},$ we can think of a sign vector as a length $n$ word in $\{+,-\},$ hence the name. We will also think of sign vectors as subsets of $[n],$ with $(+,\ldots,+)$ corresponding to the full set $[n]=\{1,...,n\}$ and $(-,\ldots,-)$ corresponding to the empty set.

    For a sign vector $\sgn,$ we write $\cH^\sgn$ for the corresponding chamber of $\cH,$ and we write $\cH_G^\sgn(\FI)$ for the restriction of this chamber to $\frf_\RR^\vee(\FI).$
\end{definition}

\begin{definition}\label{defn:unstable-unbounded}
    We say that a sign vector $\sgn\in\sgvect$ is \bit{$\FI$-unstable} if the chamber $\cH_G^\sgn(\FI)$ does not appear in $\cH_G(\FI).$\footnote{In \cite{BLPW10}, these sign vectors are instead called \bit{infeasible}, using the language of linear programming. We prefer throughout to use language referencing the geometry of the hypertoric variety.}
%\end{definition}
%
%Now note that a choice of $\mass\in \frf_\RR$ determines a linear functional on $\ffv_\RR\simeq \ffv_\RR(\FI).$
%\begin{definition}

    We say that a sign vector $\sgn\in \sgvect$ is \bit{$\mass$-bounded} if the restriction of $\mass$ to the chamber $\cH_G^\sgn(\FI)$ (or equivalently, to the chamber $\cH_G^{\sgn}(0)$) is bounded.
\end{definition}
%In general,
Recall the following features of hyerplane arrangements:
\begin{definition}
    A hyperplane arrangement in $\RR^k$ is \bit{simple} if any $k$ hyperplanes intersect in a single point and if the intersection of any $k+1$ hyperplanes is empty.
    
    The arrangement $\cH_G(\FI)$ is \bit{unimodular} if there is some basis for $\fg_{\ZZ}$ where $p^{\vee}_{\ZZ}$ is given by a totally unimodular integer matrix.
\end{definition}

%
%The space $(i^{\vee}_{\RR})^{-1}(\FI)$ is a torsor for $\mathfrak{f}^{\vee}_{\RR}$. 
%Choosing a lift $\tilde{\FI}$ of $\FI$ along $i^{\vee}_{\ZZ}$ determines an isomorphism
%\begin{align*}
%(i^{\vee}_{\RR})^{-1}(\FI) &\cong \mathfrak{f}^{\vee}_{\RR} \\
%x &\mapsto x-\tilde{\FI}.
%\end{align*}

\begin{definition}
    We say that the parameters $(\FI,\mass)\in \fgv_\RR\times \frf_\RR$ 
    are \bit{generic} 
    if they satisfy the following conditions:
    \begin{itemize}
        \item $\FI$ is chosen so that the hyperplane arrangement $\cH_G(\FI)$ is simple.
        \item $\mass$ is chosen so that it is not constant on any 1-dimensional flat of $\cH_G(\FI)$ (or equivalently, on any 1-dimensional flat of $\cH_G(0)$).
    \end{itemize}
\end{definition}

\subsection{Gale duality}
Consider the exact sqeuence of tori obtained from \Cref{eq:basic-exact-sequence} by taking duals:
\begin{equation}
    \begin{tikzcd}
        1\arrow[r] &
        F^\vee \arrow[r,"p^\vee"]&
        D^\vee \arrow[r, "i^\vee"]&
        G^\vee \arrow[r]&
        1.
    \end{tikzcd}
\end{equation}
We may repeat the constructions of \Cref{sec:hyperplane} with $G$ replaced by $F^\vee.$ Now we use $\mass\in \frf$ as the parameter determining a translation of the subspace $\fg_\RR\subset \fd_\RR\simeq \RR^n,$ and by intersecting with the coordinate hyperplane arrangement in $\RR^n$ we obtain a hyperplane arrangement $\cH_{F^\vee}(\mass)$ in $\fg_\RR.$

\begin{definition}
    We say that the polarized hyperplane arrangements $(\cH_G(\FI),\mass)$ and $(\cH_{F^\vee}(\mass), \FI)$ are \bit{Gale dual} to one another.
\end{definition}

The basic theorem of linear programming is that the concepts of instability and unboundedness are related to each other by this duality.
\begin{theorem}[\cite{BLPW10}*{Theorem 2.4}]
    A sign vector $\sgn\in \sgvect$ is $\FI$-unstable for $(\cH_G(\FI),\mass)$ if and only if it is $\FI$-unbounded for $(\cH_{F^\vee}(\mass),\FI).$
\end{theorem}

\subsection{Toric hyperk\"ahler manifolds}
Before moving to hyperk\"ahler manifolds, we begin with toric symplectic algebraic stacks. To the exact sequence \Cref{eq:basic-exact-sequence}, we associate the stack
\[
\htvar_G:=\rT^*(\AA^n/G),
\]
where $D$ is the torus of diagonal automorphisms of $\AA^n,$ and $G$ acts on $\AA^n$ as a subtorus of $D$. This stack may be understood as the global quotient
\[
\rT^*(\AA^n/G) = \mu_G^{-1}(0)/G,
\]
where $\mu_G:\rT^*\AA^n\to \fgv$ is the moment map for the Hamiltonian action of the torus $G.$ This moment map factors as the composite
\[
\rT^*\AA^n\xrightarrow{\mu_D}\AA^n\simeq \fdv \xrightarrow{i^\vee} \fgv,
\]
with the first map given by
\[
\mu_D(x_1,\ldots,x_n, y_1,\ldots,y_n) = (x_1y_1,\ldots,x_ny_n).
\]

\begin{definition}
    Let $\FI\in \fgv_\ZZ$.
    The \bit{hypertoric variety}\footnote{If the hyperplane arrangement $\cH_G(\FI)$ is not unimodular, this is a slight misnomer, as we will treat $\htvar_G(\FI)$ as a DM stack in this case.}
    associated to the exact sequence \Cref{eq:basic-exact-sequence} and the parameter $\FI$ is the GIT quotient
    $\htvar_G(\FI):=\mu_G^{-1}(0)\redu_\FI G.$
    Equivalently, we may write $\htvar_G(\FI)=(\rT^*(\CC^n/G))^{\FI-ss}\subset \rT^*(\AA^n/G)$ as the $\FI$-semistable locus inside the stack $T^*(\AA^n/G).$
    When $\kk=\CC,$ the underlying complex-analyic space of $\htvar_G(\FI)$ is called a \bit{toric hyperk\"ahler manifold.}
\end{definition}

\begin{proposition}[\cite{Biel-Dan}*{Theorems 3.2 \& 3.3}]
    If $\FI$ is generic, then the space $\htvar_G(\FI)$ is a smooth Deligne-Mumford stack. If moreover the hyperplane arrangement $\cH_G(\FI)$ is unimodular, then there are no strictly semistable points, and $\htvar_G(\FI)$ is a smooth variety.
\end{proposition}

When $\kk=\CC,$ the complex-analytic spaces $\htvar_G(\FI)$ were first introduced in \cite{Goto}, and further studied in \cites{Konno, Biel-Dan,Haus-Sturm}.% We summarize some of their basic properties here.
As these spaces have traditionally been studied as complex manifolds, they are often constructed not as GIT quotients but as hyperk\"ahler quotients, as we now recall.

In the complex setting, not only is the action of $G$ on $T^*\CC^n$ Hamiltonian (with moment map $\mu_G$) for the holomorphic symplectic form on $T^*\CC^n,$
but the action of the compact torus $G_c$ is also Hamiltonian for the real K\"ahler form on $T^*\CC^n = \CC^{2n}.$ The real moment map for this action is the composite
\[
T^*\CC^n \xrightarrow{\mu_{D_c}} \RR^n\simeq \fdv_\RR\xrightarrow{i^\vee_\RR}\fgv_\RR,
\]
with the first map given by
\[
\mu_{D^c}(x_1,\ldots,x_n,y_1,\ldots,y_n) = 
\left(|x_1|^2-|y_1|^2, \ldots, |x_n|^2-|y_n|^2\right).
\]
%Given a character $\FI\in \fgv_\ZZ$ of $G$, we can produce a space $\htvar_G(\FI)$ which admits the following equivalent descriptions.
%\begin{enumerate}
%    \item $\htvar_G(\FI)= \rT^*\CC^n\redu_\FI G$ is the complex symplectic reduction, at parameter $t$, of the complex symplectic variety $\rT^*\CC^n$ by the Hamiltonian action of the torus $G.$ By definition, this means that $\htvar_G(\FI)$ admits a presentation as the GIT quotient $\mu_G^{-1}(0)\redu_\FI G.$
%    \item Alternatively, to stay completely within the realm of real symplectic geometry, we may think of the above construction as the real symplectic reduction $\mu_G^{-1}(0)\redu_\FI G_c$ of $\mu_G^{-1}(0)$ by the Hamiltonian action of the compact real form $G_c$ of the torus $G.$
%    \item $\htvar_G(\FI) = (\rT^*(\AA^n/G))^{\FI-ss}\subset \rT^*(\AA^n/G)$ is the $\FI$-semistable locus inside the stack $\AA^n/G.$
%\end{enumerate}
%The equivalence of these descriptions is a standard application of the Kempf-Ness theorem.
%
%{\color{red} (Expand on the following examples.)}
The following is a standard application of the Kempf-Ness theorem: see for instance \cite{Nakajima-Hilbertschemes}*{Ch. 3}. 
\begin{lemma}
    For $\FI$ generic, there is an isomorphism of complex manifolds (or orbifolds if $\cH_G(\FI)$ is not unimodular)
    \[
    \htvar_G(\FI)\simeq \left(\mu_G^{-1}(0)\cap\mu_{G^c}^{-1}(\FI)\right)/G_c.
    \]
\end{lemma}
%\begin{example}
%    Let $D=(\GG_m)^{n+1},$ and $T=\ker\left(D\to \GG_m\right)$ be the ``antidiagonal'' torus given by the kernel of the multiplication map $(z_1,\ldots,z_{n+1})\mapsto z_1\cdots z_{n+1}.$ Then for generic $\FI,$ $\htvar_G(\FI)$ is the resolution of the $A_n$ singularity.
%\end{example}
%\begin{example}
%Let $D=(\GG_m)^{n+1},$ and $T = \GG_m\hookrightarrow D$ be the diagonal subtorus. Then for nonzero $\FI,$ we have $\htvar_G(\FI) = \rT^*\PP^n.$
%\end{example}

%\begin{definition}
%    The variety $\htvar_G(\FI)$ defined above is the \bit{toric hyperk\"ahler manifold} (or \bit{hypertoric variety}) associated to the exact sequence \Cref{eq:basic-exact-sequence} and the parameter $\FI.$
%\end{definition}

One very useful feature of the spaces $\htvar_G(\FI)$ is the presence of a dilating $\GG_m$-action.
\begin{definition}\label{defn:scaling-torus}
    Let $\SS=\GG_m$ be the 1-dimensional torus which acts on $\rT^*\AA^n\simeq \AA^{2n}$ by scaling each of the coordinates with weight $1.$
\end{definition}
\begin{lemma}
    The action of $\SS$ on $\rT^*\AA^n$ descends to an action on $\htvar_G(\FI)$ which scales the symplectic form with weight 2.
\end{lemma}
%\begin{proof}
%    The actions of $G$ and $\SS$ on $\rT^*\AA^n$ commute, and the unstable locus in $\rT^*(\AA^n/G)$ is invariant under $\SS.$ The 
%\end{proof}

\subsection{Toric holomorphic Lagrangians}
We now discuss some interesting Lagrangian subspaces of $\htvar.$ As usual, our discussion begins with the basic case of $\htvar=\rT^*\CC^n.$
%\begin{definition}
%    A \bit{sign vector} (of length $n$, although this will always be suppressed in our notation) is a length $n$ word in the alphabet $\{+,-\}.$ The set $\{\pm\}^n$ of sign vectors is in bijection with the set $2^{[n]}$ of subsets of $[n]=\{1,\ldots,n\},$ where $+$ (resp. $-$) indicates the elements that are included in (resp. excluded from) the subset. Using this bijection, we will also refer to subsets of $[n]$ as sign vectors.
%\end{definition}
%
%The basic role played by sign vectors is to label toric strata in $\CC^n,$ and hence some distinguished Lagrangian subvarieties of $\rT^*\CC^n.$
\begin{definition}\label{defn:xsgn}
We define the following collection of conormal Lagrangians in our hypertyric spaces, indexed by sign vectors $\sgn.$
\begin{itemize}
    \item Let $\sgn\in \sgvect$ be a sign vector.
%    \begin{enumerate}
%        \item We write $\cS_\alpha\subset \CC^n$ for the toric subvariety defined by
%        \[
%        \cS_\alpha:=\{(z_1,\ldots,z_n)\in \CC^n\mid z_i=0\text{ for }i\notin I\}.
%        \]
%        \item We write $\LL_I\subset \rT^*\CC^n$ for the conormal $\LL_I:=\rT^*_{\cS_I}\CC^n$ to the subvariety $\cS_I.$
%        \item %The subset $\cS_I\subset \CC^n$ is $G$-invariant, and 
%        We write $\ocS_I\subset \CC^n/G$ for the image of $\cS_I$ in the quotient stack, and 
%        \[\oLL_I:=\rT^*_{\ocS_I}(\CC^n/G)\subset \rT^*(\CC^n/G)\] for its conormal.
%    \end{enumerate}
%We write $X_G^\sgn\subset \CC^n/G$ for the image under the quotient map $\CC^n\to \CC^n/G$ of the intersection of coordinate hyperplanes
We write
\[
X^\sgn:=\{(z_1,\ldots,z_n)\in \AA^n\mid z_i=0\text{ for }i\notin \sgn\}\subset \AA^n
\]
for the subspace of $\AA^n$ given by the intersection of coordinate hyperplanes.
We denote its image in $\AA^n/G$ by $X_G^\sgn,$ and we write
$
\LL_G^\sgn:=\Conorm_{X_G^\sgn}(\AA^n/G)
$
for its conormal.
    \item For $\sgnsub\subset \sgvect$ a collection of sign vectors, we write
    $X_G^\sgnsub:=\bigsqcup_{\sgn\in\sgnsub}X_G^\sgn$ and
    $\LL_G^\sgnsub:=\bigcup_{\sgn\in \sgnsub}\LL_G^\sgn.$
    \item We write $X_G:= \bigsqcup_{\sgn\in\sgvect}X_G^\sgn$ and
    $\LL_G:=\bigcup_{\sgn\in\sgvect}\LL_G^\sgn.$
    \end{itemize}
    As $X_G^\sgn\hookrightarrow \AA^n/G$ is a closed immersion, we may understand $\LL_G^\sgn$ as a subspace of $\rT^*(\AA^n/G).$ Similarly, we treat the singular Lagrangians $\LL_G^\sgnsub$ and $\LL_G$ as subspaces of $\rT^*(\AA^n/G)$ as well.
\end{definition}

We will be interested in descending these Lagrangians to the hypertoric varieties $\htvar_G(\FI),$ for generic $\FI\in\fgv_\ZZ.$
\begin{definition}
    We write
    \[
    \LL_G(\FI,0):=\LL_G\cap (\rT^*(\AA^n/G))^{\FI-ss}\subset \htvar_G(\FI)
    \]
    for the locus of $\FI$-semistable points in the Lagrangian $\LL_G.$
\end{definition}

\begin{proposition}
    The locus of $\FI$-unstable points
    \[
    \LL_G^{\FI\unstab}:=\LL_G\setminus \LL_G(\FI,0)\hookrightarrow \LL_G
    \]
    in the Lagrangian $\LL_G$ is given by the union $\LL_G^{\FI\unstab}=\bigcup_{\sgn\in \FI\unstab}\LL_G^\sgn,$ where the set of $\FI$-unstable sign vectors $\sgn$ 
    is defined in \Cref{defn:unstable-unbounded}.
\end{proposition}
\begin{proof}
    This follows from the stability criterion \cite{Hoskins-stratifications}*{Proposition 2.7}.
\end{proof}

The Lagrangian $\LL_G(\FI,0)\subset \htvar_G(\FI)$ admits a simple description. Observe that when $G=\{e\}$ is trivial, the Lagrangian $\LL=\LL_G\subset \rT^*\AA^n$ is a normal crossings variety, and when $\kk=\CC,$ the chambers in the hyperplane arrangement $\cH$ (all orthants) are actually the moment polytopes for the toric (actually, all $\CC^n$) components of $\LL.$ This fact persists after the GIT quotient:

\begin{theorem}[\cite{Biel-Dan}*{Theorem 6.5}]
    %The Lagrangian $\LL_G(\FI,0)$ is a normal crossings union of toric varieties, meeting according to the combinatorics of the chambers in $\cH_G(\FI).$
    The component $(\LL_G^\sgn)^{\FI-ss}$ of $\LL_G(\FI,0)$ is a toric variety corresponding to the polytope $\cH_G^\sgn,$ and these components meet according to the incidences of chambers in $\cH_G.$ When $\kk=\CC,$ the chambers $\cH_G^\sgn$ are literally the moment polytopes for the restriction to $\LL_G(\FI,0)$ of the real moment map $\mu_{G^c}.$
\end{theorem}

In order to discuss (2-)categories $\cO,$ we will need one further choice of parameter: the attraction parameter $\mass\in \frf_\ZZ,$ 
%If $\mass$ is integral, it can be understood as a
which we understand as a
cocharacter $\mass:\GG_m\to F$ of the torus $F$ which acts on $\htvar_G(\FI),$ determining a Hamiltonian action of a 1-dimensional torus $\GG_m$ on $\htvar_G(\FI)$.
%in the non-integral case, it still specifies a complex vector field $V_\mass$ on $\htvar_G(\FI).$ 

%As we will explain in the next section, the real part of $\htvar_G(\FI)$ will be a Liouville vector field, and as such we will be interested in its stable locus or {\em skeleton}, i.e., the set of points in $\htvar_G(\FI)$ which do not flow off to infinity under this flow.
%\begin{definition}
%    A sign vector $I$ is \bit{bounded} if the restriction of $\mass$ to the chamber $C_I$ is bounded %above.
%\end{definition}
\begin{definition}\label{defn:bounded-attracting}
    A point $x\in\rT^*(\AA^n/G)$ is \bit{$\mass$-bounded} if %$\lim_{\lambda\to\infty}m(\lambda)\cdot x$ exists. {\color{red} (Fix this.)}
    there exists a $\GG_m$-equivariant map $\AA^1_\infty\to \htvar_G(\FI)$
    whose image contains $x,$
    where we write $\AA^1_\infty$ for the weight-$(-1)$ representation of $\GG_m.$
\end{definition}
When $\kk=\CC,$ so the variety $\htvar_G(\FI)$ is a complex manifold and $\GG_m=\CC^\times$, then \Cref{defn:bounded-attracting} is just the condition that the limit $\lim_{\lambda\to \infty}m(\lambda)\cdot x$ exists. In general, we refer to \cite{Drinfeld-Gaitsgory-hyperbolic}*{\S 1} for more information about the geometry of $\GG_m$-actions on algebraic spaces.

\begin{lemma}\label{lem:stable-set}
    The $\mass$-bounded locus $\LL_G^{\mass\bdd}$ in the Lagrangian $\LL_G$ is the union of components $\LL_G^\sgn$ corresponding to sign vectors $\sgn$ which are $\mass$-bounded in the sense of \Cref{defn:unstable-unbounded}.
\end{lemma}

\begin{definition}\label{defn:cato-skel}
    Fix parameters $\FI\in \fgv_\ZZ, \mass\in \frf_\ZZ.$ Then we write $\LL_G(\FI,\mass):=\LL_G(\FI,0)^{\mass\bdd}\subset \htvar_G(\FI)$ for the 
    subset of $\mass$-bounded points in $\LL_G(\FI,0),$ and we call it the
    \bit{category $\cO$ skeleton} of $\htvar_G(\FI).$
\end{definition}

\begin{remark}
    The space of parameters $(\FI,\mass)$ is divided into chambers by a finite number of rational walls, and none of the quantities we study in this paper changes as $(\FI,\mass)$ moves within a fixed chamber. This suggests that it is reasonable to take $(\FI,\mass)\in\fgv_\RR\times\frf_\RR$ to be continuous and real-valued, and indeed it is possible to make sense of the above constructions for such parameters: $\FI$ still defines a closed subset of unstable points, and $\mass$ defines a Hamiltonian vector field (which may not integrate to a $\CC^\times$-action) which can still be used to define the bounded points. Nevertheless, for simplicity, throughout this paper we restrict our parameters to have integral values.
\end{remark}
\section{Hypertoric categories $\cO$}
%{\color{red} (Some explanation about BGG category $\cO$ and generalization in \cite{BLPW16}, building on \cite{BLPW12}, goes here.)}
Throughout this section, we fix a generic choice of parameters $(\FI,\mass)\in \fg^\vee_\ZZ\times \frf_\ZZ,$ determining a Lagrangian $\LL(\FI,\mass)\subset \rT^*(\CC^n/G).$ Hypertoric category $\cO,$ first studied in \cite{BLPW10,BLPW12}, should be a category of regular DQ-modules or microlocal perverse sheaves on this Lagrangian. We will refer to these two flavors of category $\cO$ as the {\em de Rham} and {\em Betti} categories $\cO,$ respectively. Here we will give a brief review of these categories. 
%recalling the main results we need. 
%We refer to \cite{BLPW12} for further information about the de Rham category $\cO$;
%a more detailed study of the Betti category $\cO$ will appear in future work of the authors with Laurent C\^ot\'e.
%a more thorough account of these categories will appear in the companion paper \cite{CGH}.
The proofs of the statements about de Rham category $\cO$ can be found in \cites{BLPW10,BLPW12},
and those about Betti category $\cO$ will appear in \cite{CGH}.

\subsection{de Rham category $\cO$}\label{subsec:derham-cato}
Let $\cE$ be a deformation quantization of the variety $\htvar_G(\FI),$ 
which may be produced as in \cite{Kashiwara-Rouquier} by quantum Hamiltonian reduction of the sheaf of microdifferential operators on $\rT^*\CC^n,$ followed by restriction to $\htvar_G(\FI).$
We write $\cW:=\cE[\hbar^{-\frac{1}{2}}],$ so that $\cW$ is a sheaf of $\CC((\hbar^{\frac{1}{2}}))$-algebras on $\htvar_G(\FI),$ which is moreover 
equivariant for the action of the torus $\SS$ defined in \Cref{defn:scaling-torus},
%(where the $\CC^\times$ action is inherited from the fiber scaling action on $\rT^*\CC^n$), 
with $|\hbar|=2.$ 
%To each component $\LL_G^\sgn$ of the Lagrangian $\LL_G(\FI,\mass),$ we can associate a simple $\CC^\times$-equivariant $W$-module, which we denote by $L_\sgn.$
\begin{definition}
    An $\cE$-module $\cF$ is \bit{coherent} if $\cF/\hbar$ is a coherent sheaf on $\htvar_G(\FI).$
    A $\cW$-module is \bit{good} if it admits a coherent $\cE$-lattice. We write $\Good_{\cW}^{\SS}$ for the derived category of $\SS$-equivariant good $\cW$-modules.
\end{definition}

We refer to \cite{Kashiwara-Rouquier} for the theory of $\SS$-equivariant good $\cW$-modules.
\begin{definition}[\cite{BLPW12}*{Definition 6.2}]\label{defn:derham-cato}
    The \bit{de Rham category $\cO$} is the full subcategory of $\Good_{\cW}^{\SS}$ on those $\cF$ such that
    \begin{itemize}
        \item the support of $\cF$ is contained (set-theoretically) in $\LL_G(\FI,\mass),$ and
        \item $\cF$ admits an $\cE$-lattice that is preserved by the action of $m$.
    \end{itemize}
    We denote this category by $\cO^{\dR}(\htvar_G(\FI),\LL_G(\FI,\mass)).$
\end{definition}

This category is studied extensively in \cite{BLPW12}, where it is shown to be equivalent to the combinatorially-defined category introduced in \cite{BLPW10}. An alternative approach to studying this category is given in \cite{Web-gencat-O}*{\S 3}. We will summarize some facts about this category.

\begin{theorem}[\cite{BLPW12}]
$\cO^{\dR}(\htvar_G(\FI), \LL_G(\FI,\mass))$ is a highest weight category, with simple objects $L^\sgn$ corresponding to components of the skeleton $\LL_G(\FI,\mass)$. Each $L^\sgn$ has a projective cover $P^\sgn.$
\end{theorem}
%The de Rham category $\cO$ can thus be generated by its simple objects. 
%whose endomorphism algebras we now recall.
\begin{corollary}
    $\cO^{\dR}(\htvar_G(\FI),\LL_G(\FI,\mass))$ is generated by the simple objects $L^\sgn.$
\end{corollary}
%Moreover, the endomorphism algebra of simples in 
The category $\cO^{\dR}(\htvar_G(\FI),\LL_G(\FI,\mass))$ 
may therefore be described explicitly by calculating the endomorphism algebra of its simple objects.
The main technical tool used in performing this calculation is the \bit{categorical Kirwan surjectivity} theorem \cite{BPW}*{Theorem 5.31}, which says that the functor induced on deformation quantizations by pullback along the open inclusion $\htvar_G(\FI)\hookrightarrow T^*(\CC^n/G)$ admits a fully faithful left adjoint.
As a result, we obtain the following description of the de Rham category:
%admits an economical description., using the following analogue of \Cref{prop:muperv-presentation}:
\begin{proposition}[\cite{Web-gencat-O}*{Proposition 2.6}]
    There is an embedding 
    %{\color{red} (Actually, this should just be an embedding, right?)}
    \begin{equation}\label{eq:derham-quotient}
    \cO^{\dR}(\htvar(\FI),\LL_G(\FI,\mass))
    \hookrightarrow
    %is equivalent to the quotient category 
    \Dmod_{\LL_G^{\mass\bdd}}(\CC^n/G)/\Dmod_{\LL_G^{\mass\bdd,\FI\unstab}}(\CC^n/G)
    \end{equation}
    of the de Rham category $\cO$ into the quotient of the category of D-modules with $\mass$-bounded microsupport by those with purely unstable microsupport.
\end{proposition}

%Recall that for $\sgn\in\sgvect,$ we write $X^\sgn_G$ for the image in $\CC^n/G$ of the intersection of coordinate hyperplanes specified by $\sgn.$ Similarly, for $\sgnsub\subset \sgvect,$ we will write $X^\sgnsub_G:=\bigsqcup_{\sgn\in \sgnsub}X^\sgn_G.$

In the quotient presentation \Cref{eq:derham-quotient}, the simples in $\cO^{\dR}(\htvar(\FI),\LL_G(\FI,\mass))$ correspond to the pushforwards of functions D-modules on the subvarieties $X^\sgn_G\hookrightarrow \CC^n/G$
for the $\mass$-bounded sign vectors $\sgn.$
Before quotienting out the unstably-supported D-modules, the endomorphism algebra of these objects is given by the convolution algebra
$C^\BM_*(X^{\mass\bdd}\times_{\CC^n/G}X^{\mass\bdd})$
of Borel-Moore chains on fiber products of the varieties $X_G^\sgn$.
%, with
%algebra structure induced by convolution.
%As a result, we obtain the following description of the endomorphism algebra of simples in the de Rham category $\cO.$
The endomorphism algebra of simples in category $\cO$ can therefore be described as a quotient of this algebra:
\begin{theorem}[\cite{Web-gencat-O}*{Corollary 3.3}]\label{thm:derham-o-quotient}
%    Let $A=C^\BM_*(X^{m\bdd}\times_{\CC^n/G}X^{m\bdd})$ and $I\subset A$ the ideal generated by the subalgebra $C^\BM_*(X^{m\bdd,\FI\unstab}\times_{\CC^n/G} X^{m\bdd,\FI\unstab}).$ Then there is an equivalence
%    \begin{equation}
%    \cO^{\dR}(\htvar(\FI),\LL_G(\FI,\mass))\simeq A/I.
%    \end{equation}
Let $A^{\dR}$ be the quotient of the algebra
$C^\BM_*(X^{m\bdd}\times_{\CC^n/G}X^{m\bdd})$ by the ideal generated by
$C^\BM_*(X^{m\bdd,\FI\unstab}\times_{\CC^n/G} X^{m\bdd,\FI\unstab}).$
Then there is an equivalence
$\cO^{\dR}(\htvar(\FI),\LL_G(\FI,\mass))\simeq\Perf_{A^{\dR}}$ between the de Rham category $\cO$ and the category of perfect modules over $A^{\dR},$ induced by an equivalence $A^{\dR}\simeq\End_{\cO^{\dR}}(\bigoplus L^\sgn)$ between $A^{\dR}$ and the endomorphism algebra of the simple objects in $\cO^{\dR}.$
\end{theorem}

\subsection{Betti category $\cO$ and Riemann-Hilbert}\label{sec:betti-cato}
Following \cite{CGH}, we will understand Betti category $\cO$ as a category of microlocal perverse sheaves: see \cite{NS20} for the definition of microlocal sheaves on a general (stably polarized) Weinstein manifold (extending the original work of \cite{Kashiwara-Schapira} on cotangent bundles) and \cite{perverse-microsheaves} for the definition of the perverse t-structure on a stably holomorphically polarized holomorphic Lagrangian. In this subsection, we give an abbreviated presentation of the calculation of Betti category $\cO$ from \cite{CGH}, and we refer there for more details.
%the stack of microlocal sheaves was originally defined in \cite{Kashiwara-Schapira} for cotangent bundles
%, as defined in \cite{NS20} (which extends the original theory of \cite{Kashiwara-Schapira} from cotangent bundles to the setting of stably polarized Weinstein manifolds).
\begin{notation}
    In this section, we take our hypertoric variety $\htvar_G(\FI)$ to be defined over $\CC,$ and we write $\kk$ for the coefficients in which our perverse sheaves are valued.
\end{notation}
\begin{warn}
%Microlocal sheaves on a holomorphic Lagrangian (with a holomorphic stable polarization) admit a t-structure whose heart is the microlocal perverse sheaves, defined in \cite{perverse-microsheaves}.
    As noted in \Cref{sec:notation}, our conventions for the category $(\mu)\Perv_\Lambda$ of (microlocal) perverse sheaves along a Lagrangian $\Lambda$ differ in two ways from the standard ones: 
    \begin{enumerate}
        \item We do not require objects to have finite-dimensional microstalks but instead write $(\mu)\Perv_\Lambda$ for the category of compact objects in the presentable category of perverse sheaves with possibly infinite-dimensional microstalks. For instance, $\Perv_{\CC^\times}(\CC^\times)\simeq \Loc(\CC^\times)$ denotes not the category 
        $\Mod^{\text{perf}/\kk}_{\kk[x^\pm]}$ of
        finite-dimensional
        %(i.e., perfect over $\kk$)
        %\footnote{``Finite-dimensional'' always means ``perfect with respect to the coefficient ring $\kk$.''}
        $\kk[x^\pm]$-modules but rather the category $\Perf_{\kk[x^\pm]}$ of perfect $\kk[x^\pm]$-modules ---
        including the free module $\kk[x^\pm],$ which corresponds to the universal local system on $\CC^\times$ given by the pushforward of the constant sheaf along the universal covering map $\CC\to\CC^\times.$
        \item We denote by $(\mu)\Perv_\Lambda$ not the abelian category of (microlocal) perverse sheaves but rather its derived category. 
        %This derived category has a t-structure, and to denote the abelian category which is its heart we write $(\mu)\Perv_\Lambda^\heart.$
    \end{enumerate}
\end{warn}

\begin{definition}[\cite{CGH}]
%\footnote{Note that \cite{CGH} understands Betti category $\cO$ as an abelian category, whereas our convention is to pass to its derived category.}]
\label{defn:betti-cato}
    The \bit{Betti category $\cO$} is the category 
    \[
    \cO_{\Bet}(\htvar_G(\FI),\LL_G(\FI,\mass)):=\muPerv_{\LL_G(\FI,\mass)}(\LL_G(\FI,\mass))
    \]
    of microlocal perverse sheaves on the Lagrangian $\LL_G(\FI,\mass),$ with stable polarization induced from the fiber polarization of $\rT^*(\CC^n/G)$ by \cite{NS-higgs}*{Corollary 4.5}.
\end{definition}

%As in the de Rham situation, we can describe this category via the open inclusion $\LL_G(\FI,\mass)\hookrightarrow \LL_G^{\mass\bdd},$ with closed complement $\LL_G^{\mass\bdd,\FI\unstab}$: namely, the category of microlocal sheaves on $\LL_G(\FI,\mass)$ admits a presentation as the quotient category
%\begin{equation}\label{eq:betti-mush-as-quotient}
%    \mu\Sh_{\LL_G(\FI,\mass)}(\LL_G(\FI,\mass)) \simeq
%    \mu\Sh_{\LL_G^{\mass\bdd}}(\LL_G^{\mass\bdd}) / \mu\Sh_{\LL_G^{\mass\bdd,\FI\unstab}}(\LL_G^{\mass\bdd,\FI\unstab}).
%\end{equation}
%The categories on the right-hand-side of \Cref{eq:betti-mush-as-quotient} can both be understood as honest constructible-sheaf categories on the stack $\CC^n/G,$ so that we can rewrite \Cref{eq:betti-mush-as-quotient} as
%\begin{equation}\label{eq:bcato-subquotient}
%    \mu\Sh_{\LL_G(\FI,\mass)}(\LL_G(\FI,\mass))\simeq 
%    \Sh_{\LL_G^{\mass\bdd}}(\CC^n/G)/
%    \Sh_{\LL_G^{\mass\bdd,\FI\unstab}}(\CC^n/G).
%\end{equation}
%
%\subsection{Algebraic description}
Recall that there is a locally closed embedding $\LL_G(\FI,\mass)\hookrightarrow \LL_G=\LL/G,$
where $\LL\subset \rT^*\CC^n$ is the union of conormals to intersections of coordinate hyperplanes. This factors as a composition
\[
\LL_G(\FI,\mass) \hookrightarrow \LL_G(\FI,0)\hookrightarrow \LL_G,
\]
of a closed embedding followed by an open embedding. Accordingly, Betti category $\cO$ may be computed in a three-step process, by first understanding the category $\muPerv_{\LL_G}(\LL_G) = \Sh^G_\LL(\CC^n),$ microlocally restricting to the open set of $\FI$-semistable points, and then quotienting by linking disks to $\mass$-unbounded components of $\LL(\FI,0).$

\begin{definition}
    Consider the quiver
    $Q = \left(
    \begin{tikzcd}
        \bullet \ar[r, shift left, "u"] & \ar[l, shift left, "v"] \bullet
    \end{tikzcd}
    \right)
    $
    and write $\kk Q$ for its path algebra. Let $m\in Z(\kk Q)$ denote the central element $m:=1-(uv+vu).$ We define algebras
    \[
    A_1:=\kk Q[m^\pm], \qquad A_n:=(A_1)^{\otimes n}\simeq (\kk Q)^{\otimes n}[m_1^\pm,\ldots,m_n^\pm].
    \]
\end{definition}

\begin{theorem}[\cite{GGM}]\label{cor:perv-cn-quiver}
    There is an equivalence of categories $\Perv_\LL(\CC^n)\simeq \Perf_{A_n}$ between toric-stratified perverse sheaves on $\CC^n$ and perfect modules over the algebra $A_n,$ intertwining the action of $\Loc(\CC^\times)^n$ with the $\CC[m_1^\pm,\ldots,m_n^\pm]$-linear structure induced by the central map $\CC[m_1^\pm,\ldots,m_n^\pm]\to Z(A_n).$
\end{theorem}
\begin{definition}
    Let $D^\vee$ be the $n$-dimensional torus whose algebra of functions $\CC[D^\vee]$ is the ring $\CC[m_1^\pm,\ldots,m_n^\pm],$ and let $G\subset D$ be a subtorus, so that $\CC[G^\vee]$ embeds as a subalgebra of $\CC[D^\vee].$
    We write
    \[
    A_G:=A_n\otimes_{\CC[G^\vee]}\CC
    \]
    for the algebra obtained from $A_n$ by trivializing the $G$-monodromies.
\end{definition}
\begin{corollary}
    There is an equivalence of categories $\Perv_{\LL_G}(\CC^n/G)\simeq \Perf_{A_G}.$
\end{corollary}
%\begin{lemma}
%    The category $\Sh_\LL(\CC^n)$ is equivalent to the derived category of $\Perv_\LL(\CC^n).$ 
%{\color{red}(Change this to a remark and possibly move it.)}
%\end{lemma}
%\begin{proof}
%    It is sufficient to check this in the case $n=1.$
%    The category $\Sh_\LL(\CC)$ 
%is equivalent to the category of representations of the exit-path category of the stratified space $\CC = \CC^\times\sqcup \{0\}.$ Such a representation is the data of a pair of vector spaces $V, \Psi$ together with a restriction map $r:V\to \Psi,$ an invertible monodromy map $m:\Psi\xrightarrow{\sim}\Psi,$ and a homotopy $\eta:mr\sim m.$
%
%    In this presentation, $\Phi$ is recovered as $\Cone(r),$ the map $v:\Psi\to \Phi$ is the canonical map, and the map $u$ (together with the relations that $uv-1$ and $vu-1$ are the monodromy automorphisms) come from the homotopy $\eta$: see for instance \cite{CDW}*{Remark 5.1.1} for a detailed account. This gives the equivalence between representations of the exit-path category and the linear-algebraic data \Cref{eq:basic-perverse-diagram}.
%\end{proof}
\begin{notation}
    We may denote an $A_n$-module as an 
    $n$-hypercube of vector spaces $\{V_\sgn\mid \sgn\in \sgvect\}$ equipped with two families of linear maps 
    \begin{equation}\label{eq:pervcn-linalg-diagram}
        u_i:V_\sgn \rightleftarrows V_{\sgn\cup\{i\}}: v_i
    \end{equation}
    such that all the operators $u_i$ and $v_j$ commute, and such that $1-u_iv_i$ and $1-v_iu_i$ are invertible. This data will specify an $A_G$-module under the additional condition that the appropriate products of $(1-u_iv_i)$ and $(1-v_iu_i),$ corresponding to monodromies in $G\subset D,$ are trivial.
\end{notation}
\begin{example}
    In case $n=1,$ \Cref{cor:perv-cn-quiver} recover the classical description of perverse sheaves on $(\CC,0)$ in terms of the data
    \[
    \begin{tikzcd}
        \Phi\ar[r, shift left, "var"] & \ar[l, shift left, "can"] \Psi
    \end{tikzcd}
    \]
    of their vanishing and nearby cycles, with canonical and variation maps between them. (The general case may be recovered from this one by taking a tensor product.)
\end{example}

\begin{definition}\label{defn:betti-1cat-projectives}
    For $\sgn\in\sgvect,$ we write $P^\sgn_G$ for the object of $\Perv_{\LL_G}(\CC^n/G)$ which corepresents the functor
    $\Perv_{\LL_G}(\CC^n/G)\to \Mod_\kk$ taking a diagram \Cref{eq:pervcn-linalg-diagram} to the vector space $V_\sgn.$ We simply write $P^\sgn:=P^\sgn_{e}$ when $G$ is trivial.
\end{definition}
\begin{remark}
When $\alpha$ corresponds to the zero-section $\CC^n\subset \LL,$ the object $P^\sgn$ can be constructed geometrically as the !-extension of the universal local system on $(\CC^\times)^n$; the case of general $\alpha$ follows by Fourier transforms. Algebraically, $P^\sgn$ is the projective $A_n$-module $A_n e_\alpha,$ where $e_\alpha$ is the vertex idempotent corresponding to $\alpha$.
\end{remark}
    The object $P=\bigoplus_{\sgn\in \sgvect}P^\sgn$ is a projective generator for the category $\Perv_\LL(\CC^n).$
%\begin{proof}
%    Projectivity is clear from the definition of $P^\sgn,$ and generation follows from the fact that if an object $(V_\bullet, u,v)$ of $\Perv_\LL(\CC^n)$ has $V_\sgn\simeq 0$ for all $\sgn,$ then it is the zero object, 
%\end{proof}
\begin{example}\label{ex:projectives-formula}
    Let $n=1.$ Then %in terms of \Cref{eq:basic-perverse-diagram}, 
    as $A_1$-modules,
    the objects $P^\Phi,P^\Psi$ 
    are given by the respective diagrams
    \[
    \begin{tikzcd}
        \kk[m^\pm] \ar[r,"\sim", shift left]&
        \kk[m^\pm], \ar[l, "1-m", shift left]&
        \kk[m^\pm] \ar[r,"1-m", shift left]&
        \kk[m^\pm]. \ar[l, "\sim", shift left]
    \end{tikzcd}
    \]
    As perverse sheaves, one of these is given by the proper pushforward, along the embedding $\CC^\times\to \CC,$ of the universal local system (with stalk $\kk[m^\pm]$ and monodromy $m$) on $\CC^\times,$ and the other is its Fourier transform.
\end{example}

\begin{remark}\label{rem:dim1-microrestriction}
    As \Cref{ex:projectives-formula} shows, restriction of $\Perv_\LL(\CC)$ to the open subset $\CC^\times\subset \CC\subset \LL$, or microlocal restriction to $T^*_0\CC\setminus\{0\},$ is right adjoint to the inclusion in $\Perv_\LL(\CC)$ of the full subcategory generated by $P^\Psi$ or $P^\Phi.$ Below, we will see that a similar formula applies for more general microlocal restrictions.
\end{remark}

Observe that the unit in $A_G$ is a sum of $2^n$ orthogonal idempotents, corresponding to the vertices in the quiver $Q$ and therefore in bijection with the components $\LL_G^\sgn$ of the Lagrangian $\LL_G.$
\begin{definition}
    Let $e_\cF\in A_G$ be the sum of the idempotents corresponding to vertices $\sgn$ which are $\FI$-semistable. (The $\cF$ is for ``feasible,'' the linear-progreamming condition satisfied by these.) We write $A_G(\FI,0)$ for the algebra 
    \[A_G(\FI,0):=e_\cF A_G e_\cF.\]
\end{definition}

\begin{theorem}[\cite{CGH}*{Corollary 6.1.8}]\label{thm:cgh-microrestriction}
    There is an equivalence of categories
    \[
    \mu\Perv_{\LL_G}(\LL_G(\FI,0))\simeq \Perf_{A_G(\FI,0)}.
    \]
\end{theorem}
\begin{remark}
    We may equivalently describe $\Perf_{A_G(\FI,0)}$ as the full subcategory of $\Perf_{A_G}$ generated by the projective object $P_G^\cF:=\bigoplus_{\sgn\in \cF}P^\sgn.$
\end{remark}
\begin{corollary}\label{cor:1cat-kirwansurj}
    The restriction map
    \[
    \Perv_{\LL_G}(\CC^n/G) = \muPerv_{\LL_G}(\LL_G) \to \muPerv_{\LL_G}(\LL_G(\FI,0))
    \]
    has a fully faithful left adjoint whose essential image is the subcategory generated by the projective objects $P^\sgn$ for $\sgn$ $\FI$-semistable.
\end{corollary}

We recall briefly the proof of \Cref{thm:cgh-microrestriction}, as we will repeat it one categorical level higher in the following section. First, consider the moment map $\LL_G(\FI,0)\to \ffv_\RR(\FI).$ The latter space is the home of the hyperplane arrangement $\cH_G(\FI),$ whose chambers are the moment map images of the toric components of $\LL_G(\FI,0).$ This space has a natural open cover, indexed by faces of the hyperplane arrangement $\cH_G(\FI)$: to a face $\Delta,$ we associate the union of all faces whose closures contain $\Delta.$ Lifting this open cover along the moment map, we obtain a cover of $\LL_G(\FI,0)$ by Lagrangians $\LL_G^\Delta,$ indexed by the face poset $\bP(\FI)$ of $\cH_G(\FI).$ By descent, we have
\begin{equation}\label{eq:muperv-limit}
\muPerv_{\LL_G}(\LL_G(\FI,0))\simeq \varprojlim_{\Delta\in \bP(\FI)}\muPerv_{\LL_G}(\LL_G^\Delta).
\end{equation}

\Cref{thm:cgh-microrestriction} is now a combination of two facts.
    First, we understand the categories $\muPerv_{\LL_G}(\LL_G^\Delta)$: as described in \Cref{rem:dim1-microrestriction}, in dimension 1 we understand the microlocal restrictions of $\Perv_\LL(\CC)$ to the two natural open subsets $\CC^\times\times \CC, \CC\times \CC^\times \subset \CC\times \CC = T^*\CC,$ and each $\LL_G^\Delta\subset \LL_G$ can be obtained by restricting to a product of such open sets. By considering these product open sets, we obtain the following calculation:
\begin{lemma}[\cite{CGH}*{Corollary 5.3.5}]
    There are equivalences of categories
    \[
        \mu\Perv_{\LL_G}(\LL_G^\Delta) \simeq e_\Delta A_G(\FI,0)e_\Delta,
    \]
    with restriction functors induced by the bimodules $e_\Delta A_G(\FI,0)e_{\Delta'},$ where we write $e_\Delta\in A_G(\FI,0)$ for the sum of the idempotents corresponding to chambers whose closure contains $\Delta.$
\end{lemma}

Second, we compute the limit \Cref{eq:muperv-limit}:
\begin{lemma}[\cite{GMW}*{Lemma C.2}]
    The natural functor
    \[
    \Mod_{A_G(\FI,0)}\to \varprojlim_{\Delta\in \bP(\FI)}\Mod_{e_{\Delta}A_G(\FI,0)e_{\Delta}}
    \]
    is an equivalence.
\end{lemma}
This lemma completes the proof of \Cref{thm:cgh-microrestriction}. By applying stop removal, we can complete the description of Betti category $\cO.$

\begin{definition}
    Let $e_{\cF\cap\cB^c}\in A_G(\FI,0)$ be the sum of idempotents corresponding to chambers in $\cH_G(\FI)$ which are $\mass$-bounded, and write $\cI_{\cF\cap \cB^c}\subset A_G(\FI,0)$ for the ideal $\cI_{\cF\cap \cB^c}:=A_G(\FI,0)e_{\cF\cap \cB^c}A_G(\FI,0)$.
    We write $A_G(\FI,\mass)$ for the algebra
    \begin{equation}\label{eq:definition-agtm-quotient}
    A_G(\FI,\mass):=A_G(\FI,0)/\cI_{\cF\cap \cB^c}.
    \end{equation}
\end{definition}

\begin{corollary}\label{cor:betti-cato-algebraic}
    There is an equivalence of categories
    \begin{equation}\label{eq:betti-category-o-is-agtm}
    \mu\Perv_{\LL_G(\FI,\mass)}(\LL_G(\FI,\mass))\simeq \Mod_{A_G(\FI,\mass)}.
    \end{equation}
\end{corollary}
\begin{remark}
    In fact, in \cite{CGH} is proved the significantly stronger statement that the right-hand side of \Cref{eq:betti-category-o-is-agtm} is equivalent to the whole category $\muSh_{\LL_G(\FI,\mass)}(\LL_G(\FI,\mass)).$ This is not automatic from the above calculations, since a priori the quotient \Cref{eq:definition-agtm-quotient} may have to be taken in an appropriately derived sense. To rule this out, it is necessary to check that the ideal $\cI_{\cF\cap \cB^c}$ is {\em stratifying} in the sense of \cite{CPS}. However, we will not need this stronger statement here, as we are only interested in the heart of the perverse t-structure.
\end{remark}

We can give an explicit presentation of the algebra $A_G(\FI,\mass)$ as follows:

%\begin{definition}%\label{defn:quiver-algebra}
\begin{lemma}\label{lem:quiver-algebra}
    Let $Q_\FI$ be the quiver whose vertices are indexed by $\sgn\in \sgvect$ which are $\FI$-stable, and whose arrows come in pairs
    \begin{equation}\label{eq:vertex-pair}
     u_i:\sgn \rightleftarrows \sgn\cup\{i\}:v_i
    \end{equation}
for each pair of vertices of the form $(\sgn,\sgn\cup\{i\}).$

Then $A_G(\FI,\mass)$ is equivalent to the algebra $\kk Q\otimes\kk[\pi_1F]/I,$ where $I$ is the ideal generated by the following relations:
\begin{enumerate}
    \item $e_\sgn = 0$ for each $\sgn$ which is $\mass$-unbounded.
    %, where we write $e_\sgn$ for the idempotent (length-0 path) at $\sgn.$
    \item For each square in $Q$ of the form
    \[
    \begin{tikzcd}
        \sgn    \arrow[r, shift right] \arrow[d, shift right] & 
        \sgn'   \arrow[l, shift right] \arrow[d, shift right]\\
        \sgn''  \arrow[r, shift right] \arrow[u, shift right] & 
        \sgn''' \arrow[l, shift right] \arrow[u, shift right],
    \end{tikzcd}
    \]
    the two length-2 paths from a vertex to its diagonal opposite commute.
    \item For each pair of vertices as in \Cref{eq:vertex-pair}, we have $1-u_iv_i = \overline{m}_i e_{\alpha\cup\{i\}},$ $1-v_iu_i = \overline{m}_ie_{\alpha},$ where $\overline{m}_i$ is the image in $\kk[\pi_1F]$ of the $i$th generator in $\kk[\pi_1D]\simeq \kk[\ZZ^n]\simeq\kk[m_1^\pm,\ldots,m_n^\pm].$
\end{enumerate}
\end{lemma}

\begin{remark}
    The algebra of \Cref{lem:quiver-algebra} may be most efficiently described in the language of the hyperplane arrangement $\cH_G(\FI)$ of \Cref{defn:hyperplane-arrangement}. The vertices $\sgn$ which are $\FI$-stable correspond to chambers in $\cH_G(\FI),$ and pairs of $\FI$-stable vertices $\alpha,\alpha\cup\{i\}$ are separated by the $i$th hyperplane in the arrangement. Paths in the quiver $Q$ now correspond to paths between chambers in $\cH_G(\FI),$ and the relation (2) identifies the various shortest paths between any pair of chambers.
\end{remark}

\begin{example}\label{ex:betti-category-o}
    Let $n=2$ and $G=\CC^\times\hookrightarrow(\CC^\times)^2$ be the kernel of the multiplication map 
    \[
    \begin{tikzcd}
        (\CC^\times)^2\arrow[r,"m"]& \CC^\times.
    \end{tikzcd}
    \]
    Then $\Perv_{\LL_G}(\CC^2/G)$ is the category of diagrams
    \[
    \begin{tikzcd}
        V^{\{1\}} \arrow[d, "v_1", shift left] \arrow[r, "u_2", shift left] & 
        V^{\{1,2\}} \arrow[d, "v_1", shift left] \arrow[l, "v_2", shift left] \\
        V^\emptyset \arrow[u, "u_1", shift left] \arrow[r, "u_2", shift left] &
        V^{\{2\}}, \arrow [u, "u_1", shift left] \arrow[l, "v_2", shift left]
    \end{tikzcd}
    \]
    where each vector space $V^\sgn$ is equipped with an automorphism $M_\sgn,$ with the relations 
    \begin{align*}
        1-u_1v_1 &= M, & 1-u_2v_2 &= M^{-1},\\
        1-v_1u_1 &= M, & 1-v_2u_2 &= M^{-1}.
    \end{align*}
    For an appropriate choice of parameters $\FI,\mass,$ the component $\LL_G^{\{2\}}$ becomes $\FI$-unstable, while the smooth points of $\LL_G^{\emptyset}$ become $\mass$-unbounded. Now $\mu\Perv(\LL_G(\FI,\mass))$ is the quotient of $\Perv_{\LL_G}(\CC^2/G)$ by the simple object at $V^{\{2\}}$ and the projective object at $V^{\emptyset}.$ This has the effect of forgetting $V^{\{2\}}$ and setting $V^{\emptyset}$ to 0, so that $\mu\Perv(\LL_G(\FI,\mass))$ is the category of diagrams
    \[
    \begin{tikzcd}
        V^{\{1\}} \arrow[r, "u_2", shift left] \arrow[d, "v_1", shift left]& 
        V^{\{1,2\}}  \arrow[l, "v_2", shift left] \\
        0, \arrow [u, "u_1", shift left] &
        { }
    \end{tikzcd}
    \]
    with monodromies at $V^{\{1\}}$ and $V^{\{1,2\}}$ and the same relations as above (except those involving the bottom-right corner). Setting the bottom-left corner to 0, together with the two relations at the upper-left vertex, forces $v_2u_2=0,$ which also forces the final relation (invertibility of $u_2v_2$), and we conclude that $\mu\Perv(\LL_G(\FI,\mass))$ is the category of diagrams of vector spaces
    \[
    \begin{tikzcd}
        V^{\{1\}} \arrow[r, "u_2", shift left] & 
        V^{\{1,2\}} \arrow[l, "v_2", shift left]
    \end{tikzcd}
    \]
    with the single relation $v_2u_2=0.$
    This is the usual description of the BGG category $\cO$ for $SL(2).$
\end{example}

We may reframe \Cref{cor:betti-cato-algebraic} as giving a presentation of Betti category $\cO$ as a subquotient of the category $\Perv_{\LL_G}(\CC^n/G)$:
\begin{corollary}\label{cor:betti-cato-as-subquotient}
    There is an equivalence
    %{\color{red}(maybe edit in terms of feasible, unbounded)}
    \[
    \muPerv_{\LL_G(\FI,\mass)}(\LL_G(\FI,\mass))\simeq 
    \frac{\langle P_G^\sgn\mid \sgn\text{ $\FI$-semistable} \rangle}{\langle P_G^\beta \mid \beta \text{ $\FI$-semistable, $\mass$-unbounded}\rangle}.
    \]
\end{corollary}

Finally, we note that although in \Cref{subsec:derham-cato} we described $\cO^{\dR}$ in terms of the endomorphism algebra of simples, it may also be described in terms of the endomorphism algebra of projectives. (That the projectives may be described in terms of the simples is a strong finiteness property of category $\cO.$) The resulting algebra, described in \cite{BLPW10}*{\S 3}, is a quiver algebra of the same form as $A^\Bet.$
The following is a microlocal form of the Riemann-Hilbert correspondence.
\begin{proposition}[\cite{CGH}*{Corollary 1.2.2}]\label{prop:betti-is-derham-cato}
    Let $\kk=\CC.$ Then there is an equivalence
    \begin{equation}\label{eq:catsO-Riemann-Hilbert}
    \cO^\dR(\htvar_G(\FI),\LL_G(\FI,\mass))\simeq \cO^\Bet(\htvar_G(\FI),\LL_G(\FI,\mass))
    \end{equation}
    between the Betti and de Rham categories $\cO.$ 
\end{proposition}
%We refer to \cite{CGH} for more details and an explicit form of this equivalence.
%\begin{proof}[Proof (\Cref{thm:betti-cato-algebraic})]
   %A priori, the Betti category $\cO_\Bet(\htvar_G(\FI),\LL_G(\FI,\mass))$ is a subquotient of $\Perv_{\LL_G}(\CC^n/G)$: namely, we may first pass to the full subcategory 
  % $\Perv_{\LL_G^{\mass\bdd}}(\CC^n/G)\hookrightarrow\Perv_{\LL_G}(\CC^n/G)$ on the objects with vanishing microstalks at $\mass$-unbounded points of $\LL_G$; from \Cref{eq:bcato-subquotient}, we know that the Betti category $\cO$ is a quotient of this 
%   First, observe that $\Perv_{\LL_G^{\mass\bdd}}(\CC^n/G)$ admits an embedding 
%   \[
%   \Perv_{\LL_G^{\mass\bdd}}(\CC^n/G)\hookrightarrow\Perv_{\LL_G}(\CC^n/G)
%   \]
%   as the full subcategory on the objects with vanishing microstalks at $\mass$-unbounded points of $\LL_G$; in the description of \Cref{lem:perv-cng-quiver}, these microstalks are the vector spaces $V_\sgn$ for the $\mass$-unbounded $\sgn.$ From \Cref{eq:bcato-subquotient}, we have a presentation of Betti category $\cO$ as a quotient of this category by the subcategory $\Perv_{\LL_G^{\mass\bdd,\FI\unstab}}(\CC^n/G),$ which is generated by the simple objects at $\FI$-unstable components of $\LL_G$.
   %The quotient by these simple objects is a localization, with local objects given by the $\FI$-stable projectives.
 %  {\color{red} Finish this.}
%\end{proof}

\section{Microlocal perverse schobers}
Our A-side 2-category $\cO$ will be a category of microlocal perverse schobers on the skeleton $\Oskel,$ categorifying the presentation of $\cO^\Bet$ as a category of microlocal perverse sheaves. Unfortunately, microlocal perverse schobers do not yet have a definition. 
%To circumvent this problem, we imitate the method of \Cref{sec:betti-cato}, using the embedding (which we now treat as a definition) of our A-side 2-category in the 2-category of perverse schobers on $\CC^n/G.$
However, \Cref{cor:betti-cato-as-subquotient} tells us that microlocal perverse sheaves on $\Oskel$ can be defined as a subquotient of the category of perverse sheaves on $\CC^n/G,$ described in terms of the projective generators of that category. %supports in closed subsets of $\Oskel.$ 
We will therefore define microlocal perverse sheaves of categories on $\Oskel$ as the analogous subquotient in the 2-categorical setting. To check that this definition is reasonable, we will show that the Kirwan surjectivity result of \Cref{cor:1cat-kirwansurj} continues to hold in this setting.

\subsection{Spherical functors and perverse schobers}
We begin by recalling that 2-category from \cite{GMH}. We first need the notion of a spherical adjunction. 
\begin{definition}[\cite{Schanuel-Street}]
    The \bit{free adjunction} is the (discrete, i.e., $(2,2)$ rather than $(\infty,2)$) 2-category $\Adj$ generated by 1-morphisms
    \begin{equation}\label{eq:universal-adjunction}
     \begin{tikzcd}[column sep=2cm]
     \Phi
     \arrow[yshift=0.9ex]{r}{L}
     \arrow[leftarrow, yshift=-0.9ex]{r}[swap]{R}
     &
     \Psi
     \end{tikzcd}
    \end{equation}
    and the 2-morphisms $\id_\Phi\to RL, LR\to \id_\Psi$, with the relations
    that the compositions $R\to RLR\to R$ and $L\to LRL\to L$ are the identity maps $\id_R$ and $\id_L,$ respectively.
    %, witnessing the adjunction $L\adj R.$
    \end{definition}
    \begin{theorem}[\cites{RV-adj}]\label{thm:adj-corepresents-adjunctions}
        The 2-category $\Adj$ corepresents adjunctions in 2-categories.
    \end{theorem}
    \begin{definition}
    Let $\StAdj:=\Sigma^{(\infty,2)}\Adj$ be the $\kk$-linear stabilization of the 2-category $\Adj.$ The {\em universal (co)twist} is the 1-morphism
    \[
    T_\Phi:=\fib (\id_\Phi\to RL) \qquad (\text{resp., } T_\Psi:=\cofib(LR\to \id_\Psi))
    \]
    in $\StAdj$ given by the (co)fiber of the (co)unit of the universal adjunction \Cref{eq:universal-adjunction}.
    
    The  \bit{free spherical adjunction} is the stable 2-category 
    \[
    \Sph:=\StAdj[T_\Phi^{-1},T_\Psi^{-1}] 
    \]
    obtained from $\Adj$ by
    localizing at the 1-morphisms $T_\Phi,T_\Psi.$\footnote{This differs from the notation in \cite{GMH}, where $\Sph$ was synonymous with the 2-category 
    $\PS_\LL(\CC).$}

    Let $\cC\rightleftarrows \cD$ be an adjunction in stable categories, which (by \Cref{thm:adj-corepresents-adjunctions} and the stability of $\cC$) corresponds to a functor $F:\StAdj\to \Cat.$ If $F$ admits a factorization $\StAdj\to\Sph\to\Cat,$ then we say that the adjunction $\cC\rightleftarrows \cD$ is \bit{spherical}, and that each of the adjoint functors is a \bit{spherical functor}.
    %The \bit{free spherical $n$-cube} is the $n$-fold product $\Sph^n.$
 %   
%    The 2-category of \bit{spherical adjunctions} is the 2-category
%    \[
%        \Sph:= \Fun(\StAdj[T_\Phi^{-1}, T_\Psi^{-1}], \St)
%    \]
%    of adjunctions in stable categories whose twist and cotwist are invertible endofunctors.
%    The 2-category of \bit{spherical $n$-cubes} is the $n$-fold tensor product $\Sph_n:=(\Sph)^{\otimes n}.$
\end{definition}
As usual, let $\LL=\bigcup_\sgn \rT^*_{X^\sgn}\CC^n\subset T^*\CC^n$ be the union of conormals to intersections of coordinate hyperplanes.
\begin{definition}
The 2-category of \bit{perverse schobers on $\CC^n$}
with singular support in $\LL$
%the conormal to toric strata 
(and coefficients in the 2-category $\St$ of stable categories) is the 2-category $\PS_\LL(\CC^n):=\Fun(\Sph^{\otimes n},\St)$ of $\Sph^{\otimes n}$-diagrams in stable categories.
\end{definition}

\begin{notation}
    %An object of $\PS_\LL(\CC^n)$ can be specified by an $n$-cube of commuting (more precisely, {\em Beck-Chevalley} in the sense of \cite{CDW}) spherical adjunctions. The vertices of these cubes are in natural bijection with the components of the Lagrangian $\LL$ or with sign vectors. In analogy with \Cref{cor:perv-cn-quiver}, we will denote an object of $\PS_\LL(\CC^n)$ by $(\cC_\bullet, u, v),$ where $\cC_\sgn$ is the category corresponding to $\sgn\in\sgvect,$ and
    An object of $\PS_\LL(\CC^n)$ is specified by the data of a \bit{spherical categorical n-cube} in the sense of \cite{CDW}*{Definition 5.4.3}, namely a functor $2^{[n]}\to \St$ (where we treat $2^{[n]}$ as a partially ordered set) such that each edge of the cube is a spherical functor,
    and each face of the cube satisfies the Beck-Chevalley condition of \cite{CDW}*{Definition 4.5.7}.
    %We will denote such an object as $(\cC^\bullet, u, v),$ where $\cC^\bullet$ 
    Such an object consists of
    a collection of categories $(\cC^\sgn)_{\sgn\in \sgvect},$ together with spherical adjunctions
    \[
    u_i:\cC_{\sgn}\rightleftarrows\cC_{\sgn\cup\{i\}}:v_i
    \]
    satisfying the appropriate commutativity conditions. We will denote this object by $(\cC^\bullet, u,v).$
\end{notation}
%\begin{definition}
%    Given a sign vector $I\in \{\pm\}^n,$ there is a 2-functor $\PS_\LL(\CC^n)\to \St$ which picks out the category at the corresponding vertex of the $n$-cube. We say that an object of $\PS_\LL(\CC^n)$ corepresenting this functor is a \bit{cocore} to the corresponding component $\LL_I.$
%\end{definition}

The $D=(\CC^\times)^n$ action on $\CC^n$ should make $\PS_\LL(\CC^n)$ into a topological $D$-2-category (i.e., the underlying 2-category of a functor $BD\to 2\Pr^{L,\st},$ where the target is the 3-category of presentable stable 2-categories), and indeed this structure was described in \cite{GMH}.

\begin{definition}\label{defn:universal-twist}
    The \bit{universal twist} is the central element $\bT\in\Aut(\id_{\Sph})$ which acts on a spherical adjunction
\end{definition}
%
%
%\begin{lemma}[\cite{GMH}*{Propositions 3.16, 3.22}]\label{lem:sph-equviariant}
%    $\PS_\LL(\CC^n)$ admits a lift to an object of the 3-category $\threeloc(BD)$ of local systems of 2-categories over $BD.$ This structure is induced from an upgrade of $\Sph$ to an object of the 2-category $\threeloc(B\CC^\times),$
%    equipping a spherical adjunction
    \[
        \begin{tikzcd}[column sep=2cm]
        \cC_\Phi
        \arrow[yshift=0.9ex]{r}{F}
        \arrow[leftarrow, yshift=-0.9ex]{r}[yshift=-0.2ex]{\bot}[swap]{F^R}
        &
        \cC_\Psi~
        \end{tikzcd}
    \]
%    with the monodromy automorphism given by the \bit{universal twist}
by the automorphism
    \begin{equation*}
    \begin{tikzcd}[column sep=2cm]
\cC_\Phi
\arrow[yshift=0.9ex]{r}{F}
\arrow[leftarrow, yshift=-0.9ex]{r}[yshift=-0.2ex]{\bot}[swap]{F^R}
\arrow[dashed]{d}[swap]{T_\Phi}
&
\cC_\Psi
\arrow[dashed]{d}{T_\Psi[-2]}
\\
\cC_\Phi
\arrow[yshift=0.9ex]{r}{F}
\arrow[leftarrow, yshift=-0.9ex]{r}[yshift=-0.2ex]{\bot}[swap]{F^R}
&
\cC_\Psi~.
\end{tikzcd}
\end{equation*}
%\end{lemma}
%In other words, the action of $D$ on $\CC^n$ induces $n$ commuting automorphisms (more precisely, an action of $\pi_1(D)$) on $\id_{\PS_\LL(\CC^n)},$ or equivalently $n$ commuting automorphisms (commuting with other morphisms in the 2-category) on each spherical $n$-cube; moreover, if we think of a spherical $n$-cube as a product of $n$ spherical adjunctions, then the $i$th automorphism is given by the universal twist associated to the $i$th adjunction.

\begin{lemma}[\cite{GMH}, Proposition 3.16]
    The universal twist is the generator of an $\EE_2$-map $\ZZ\to \End(\id_\Sph),$ defining a topological $B\ZZ$-action on $\Sph.$ 
    %By taking products
\end{lemma}
By taking products, we obtain a topological $\ZZ^n\simeq D$-action on $\Sph^n,$ which we can restrict along the inclusion $G\to D$ to obtain a $G$-action.
\begin{definition}
    We write $\Sph_G^n$ for the coinvariants of the $G$-action on $\Sph.$ (See \cite{GMH}*{Lemma A.4} for the construction of these coinvariants.)
    The 2-category of \bit{$G$-equivariant perverse schobers} is defined as
    \[
    \PS_{\LL_G}(\CC^n/G):=\Fun(\Sph^n_G,\St).
    \]
\end{definition}
The 2-category $\PS_{\LL_G}(\CC^n/G)$ is therefore obtained as the invariants $\PS_{\LL}(\CC^n)^G$ for the $G$-action on $\PS_{\LL}(\CC^n).$ An object of the resulting 2-category may be specified by the data of a spherical $n$-cube equipped with trivializations of the appropriate compositions of (co)twist automorphisms on each category $\cC^\sgn.$ We will continue to denote such an object by $(\cC^\bullet, u, v),$ leaving the trivializations implicit in our notation.

Certain objects in $\PS_{\LL_G}(\CC^n/G)$ will play a distinguished role. Recall that an object of $\Sph^n$ may be specified by a choice of $\sgn\in\sgvect$; we continue to write $\sgn$ to denote the image of this object in $\Sph^n_G.$
\begin{definition}\label{defn:cat-proj-aside}
    %We write $P^\sgn_G$ for the object of $\Fun(\Sph^n_G, \St)=\PervCat_{\LL_G}(\CC^n/G)$ corepresented by $\sgn$: as a functor on $\Sph^n_G,$ $P^\sgn$
    The object $\sgn$ in $\Sph^n_G$ determines a corepresentable functor $\Hom_{\Sph^n_G}(\sgn,-)$; we will refer to this object of 
    $\Fun(\Sph^n_G,\St) = \PS_{\LL_G}(\CC^n/G)$ by $\cP^\sgn_G.$
\end{definition}

The object $\cP^\sgn_G,$ which categorifies the projective $P^\alpha$ described in \Cref{defn:betti-1cat-projectives}, corepresents the functional on $\PS_{\LL_G}(\CC^n/G)$ which takes a $G$-invariant spherical $n$-cube $(\cC^\bullet, u, v)$ to the category $\cC^\sgn.$
The objects $\cP^\sgn_G$ are therefore generators of the 2-category $\PervCat_{\LL_G}(\Sph^n_G,\St)$; the main theorem of \cite{GMH}, which we have stated as \Cref{thm:gmh-main} above, is a calculation of their monoidal category of endomorphisms $\End\left(\bigoplus \cP^\sgn\right),$ or in other words an explicit description of the spherical $n$-cube corresponding to $\sgn.$

\subsection{Stop removal and microrestriction}
We now describe the 2-categories of perverse schobers associated to certain locally closed subsets of $\LL_G.$ First, recall that
in the 1-categorical setting, stop removal --- i.e., passage to a closed subset of the Lagrangian $\LL_G$ obtained by forgetting some of its irreducible components --- has the categorical effect of quotienting by the corresponding projective objects.
%
%We can use the $\cP^\sgn_G$ to describe perverse schobers on $\CC^n/G$ with singular support in closed subsets of $\LL_G.$
\begin{definition}
    Let $\sgnsub\subset \sgvect,$ and $\LL^I_G:=\bigcup_{\sgn\in\sgnsub}\LL_G.$ Then the 2-category of perverse schobers on $\CC^n/G$ with singular support in $\LL_G^I$ is the quotient 2-category
    \begin{equation}\label{eq:schober-stop-removal}
    \PS_{\LL_G^I}(\CC^n/G):=
    \frac{\PS_{\LL_G}(\CC^n/G)}{\langle \cP^\sgn\rangle_{\sgn\notin\sgnsub}}.
    \end{equation}
\end{definition}
%\begin{lemma}
%    The right adjoint to the quotient functor \Cref{eq:schober-stop-removal} gives a fully faithful embedding $\PS_{\LL_G^I}(\CC^n/G)\hookrightarrow \PS_{\LL_G}(\CC^n/G)$ as the sub-2-category of diagrams $(\cC^\bullet,u,v)$ where $\cC^\sgn\simeq 0$ for each $\sgn\notin\sgnsub$.
%\end{lemma}

\begin{remark}
    Suppose that $G$ is trivial, and $\LL_G^I=\CC^n\subset T^*\CC^n.$ Then an object of $\PervCat_{\CC^n}(\CC^n)$ is a spherical $n$-cube with only one nonzero category $\cC$; this implies that all the spherical functors are the zero functor, hence also that all twists are equal to $\id_\cC,$ and the canonical maps between $\id_\cC$ and the twists are the identity. This observation underlies an equivalence
    \[
    \PervCat_{\CC^n}(\CC^n)\simeq \St 
    \]
    between this 2-category and the 2-category of stable categories, confirming our intuition that $\PervCat_{\CC^n}(\CC^n)$ should be equivalent to the 2-category $\LocCat(\CC^n)$ of local systems of stable categories over $\CC^n$ (which is equivalent to $\St$ by contractibility of $\CC^n$).
\end{remark}

Unlike stop removal, microrestriction --- i.e., passage to an open subset of $\LL_G$ obtained by deleting some of its irreducible components --- needs to work differently in the 1-categorical and 2-categorical settings. We illustrate this disjunction in dimension 1.

%Let $n=1,$ so $\LL=\CC\cup T^*_0\CC,$ and write $\mathring{\LL}:=\CC^\times = \LL\setminus T^*_0\CC\subset T^*\CC.$
Let $n=1,$ so $\LL=\CC\cup T^*_0\CC.$ The $\CC^\times$-conic topology (for the action scaling $\CC^2=T^*\CC$) on $\LL$ has three nonempty open sets, namely $\LL, U_+:=\LL\setminus T^*_0\CC,$ and $U_-:=\LL\setminus \CC.$ Let $\cP^1,\cP^2$ be the projective objects of $\PervCat_\LL(\CC)$ corresponding to $\CC$ and $T^*_0\CC,$ respectively. We can thus imitate the 1-categorical microrestrction described in \Cref{rem:dim1-microrestriction} by defining a conic sheaf of 2-categories on $\LL$:

\begin{definition}
    The sheaf $\mu\PervCat_\LL$ of \bit{microlocal perverse schobers} on $\LL$ takes values $\mu\PervCat_\LL(\LL) := \PervCat_\LL(\CC)$ and $\mu\PervCat_\LL(U_\pm):=\langle \cP^\pm\rangle,$ the latter being the full sub-2-category of $\PervCat_\LL(\CC)$ generated by $\cP^\pm$; restriction maps are right adjoint to the evident inclusions.
\end{definition}

\begin{remark}
    We can now illustrate the difference between the 1- and 2-categorical situations: in the 1-categorical case we have an equivalence
    \[
    \muPerv_\LL(\CC^\times)\simeq \mu\Perv_{\CC^\times}(\CC^\times)\simeq \Loc(\CC^\times)\simeq \Mod_{\CC[x^\pm]},
    \]
    but it is no longer true that $\mu\PervCat_\LL(\CC^\times)$ is equivalent to $\LocCat(\CC^\times).$ %If we write $\Fil:=\Coh(\CC/\CC^\times)$, resp. $\Gr:=\Coh(B\CC^\times)$ for the monoidal categories of filtered, resp. graded vector spaces, then the above definition gives
    Rather, the above definition gives
%    \[\mu\PervCat_\LL(\CC^\times)\simeq \Mod_{\Fil}(\St)\not\simeq \Mod_{\Gr}(\St)\simeq \LocCat(\CC^\times).\]
\begin{equation}\label{eq:microrestriction-differs-remark}
    \mu\PervCat_\LL(\CC^\times)\simeq \Mnd^{sph}\not\simeq \Auts\simeq \LocCat(\CC^\times),
\end{equation}
where the left-hand side of \Cref{eq:microrestriction-differs-remark} is the 2-category of spherical monads --- i.e., the 2-category of categories equipped with a monad $M$ whose twist endomorphism is invertible --- whereas the right-hand side is the 2-category of categories equipped with an automorphism. The latter is the {\em non-full} sub-2-category of the former obtained by forgetting the monad (with its unit and multiplication morphisms) and remembering only the twist automorphism. See also \cite{GMH}*{Remark 2.7} for related discussion.
\end{remark}

%Following our treatment of the hypertoric (Betti 1-)category $\cO$, we define our A-side 2-category $\cO$ as a subquotient of $\PervCat_{\LL_G}(\CC^n/G).$ 
In the setting of the Betti 1-category $\cO,$ microrestriction to the $\FI$-semistable locus had a fully faithful left adjoint, embedding $\mu\Perv_{\LL_G}(\LL_G(\FI,0))$ as the full subcategory of $\Perv_{\LL_G}(\CC^n/G)$ generated by the $\FI$-feasible projective object $P^\cF:=\bigoplus_{\sgn\in \cF}(P^\sgn).$ We will take this as a definition of our microrestriction in the 2-categorical setting.

\subsection{2-category $\cO$}

Let $\langle \cP^\cF_G\rangle\subset \PervCat_{\LL_G}(\CC^n/G)$ be the sub-2-category generated by the objects $\{\cP_G^\sgn\}_{\sgn\in \cF}$ for feasible sign vectors $\sgn,$ and $\langle \cP_G^{\cF\cap \cB^c}\rangle$ the sub-2-categry generated by objects $\cP_G^\sgn$ corresponding to feasible and unbounded sign vectors.

\begin{definition}\label{defn:betti-2cat-o}
    Fix a stability parameter $\FI$ and attraction parameter $\mass.$ We define the 2-category of \bit{microlocal perverse schobers} on $\LL_G(\FI,\mass)$ as the quotient 2-category
    \begin{equation}
        \mu\PervCat(\LL_G(\FI,\mass)):=\frac{\langle \cP_G^\cF\rangle}{\langle \cP_G^{\cF\cap\cB^c}\rangle}.
    \end{equation}
    For generic $\FI$ and $\mass,$ we will refer to this as the \bit{A-side 2-category $\cO$}.
\end{definition}

\begin{remark}
Categorifying \Cref{lem:quiver-algebra}, we can give an explicit description of objects in the 2-category $\mu\PervCat(\LL_G(\FI,\mass)).$
%    \bit{Quiver data} for $\mu\PervCat(\LL_G(\FI,\mass))$ is the 
Consider the following collection of data:
    \begin{itemize}
        \item For each $\sgn$ which is both $\FI$-stable and $\mass$-bounded, a category $\cC^\sgn$.
        \item For each $\cC^\sgn,$ a map $\pi_1(F)\to \Aut(\cC^\sgn)$, which we understand as a map $\pi_1(D)\to \Aut(\cC^\sgn)$ with a trivialization of the composite $\pi_1(G)\to\pi_1(D)\to\Aut(\cC^\sgn).$ We write $T^\sgn_i\in\Aut(\cC^\sgn)$ for the images of the canonical generators of $\pi_1(D)\simeq \ZZ^n.$
        \item On each $\cC^\sgn,$ a monad $M^\sgn_i$ for each $i\in \sgn,$ and a comonad $C^\sgn_i$ for each $i\notin \sgn.$
        \item For each $t$-stable pair $(\sgn,\sgn\cup\{i\}),$ a spherical adjunction
        \[
        \begin{tikzcd}
        \cC^\sgn \arrow[r, "u_i"', shift right] & 
        \cC^{\sgn\cup\{i\}} \arrow[l, shift right, "v_i"']
        \end{tikzcd}
        \]
        together with equivalences $M^{\sgn\cup\{i\}}_i\simeq u_iv_i, C^\sgn_i\simeq v_iu_i$ of monads (resp. comonads), where we take the convention that $\cC^\sgn\simeq 0$ if $\sgn$ is not $\mass$-bounded
        \item For each $\cC^\sgn,$ an equivalence
        $T_i^\sgn\simeq\fib(\id_{\cC^\sgn}\xrightarrow{\eta}M)$ for each $i\in \sgn,$
        and an equivalence
        $T_i^\sgn\simeq\cofib(\id_{\cC^\sgn}\xrightarrow{\eta}M)[-2]$
        for each $i\notin \sgn.$
    \end{itemize}
    Impose moreover the condition that %    We moreover require these data to satisfy the conditions that 
    the $u_i$'s and $v_i$'s all commute with each other and with the monads $M_i^\sgn$ and comonads $C_i^\sgn.$
    
Such a collection of data determines an object in the 2-category $\mu\PervCat(\LL_G(\FI,\mass)).$ We expect that it is possible to define a stable 2-category directly from this quiver data, although we have not done this here. If done correctly, this combinatorially defined 2-category will agree with the 2-category $\mu\PervCat(\LL_G(\FI,\mass))$ defined in \Cref{defn:betti-2cat-o}.
\end{remark}

\begin{example}\label{ex:2catO-A}
    Here we consider the categorification of \Cref{ex:betti-category-o}. Once again, let 
    \[
    G\simeq\CC^\times\hookrightarrow (\CC^\times)^2
    \]
    be the kernel of the multiplication map. Before imposing $G$-equivariance, we have that $\PervCat_{\LL}(\CC^2)$ is the 2-category of Beck-Chevalley squares of spherical functors
    \begin{equation}\label{eq:pervcat-c2-example}
    \begin{tikzcd}
        \cC^{\{1\}}\ar[d, "v_1"'] & \cC^{\{1,2\}}\ar[l, "v_2"']\ar[d, "v_1"]\\
        \cC^{\emptyset} & \cC^{\{2\}}\ar[l, "v_2"],
    \end{tikzcd}
    \end{equation}
%    where we will write $u_i:=$
    Sphericality equips each category $\cC^\sgn$ with two automorphisms, the relevant two of
    \begin{align*}
%       \begin{tikzcd}
           \fib(\id_{\cC^\sgn}\to&v_2^Rv_2),&
           \cofib(v_2v_2^R\to&\id_{\cC^\sgn})[-2],&
           \fib(\id_{\cC^\sgn}\to&v_1^Rv_1),&
           \cofib(v_1v_1^R\to&\id_{\cC^\sgn})[-2].
%       \end{tikzcd} 
    \end{align*}
    We obtain $\PervCat_{\LL_G}(\CC^2/G)$ from this 2-category by imposing additionally an equivalence between the two automorphisms at each category $\cC^\sgn.$ 
    
    Finally, to produce the 2-category $\mu\PervCat(\LL_G(\FI,\mass)),$ we must quotient by the simple object supported on $\cC^{\{2\}}$ and by the projective $\cP_G^\emptyset.$ The first of these operations has the effect of forgetting the lower-right vertex from the diagram \Cref{eq:pervcat-c2-example} (and the functors ingoing and outgoing from it --- but not the associated monad $v_1^Rv_1$ on $\cC^{\{1,2\}}$). The latter operation sets to 0 the category $\cC^\emptyset$, along with the monad $v_1^Rv_1$ on $\cC^{\{1\}},$ so that the automorphism at $\cC^{\{1\}}$ (which remains identified with $\cofib(v_2v_2^R\to \id_{\cC^{\{1\}}})[-2]$) is identified with $\id_{\cC^{\{1\}}}.$

    In summary, the 2-category $\mu\PervCat(\LL_G(\FI,\mass))$ is the 2-category of spherical functors
    \[
    \begin{tikzcd}
        \cC^{\{1\}}&\ar[l,"v_2"']\cC^{\{1,2\}}
    \end{tikzcd}
    \]
    together with an extra monad $``v_1^Rv_1"$ on $\cC^{\{1,2\}},$ and identifications
    \[
    \fib(\id_{\cC^{\{1,2\}}}\xrightarrow{\eta}``v_1^Rv_1")\simeq \fib(\id_{\cC^{\{1,2\}}}\xrightarrow{\eta}v_2^Rv_2),
    \qquad \cofib(v_2v_2^R\to \id_{\cC^{\{1\}}}) \simeq [2].
    \]
\end{example}

\section{Microlocal coherent schobers}\label{sec:bside}
\subsection{Coherent sheaves of categories}
%We begin by recalling the B-side 2-category associated to $\LL_G\subset \rT^*(\CC^n/G),$ as described in \cite{GMH}. 
Beginning with a symplectic variety $\mathfrak{M},$ a proposal is given in \cites{KRS,KR} for a an algebraic 2-category whose objects are supported on Lagrangian subvarieties in $\mathfrak{M},$ and it is further explained that when $\mathfrak{M}=\rT^*Y$ is a cotangent bundle, the part of the category supported on a conical Lagrangian $\LL\subset \rT^*Y$ may be described in terms of 
coherent-sheaf categories of spaces proper over $Y.$ (This may motivate the notation in \cite{Tel-ICM} for this 2-category as $\sqrt{\mathfrak{Coh}}(\rT^*Y).$)

%This 2-category is expected to be a version of the 2-category described in \cites{KR,KRS} after imposing the support condition $\LL_G$ on objects. Indeed, in \cite{KR}*{\S 4} it is explained how the 2-category associated there to a cotangent bundle $\rT^*Y$ may be described in terms of coherent-sheaf categories of spaces proper over $Y$. This will be formalized as the 2-category of coherent sheaves of categories over $Y$, currently under development by Arinkin \cite{Ari-talk} and Stefanich (see \cite{Stefanich-QCoh} for the simpler case of quasicoherent sheaves of categories). 
Given a conic Lagrangian $\LL\subset \rT^*Y$ in the cotangent bundle of an algebraic stack $Y$, we are therefore motivated to study the 2-category of coherent sheaves of categories over $Y$ with singular support in $\LL.$ A full mathematical account of this 2-category has not yet appeared in the literature, though it is expected in forthcoming works of Stefanich and Arinkin.\footnote{See \cite{Ari-talk} for motivation for the construction of this 2-category, and \cite{Stefanich-QCoh} for the simpler case of quasicoherent sheaves of ($n$-)categories.)}
%Unfortunately, a complete definition of the 2-category of coherent sheaves of categories on a smooth scheme or stack $Y$ in general is not yet available in the literature.
Nevertheless, enough is known about this 2-category for our present purposes. 

We will need only the following two facts: first, that if $X$ is a smooth stack with a proper map $X\to Y,$ then $\Coh(X)$ defines an object in this 2-category, and second, that the Hom category between two such objects $\Coh(X)$ and $\Coh(X')$ is the category $\Fun_{\Perf(Y)}^{ex}(\Coh(X),\Coh(X'))$ of exact $\Perf(Y)$-linear functors between the coherent-sheaf categories, with compositions among these Hom categories given by composition of functors.
These Hom categories have been computed directly in \cite{BZNP-kernels}.
\begin{theorem}[\cite{BZNP-kernels}*{Theorem 1.1.3}]\label{thm:BZNP-kernels}
    Let $X,X',Y$ be smooth perfect stacks, and let $X\to Y\gets X'$ be proper maps. Then there is an equivalence
    \[
    \Fun_{\Perf(Y)}(\Coh(X),\Coh(X'))\simeq \Coh(X\times_Y X')
    \]
    between the category of exact $\Perf(Y)$-linear functors and the category of coherent sheaves on the fiber product $X\times_Y X',$ associating to a coherent sheaf $K$ on $\Coh(X\times_Y X')$ the integral transform functor $p_{X',*}(p_X^*(-)\otimes K).$ Moreover, in the case $X=X',$ this is naturally a monoidal equivalence.
\end{theorem}

\Cref{thm:BZNP-kernels} justifies the following definition, due to Arinkin.
\begin{definition}\label{defn:cohcats-general}
    Let $X,Y$ be smooth perfect stacks and $f:X\to Y$ a proper map, and write $\overline{\Conorm_XY}$ for the image in $T^*Y$ of the conormal Lagrangian 
    $\Conorm_XY.$
    %in $\rT^*Y.$ 
    Then we say the 2-category of \bit{coherent sheaves of categories on $Y$ with singular support $\overline{\Conorm_XY}$} is the 2-category
    \begin{equation}\label{eq:cohcats-general}
    \CohCat_{\overline{\Conorm_XY}}(Y):=\Mod_{\Coh(X\times_YX)}(\St)
    \end{equation}
    of module categories for the monoidal category $\Coh(X\times_Y X).$
\end{definition}

\begin{remark}
    The notation in \Cref{defn:cohcats-general} may seem misleading, as a priori it is unclear that the 2-category defined by \Cref{eq:cohcats-general} depends only on the subset $\overline{\Conorm_XY}\subset\rT^*Y$, rather than on the choice of map $X\to Y$.
    %with conormal $\rT^*_XY$. 
    This invariance statement is promised as \cite{Arinkin-notes}*{Theorem 10}, but so far a proof has not appeared in the literature. We provide a proof below as \Cref{prop:invariance-for-cohcats}.
\end{remark}

\subsection{Singular support}
We would like to introduce a modification of \Cref{defn:cohcats-general} to account for non-closed singular support conditions. For this, we will need the theory of coherent singular support developed in \cite{AG}.
Recall that $Z$ a quasi-smooth Artin stack, the {\em stack of singularities} of $Z$ is the classical Artin stack
\[
\Sing(Z):=(\rT^*[-1]Z)^{cl}
%(\Spec_Z(\Sym_{\cO_Z}T[1]Z))^{cl}.
\]
obtained by taking the underlying classical stack of the $(-1)$-shifted cotangent bundle.
If $Z$ is affine, there is a map
\[
\Gamma(\Sing(Z),\cO_{\Sing(Z)})\to \HH^{\text{even}}(Z)
\]
from functions on $\Sing(Z)$ to the even Hochschild cohomology of $Z$. As a result, in general, the category $\Coh(Z)$ localizes over $\Sing(Z)$ (in the conic topology).
\begin{definition}
    Let $\Lambda\subset \Sing(Z)$ be a conic closed subset, with open complement $U$. Then we write $\Coh_\Lambda(Z)\subset \Coh(Z)$ for the full subcategory of coherent sheaves supported on $\Lambda,$ and $\Coh_U(Z):=\Coh(Z)/\Coh_\Lambda(Z)$ for the localization away from sheaves supported on $\Lambda.$
\end{definition}
%We will need some functoriality for singular support conditions.
\begin{definition}
    Given a map $g:Z\to W$ of quasi-smooth stacks, consider the correspondence
    \begin{equation}\label{eq:sing-corr-from-map}
    \begin{tikzcd}
        \Sing(Z)&\ar[l, "dg^*"']\Sing(W)\times_W Z\ar[r, "\tilde{g}"]& \Sing(W),
    \end{tikzcd}
    \end{equation}
    which we will use to define a pushforward and pullback of singular support conditions along $g$:
    \begin{enumerate}
        \item For a conic subset $U\subset \Sing(Z),$ we define 
        %the pushforward singular support condition as
        \[
        g_*U:=\tilde{g}((dg^*)^{-1}(U))\subset \Sing(W).
        \]
        \item For a conic subset $V\subset \Sing(W),$ we define 
        %the pullback singular support condition as
        \[
        g^!V:=df^*(Z\times_W V)\subset \Sing(Z).
        \]
    \end{enumerate}
\end{definition}
\begin{lemma}[\cite{AG}*{Proposition 7.1.3}]
    The pushforward and pullback functors preserve singular supports, inducing functors
    \[
    f_*:\Coh_{\Lambda_Z}(Z)\to \Coh_{f_*\Lambda_Z}(W), \qquad f^!:\Coh_{\Lambda_W}(W)\to\Coh_{f^!\Lambda_Z}(Z).
    \]
\end{lemma}
We refer to \cite{AG} for more information about coherent singular support conditions.

\subsection{Monoidal structure}

\begin{notation}
    Throughout this section, we fix smooth perfect stacks $Y, X_i$ and proper maps $f_i:X_i\to Y.$ We write
    $X_{ij}:=X_i\times_Y X_j$ for the fiber products over $Y$, and similarly $X_{ijk}:=X_i\times_Y X_j\times_Y X_k$ for the triple fiber products. 
    %We denote by $p_i$ (resp. $p_{ij}$) the projection maps from these fiber products to $X_i$ (resp. $X_{ij}$).
\end{notation}

In \Cref{thm:BZNP-kernels} and \Cref{defn:cohcats-general} we have made use of the functoriality of $\Coh$ under correspondences:
to a correspondence
\begin{equation}\label{eq:correspondence}
\begin{tikzcd}
Z & C \arrow[l, swap, "f"] \arrow[r, "g"] & Z'
\end{tikzcd}
\end{equation}
with $g$ proper, we associate the functor.
\begin{equation}\label{eq:convolution-functoriality}
g_* f^*: \Coh(Z) \to \Coh(Z').
\end{equation}
In the case where the correspondence \Cref{eq:correspondence} is
\begin{equation}\label{eq:convolution-monoidal-structure}
\begin{tikzcd}
    X_{ij}\times X_{jk}&\arrow[l, "{(p_{ij}, p_{jk})}"'] X_{ijk}\arrow[r, "p_{jk}"]& X_{ik},
\end{tikzcd}
\end{equation}
we recover the rule for composition of integral transforms:
\begin{definition}\label{defn:convolution-monoidal-structure}
Using \Cref{eq:convolution-functoriality}, the correspondence \Cref{eq:convolution-monoidal-structure} (together with the higher fiber products for associativity data) defines
a monoidal structure on the category $\Coh(X_{ii})$ together with a $\Coh(X_{ii}),\Coh(X_{jj})$-bimodule structure on the category $\Coh(X_{ij}).$ We call this the \bit{convolution monoidal structure.}
\end{definition}
%and, when $i=j,$ the monoidal structure on $X_{ii}$.
%\begin{definition}\label{defn:convolution-monoidal-structure}
%    \Cref{eq:convolution-monoidal-structure} defines the convolution monoidal structure on the category $\cC_i:=\Coh(X_{ii}),$ and the $\cC_i,\cC_j$-bimodule structure on the category $\cM_{ij}:=\Coh(X_{ij}).$
    %We write $\cC_i:=\Coh(X_{ii})$ for the monoidal categories of coherent sheaves on the fibered self-products, and $\cM_{ij}:=\Coh(X_{ij})$ for the $\cC_i,\cC_j$-bimodule categories.
%\end{definition}

%We will be interested in imposing singular support conditions on a monoidal category of the form $Coh(X\times_Y X)$, where
%$f:X\to Y$ is of the form studied in \Cref{defn:cohcats-general}. More generally, consider a number of such maps $f_i:X_i\to Y,$ and write $Z_{ij}:=X_i\times_Y X_j,$ so that $\cC_i:=\Coh(Z_{ii})$ is a monoidal category, and $\cM_{ij}:=\Coh(Z_{ij})$ is a $\cC_i,\cC_j$-bimodule category. (Similarly, we write $Z_{ijk}:=X_i\times_Y X_j\times_Y X_k$ for the triple fiber product.) This situation has previously been studied in \cite{BZCHN}, from which we cite the following proposition.
We will need to understand how this 
%functoriality of \Cref{eq:convolution-functoriality} 
algebraic structure 
on the categories $\Coh(X_{ij})$
%of \Cref{defn:convolution-monoidal-structure}
interacts with the the theory of coherent singular support. This question has been previously studied in \cites{BZNP,BZCHN}.

\begin{definition}\label{defn:convolution}
    Let $\Lambda_{12}\subset \Sing(X_{12})$ and $\Lambda_{23}\subset \Sing(X_{23}),$ and consider the diagram
    \begin{equation}\label{eq:convolution-correspondence}
    \begin{tikzcd}
        X_{12}\times X_{23}&\ar[l,"\delta_2"']
        X_{123}\ar[r,"p_{13}"]&
        X_{13},
    \end{tikzcd}
    \end{equation}
    where we write $\delta_2$ as shorthand for $(p_{12},p_{23}).$
    The \bit{convolution} of singular support conditions $\Lambda_{ij}$ is defined by
    \begin{equation}\label{eq:def-convolution}
    \Lambda_{12}*\Lambda_{23}:=(p_{13})_*(\delta_2)^!(\Lambda_{12}\boxtimes\Lambda_{23}).
    \end{equation}
    The singular support condition $\Lambda_{ij}$ is \bit{$X_{ii}$-stable} if it is preserved by convolution with $\Sing(X_{ii})$:
    \[
    \Sing(X_{ii})*\Lambda_{ij}\subset \Lambda_{ij}.
    \]
\end{definition}
\begin{remark}
    %The operation of convolution may be more succinctly understood in the language of Lagrangian correspondences: given a correspondence \Cref{eq:correspondence}, the shifted conormal to $C$ defines a Lagrangian correspondence
    The operation of convolution may be more succinctly reframed in the language of Lagrangian correspondences: 
    given a correspondence \Cref{eq:correspondence}, the shifted conormal to $C$ defines a Lagrangian correspondence
    \begin{equation}\label{eq:lag-correspondence}\begin{tikzcd}
        \rT^*[-1]X&\arrow[l] \Conorm_C[-1](X\times Y)\arrow[r]&\rT^*[-1]Y.
        \end{tikzcd}
    \end{equation}
    %where we write $\overline{T}^*$ to denote the negation of the standard symplectic structure. 
    A Lagrangian $\Lambda\to \rT^*[-1]X$ may be understood as a Lagrangian correspondence
    \[
    \begin{tikzcd}
    \pt & \arrow[l]\Lambda\arrow[r] & \rT^*[-1]X,
    \end{tikzcd}
    \]
    which we may compose with the Lagrangian correspondence $\Conorm_C[-1](X\times Y)$ to obtain a new Lagrangian
    \begin{equation}\label{eq:lagrangian-convolution}
    \Lambda\circ C:= \Lambda\times_{\Conorm_C[-1](X\times Y)}\rT^*[-1]X\to \rT^*[-1]Y.
    \end{equation}
    After taking classical part $(\Lambda\circ C)^{cl},$ we recover the singular-support condition $g_*f^!\Lambda$ obtained by transferring the singular-support condition $\Lambda^{cl}$ along the correspondence $C$.
%    \begin{lemma}
%        Let $C$ be a correspondence as above,
        %in \Cref{eq:correspondence}, 
%        and let $\Lambda \to \rT^*[-1]X$ be a Lagrangian whose classical part defines a closed conic subset $\Lambda^{cl}\subset \Sing(X).$ Then the singular-support condition $g_*f^!(\Lambda^{cl})$ is given as the classical part of the Lagrangian obt
%    \end{lemma}

Moreover, the Lagrangian correspondence \Cref{eq:lag-correspondence} may be factored as the composition of a pair of Lagrangian correspondences. Given a map $g:Z\to Z',$ the shifted conormal to the graph of $Z$ is equivalent to the pullback of the shifted cotangent bundle, 
\[
\Conorm_{\Gamma_g}[-1](Z\times Z')\simeq \rT^*[-1]Z'\times_{Z'}Z,
\]
%\[
%\begin{tikzcd}
%\overline{T}^*[-1]Z & \arrow[l] \Conorm_{\Gamma_g}(Z\times Z') \arrow[r] & 
%\end{tikzcd}
%\]
and therefore it defines the Lagrangian correspondence \Cref{eq:sing-corr-from-map}, with the pushforward and pullback $g_*, g^!$ being the respective pushforward and pullback with this Lagrangian correspondence with its reverse.
Given a correspondence \Cref{eq:correspondence}, the convolution \Cref{eq:lagrangian-convolution} is the composition of the pushforward and pulback $g_*f^!,$ implemented by the composed Lagrangian correspondence
\begin{equation}\label{eq:composed-correspondence}
\Conorm_C[-1](X\times Y) \simeq \Conorm_{\Gamma_f}[-1](X\times C)\circ \Conorm_{\Gamma_g}[-1](C\times X).
\end{equation}
In what follows, we will often be able to simplify computations involving convolution using the language of Lagrangian correspondences, with the understanding that the Lagrangians discussed in these terms are derived schemes, and at the end of a computation we must pass to their classical parts to recover a singular support condition in the singularity space $\Sing.$
\end{remark}

We would like to relate the convolution of Lagrangians to the categorical convolution structures of \Cref{defn:convolution-monoidal-structure}.
\begin{definition}
    We will write $L_i:=\Conorm_{X_i}Y$ for the conormal to the map $X_i\to Y,$ which comes equipped with a map $L_i\to \rT^*Y$ We write
    \[
    L_{ij}:=L_i\times_{\rT^*Y} L_j,\qquad L_{ijk}:=L_i\times_{\rT^*Y}L_j\times_{\rT^*Y}L_k
    \]
    for their fiber products over $\rT^*Y,$ and write $q_i$ (resp. $q_{ij}$) for the projections to $L_i$ (resp. $L_{ij}$).
\end{definition}
\begin{lemma}
    There is an equivalence
    $\rT^*[-1]X_{ij}\simeq L_{ij}.$ Taking classical parts, we obtain an equivalence
    $\Sing(X_{ij})\simeq L_{ij}^{cl}.$
\end{lemma}
\begin{proof}
%We continue to study the monoidal category of coherent sheaves on the fiber product $X\times_Y X,$ where $f:X\to Y$ is as studied in \Cref{defn:cohcats-general}. We would like to relate singular-support conditions on this category to conic subsets of the cotangent bundle $\rT^*Y,$ which we will do by
%The relevance of coherent singular-support conditions, and the relation to \Cref{defn:cohcats-general}, is indicated by the following calculation.
%\begin{lemma}
%    With $f:X\to Y$ as in \Cref{defn:cohcats-general}, there is a map
%In this case, we will define a map
%constructing a map
%    \begin{equation*}
%        F:\Sing(X\times_Y X) \to \rT^*Y,
%    \end{equation*}
%    with the property that the preimage of a conic subset of $\rT^*Y$ is a conic subset of $\Sing(X\times_YX).$
%To define the map $F$, first 
%First observe that there is an isomorphism 
%\[
%\gamma: p_i^*f_i^*\Cotan_Y\simeq p_j^*f_j^*\Cotan_Y
%\]
%between the pullbacks of $\Cotan_Y$ along the two different routes to $Y.$
%\[
%p_i:X_i\times_Y X_j\to X.
%\]
%Then we can present the cotangent complex of the fiber product $X_{ij}$ explicitly, as the complex (with first term in degree 0)
%\[
%    \Cotan_{X_{ij}}\simeq \left(p_i^*f_i^*\Cotan_Y
    %\xrightarrow{(p_1^*df^*, p_2^*df^*\gamma)} 
%    \to
%    p_i^*\Cotan_{X_i}\oplus p_j^*\Cotan_{X_j}\right),
%\]
%where the map to $p_j^*\Cotan_{X_j}$ involves precomposition with $\gamma.$
%After shifting by $[-1],$ the space of sections of this complex is precisely the fiber product $\Conorm_{X_i}Y\times_{\rT^*Y}\Conorm_{X_j}Y.$
Consider the cocartesian diagram
\[
\begin{tikzcd}
\Cotan_Y \arrow[r] \arrow[d] & \Cotan_{X_j} \arrow[r] \arrow[d] & 0 \arrow[d] \\
\Cotan_{X_i} \arrow[r] \arrow[d] & \Cotan_{X_{ij}} \arrow[d] \arrow[r] & \Cotan_{f_i} \arrow[d]  \\
0 \arrow[r] & \Cotan_{f_j} \arrow[r] & \Cotan_Y[1]
\end{tikzcd}
\]
%\[
%\begin{tikzcd}
%p_i^* f^*_i \LL_Y \cong p_j^* f^*_j \LL_Y %\arrow[r] \arrow[d] & p^*_j \LL_{X_j} \arrow[r] %\arrow[d] & 0 \arrow[d] \\
%p^*_i \LL_{X_i} \arrow[r] \arrow[d] & \LL_{X_{i,j}} \arrow[d] \arrow[r] & \LL_{p_j} \cong p^*_i \LL_{f_i} \arrow[d]  \\
%0 \arrow[r] & \LL_{p_i} \cong p^*_j \LL_{f_j} \arrow[r] & p_j^* f^*_j \LL_Y[1] \cong p_i^* f^*_i \LL_Y [1].
%\end{tikzcd}
%\]
of sheaves on $X_{ij}$. Here the upper left cocartesian square is the standard formula for the cotangent complex of a pullback, the left and upper  cocartesian rectangles define the relative cotangent complexes, and the outer cocartesian square is the definition of the suspension.

The desired map $\rT^*[-1]X_{ij}\to L_{ij}$ comes from noticing that $L_i$ and $L_j$ are the total spaces of $\Cotan_{f_i}[-1]$ and $\Cotan_{f_j}[-1]$ respectively. The fact that this is a symplectomorphism is Remark 2.20 in \cite{Calaque}.
\end{proof}

\begin{lemma}\label{lem:lijk-does-convolution}
    There is an equivalence of Lagrangian correspondences
%    \[
%    L_{ijk}\simeq \Conorm_{X_{123}}(X_{12}\times X_{23}\times X_{13})
%    \]
\[
\begin{tikzcd}
\rT^*[-1]X_{12} \times \rT^*[-1]X_{23} \arrow[d, "\sim" {anchor=south, rotate=90}] & \Conorm_{X_{123}}[-1] (X_{12} \times X_{23} \times X_{13}) \arrow[l] \arrow[r] \arrow[d,"\sim" {anchor=south, rotate=90}] &  \rT^*[-1]X_{13} \arrow[d, "\sim" {anchor=south, rotate=90}] \\
L_{12} \times L_{23} & L_{123} \arrow[l] \arrow[r]  & L_{13}
\end{tikzcd}
\]
    between the correspondence underlying the operation of convolution on Lagrangians and that given by the triple fiber product of conormals.
\end{lemma}
\begin{proof}
   There is a strongly cocartesian diagram
\[\begin{tikzcd}
\Cotan_{X_2} \arrow[rr] \arrow[dd] & & \Cotan_{X_{23}} \arrow[dd] & \\
& \Cotan_Y \arrow[ul] \arrow[rr, crossing over] & & \Cotan_{X_3} \arrow[ul] \arrow[dd] \\
\Cotan_{X_{12}} \arrow[rr] & & \Cotan_{X_{123}} &  \\
& \Cotan_{X_1} \arrow[from=uu, crossing over] \arrow[ul] \arrow[rr] & & \Cotan_{X_{13}} \arrow[ul]
\end{tikzcd}\]
of sheaves on $X_{123}$. Using the fact that the total cofiber of any face of the cube must vanish, we find
\begin{align*}
\cofib(\Cotan_{X_{13}} \to \Cotan_{X_{123}}) \simeq \cofib(\Cotan_{X_{3}} \to \Cotan_{X_{23}}) \simeq \cofib(\Cotan_Y \to \Cotan_{X_2}) \simeq \Cotan_{f_2}.
\end{align*}
The $(-1)$-shifted conormal bundle $\rT^*_{X_{123}}[-1](X_{12} \times X_{13} \times X_{23})$ is the total space of the complex
\begin{align*}
\fib(\Cotan_{X_{12}} \oplus \Cotan_{X_{13}} \oplus \Cotan_{X_{23}} \to \Cotan_{X_{123}})[-1] &\simeq \cofib(\Cotan_{X_{12}} \oplus \Cotan_{X_{13}} \oplus \Cotan_{X_{23}} \to \Cotan_{X_{123}})[-2] \\
&\simeq \cofib(\Cotan_{X_{12}} \oplus \Cotan_{X_{23}} \to \cofib(\Cotan_{X_{13}} \to \Cotan_{X_{123}}))[-2] \\
& \simeq \cofib(\Cotan_{X_{12}} \oplus \Cotan_{X_{23}} \to \Cotan_{f_2})[-2] \\
& \simeq \fib(\Cotan_{X_{12}} \oplus \Cotan_{X_{23}} \to \Cotan_{f_2})[-1].
\end{align*}
The total space of the last complex is
$
L_{12} \times_{L_2} L_{23} \simeq L_{123},
$
as desired.
\end{proof}

\begin{definition}\label{defn:map-F}
    We write
    \[
    F:\Sing(X_{ij})\to \rT^*Y
    \]
    for the map induced on classical parts by the projection $L_{ij}=L_i\times_{\rT^*Y}L_j\to \rT^*Y.$
\end{definition}

The map $F$ just defined
%The following calculation 
gives some justification for the relevance of singular support conditions, and the relation to \Cref{defn:cohcats-general}.
\begin{lemma}\label{lem:image-is-conormal}
    The image of the map $F$ coincides with the intersection in $\rT^*Y$ of the images of the conormals $\Conorm_{X_i}Y\to \rT^*Y.$ In particular, when $X_i=X_j=X,$ the image of $F$ coincides with the image of $\Conorm_XY\to \rT^*Y.$
\end{lemma}
\begin{proof}
    This is immediate from the identification of $\Sing(X_{ij})$ with the fiber product of conormals in $\rT^*Y.$
\end{proof}

We now specialize to the case where $X_i=X$ for all $i$, so that we are studying the monoidal category $\Coh(X\times_YX).$
We will use the map $F$ to pull back conic subsets of $\rT^*Y$ (or of the image of $\Conorm_XY$ in $\rT^*Y$) to singular support conditions on $X\times_YX.$ We must first check that such singular-support conditions are respected by the monoidal structure on $\Coh(X\times_YX)$:
\begin{lemma}
    Let $\Lambda\subset \rT^*Y$ be a closed conic subset. Then the pullback $F^{-1}(\Lambda)$ is stable under convolution with $\Sing(X\times_YX).$
\end{lemma}
\begin{proof}
    From \Cref{lem:lijk-does-convolution}, we know that the convolution of Lagrangians (or, in this case, coisotropics) may be computed by composition with the Lagrangian correspondence
    \[
    \begin{tikzcd}
        L_{12}\times L_{23}&\arrow[l] L_{123} \arrow[r] & L_{13},
    \end{tikzcd}
    \]
    so that $F^{-1}(\Lambda)*\Sing(X\times_YX)$ is the coisotropic in $L_{123}$ described by the fiber product
    \[
    (F^{-1}(\Lambda)\times L_{23})\times_{L_{12}\times L_{23}}L_{123}\to L_{123}.
    \]
    This fiber product imposes the condition on points of $L_{123}$ that their projection to $\rT^*Y$ lives in $\Lambda,$ and its image in $L_{13}$ again consists of points satisfying the same condition.
\end{proof}

Our main tool in relating convolution singular supports to convolution of coherent sheaf categories will be the following calculation:
\begin{proposition}[\cite{BZCHN}*{Proposition 3.30}]\label{prop:BZCHN-sing-convolution}
    Let $\Lambda_{12}\subset \Sing(Z_{12}),$ $\Lambda_{23}\subset \Sing(Z_{23})$ be $Z_{22}$-stable conic subsets. Then there is an equivalence of categories
    \[
        \Coh_{\Lambda_{12}}(Z_{23})\otimes_{\Coh(Z_{22})}\Coh_{\Lambda_{23}}(Z_{23}) \simeq
        \Coh_{\Lambda_{12}*\Lambda_{23}}(Z_{13}).
    \]
\end{proposition}

\Cref{prop:BZCHN-sing-convolution} allows us to establish the following invariance result, which justifies the notation of \Cref{defn:cohcats-general}.
\begin{proposition}\label{prop:invariance-for-cohcats}
    Suppose that the images in $\rT^*Y$ of the conormals $\Conorm_{X_1}Y$ and $\Conorm_{X_2}Y$ coincide. Then there is an equivalence of module 2-categories
    \[
    \Mod_{\Coh(X_1\times_Y X_1)}\simeq \Mod_{\Coh(X_2\times_Y X_2)},
    \]
    given by tensoring with the bimodule category $\Coh(X_1\times_Y X_2).$
\end{proposition}
\begin{proof}
    To see that the bimodule $\Coh(X_{12})$ defines a Morita equivalence between the monoidal categories $\Coh(X_{11})$ and $\Coh(X_{22}),$ we will show that the bimodule $\Coh(X_{12})$ gives the inverse Morita equivalence, by exhibiting an equivalence
    \begin{equation}\label{eq:goal-morita}
    \Coh(X_1\times_Y X_2)\otimes_{\Coh(X_2\times_Y X_2)}\Coh(X_2\times_Y X_1)\simeq \Coh(X_1\times_Y X_1).
    \end{equation}
    By \Cref{prop:BZCHN-sing-convolution}, the left-hand side of \Cref{eq:goal-morita} is equivalent to the category
    $\Coh_{\Lambda^{cl}}(X_1\times_Y X_1),$ where $\Lambda^{cl}=\Sing(X_{12})*\Sing(X_{21})$ is given by convolution of the coisotropic full singular-support conditions for $X_{12}$ and $X_{21}.$

    $\Lambda^{cl}$ is the underlying classical (and reduced) space of the coisotropic $\Lambda$ given by applying the Lagrangian correspondence $L_{121}$ to $L_{12}\times L_{23}$:
    \[
    \Lambda\simeq (L_{12}\times L_{21})\times_{L_{12}\times L_{21}}L_{121}\simeq L_{121}\to L_{11}.
    \]
    By assumption, the images in $\rT^*Y$ of the conormals to $X_1$ and $X_2$ coincide, so that the map on sets induced by $L_{121}\to L_{11}$ is surjective, and we conclude that $\Lambda^{cl}=\Sing(X_{11})$ is the entire singular-support condition for $X_{11},$ so that the left-hand side of \Cref{eq:goal-morita} is equivalent to the whole category of coherent sheaves on $X_1\times_Y X_1.$
\end{proof}

\subsection{Microlocalization}\label{sec:bside-muloc}
We now return to the case of a single map $f:X\to Y$ (where $X$ may be disconnected). Using the map $F:\Sing(X\times_YX)\simeq (\Conorm_XY\times_{\rT^*Y}\Conorm_XY)^{cl}\to \rT^*Y,$
we propose the following definition, allowing us to study coherent sheaves of categories on $Y$ microlocally away from the zero-section in $\rT^*Y.$
\begin{definition}\label{defn:cohcats-microlocal}
    Let $X\to Y$ be as in \Cref{defn:cohcats-general}, 
    %and let $U\subset \rT^*Y$ be an open subset of the cotangent bundle $\rT^*Y$ with complement $\Lambda.$
    and let $U\subset \overline{\Conorm_XY}$ be an open subset of the image in $\rT^*Y$ of the conormal $\Conorm_XY,$ with closed complement 
    $\Lambda = \overline{\Conorm_XY}\setminus U$.
    Then we define the 2-category of \bit{microlocal coherent sheaves of categories} on $U$ to be the 2-category
    \[
        \mu\CohCat(U) := \Mod_{\Coh_{F^{-1}(U)}(X\times_YX)}(\St)
        %\simeq \Mod_{\frac{\Coh(X\times_Y X)}{\Coh_{F^{-1}(\Lambda)}(X\times_YX)}}(\St)
    \]
    of module categories for the monoidal category $\Coh_U(X\times_YX)\simeq\frac{\Coh(X\times_Y X)}{\Coh_{F^{-1}(\Lambda)}(X\times_YX)}$ of coherent sheaves with singular support in $U$.
\end{definition}
%{\color{red}(Fix the notation in the above definition.)}

The microlocal analogue of \Cref{prop:invariance-for-cohcats} remains true.
\begin{proposition}\label{prop:invariance-for-mucohcats}
    Suppose that the images in $T^*Y$ of $\Conorm_{X_1}Y$ and $\Conorm_{X_2}Y$ coincide. Then there is an equivalence of 2-categories
    \[
    \mu\CohCat(U\cap\overline{\Conorm_{X_1}Y})\simeq
    \mu\CohCat(U\cap\overline{\Conorm_{X_2}Y}).
    \]
\end{proposition}
\begin{proof}
    The monoidal category $\overline{\cA}_i:=\Coh_{F^{-1}}(U)(X_i\times_Y X_i)$ is a quotient of the monoidal category $\cA_i:=\Coh(X_i\times_Y X_i)$ be the ideal $\cI_i$ generated by $\Coh_{F^{-1}(\Lambda)}(X_i\times_Y X_i).$ 
    
    From \Cref{prop:invariance-for-cohcats}, we know that the $\cA_1,\cA_2$-bimodule category $\cM:=\Coh(X_2\times_Y X_1)$ gives a Morita equivalence between $\cA_1$ and $\cA_2.$ To show that the quotient $\cM/\cM\cI_1$ induces a Morita equivalence between the quotient algebras $\overline{\cA}_i,$ it is sufficient to show that $\cI_2\cM$ and $\cM\cI_1$ are equivalent subcategories of $\cM,$ or in other words that the two convolutions
    \[
        \Coh_{F^{-1}(\Lambda)}(X_2\times_Y X_2)\otimes_{\Coh(X_2\times_Y X_2)}\Coh(X_2\times_Y X_1),
        \quad
        \Coh_{F^{-1}(\Lambda)}(X_2\times_Y X_1)\otimes_{\Coh(X_1\times_Y X_1)}\Coh(X_1\times_Y X_1)
    \]
    agree. By \Cref{prop:BZCHN-sing-convolution} we are reduced to checking that the convolutions $F^{-1}(\Lambda)*L_{21}$ and $L_{21}*F^{-1}(\Lambda)$ determine the same singular support condition inside of $\Sing(X_2\times_Y X_1).$ Since both of these convolutions agree with $F^{-1}(\Lambda)\subset \Sing(X_2\times_Y X_1),$ we are done.
\end{proof}

\begin{remark}
    By \cite{BZNP}*{Theorem 1.2.10}, the monoidal center of the monoidal category $\Coh(X\times_Y X)$ is given by the category $\Coh_{prop/Y}(\LBet Y)$ of coherent sheaves on the loop space $\LBet Y:=\Maps(S^1, Y)$ whose pushforward to $Y$ is coherent. If $Y$ is a scheme, then $\Sing(\LBet Y)\simeq \rT^*Y,$ and localization over the monoidal center of $\Coh(X\times_Y X)$ recovers the microlocalization defined above. For $Y$ a stack, $\LBet Y$ offers a refinement of this microlocalization.
\end{remark}

We now specialize to the situation that $X=\bigsqcup_\sgn X_\sgn$ is a disjoint union of components, so that the monoidal category $\Coh(X\times_YX)$ splits (non-monoidally) into its ``matrix entries'': 
\[\Coh(X\times_YX)\simeq \bigoplus_{(\sgn,\sgn')}\Coh(X_\sgn\times_Y X_{\sgn'}). \]
The diagonal entries
\begin{equation}\label{eq:integralkernel-on-xa}
\Coh(X_\sgn\times_Y X_\sgn)\simeq \Fun_{\Perf(Y)}(\Coh(X_\sgn),\Coh(X_\sgn))
\end{equation}
of this matrix are themselves monoidal categories, with unit object (corresponding to the identity functor under the equivalence \Cref{eq:integralkernel-on-xa}) given by the pushforward $(\Delta_\sgn)_*\cO_{X_\sgn}$ under the diagonal map
\[
\Delta_\sgn:X_\sgn\to X_\sgn\times_Y X_\sgn.
\]
This object plays a distinguished role:
\begin{definition}\label{defn:matrix-idempotent}
    With $X=\bigsqcup_\sgn X_\sgn$ as above, we write 
        \[
        E_\sgn:=(\Delta_\sgn)_*\cO_{X_\sgn}
        \]
        for the \bit{matrix idempotent} corresponding to $X_\sgn.$
\end{definition}

%Writing $\cA:=\Coh(X\times_Y X)$ for our monoidal category, we would like to give a description of the quotient monoidal category $\cA/(\cA E_\sgn \cA)$ in terms of singular-support conditions on $X\times_Y X.$
\begin{proposition}\label{prop:gen-by-idempotents}
    Let $\sgnsub=\{\sgn_1,\ldots,\sgn_k\}.$ Write $X_\sgnsub:=\bigsqcup_{\sgn_i\in \sgnsub} X_\sgn$ and let $\Lambda_\sgnsub$ be the image in $\rT^*Y$ of the conormal $\Conorm_{X_\sgnsub}.$ Then
    the smallest ideal in $\Coh(X\times_Y X)$ containing $\langle E_{\sgn_i}\rangle_{\sgn_i\in I}$ is $\Coh_{F^{-1}(\Lambda_I)}(X\times_YX).$
    %where $\Lambda_\sgnsub$ is the image in $\rT^*Y$ of the conormal $\Conorm_{X_\sgnsub}Y.$
\end{proposition}
\begin{proof}
    Let $E_\sgnsub=\bigoplus_{\sgn_i\in \sgnsub}E_{\sgn_i}.$ 
    %Since $E_{\sgn_i}$ 
    Since $E_\sgnsub$ is the monoidal unit of $\Coh(X_\sgnsub\times_Y X_\sgnsub),$ the smallest ideal containing $E_\sgnsub$ may be written as the tensor product
    \begin{equation}\label{eq:ideal-as-relative-tensor}
        \Coh(X\times_YX_\sgnsub)\otimes_{\Coh(X_\sgnsub\times_Y X_\sgnsub)}\Coh(X_\sgnsub\times_YX).
    \end{equation}
    More
    By \Cref{prop:BZCHN-sing-convolution}, \Cref{eq:ideal-as-relative-tensor} is equivalent to $\Coh_\Lambda(X\times_Y X),$ where $\Lambda$ is the convolution of $\Sing(X\times_Y X_\sgnsub)$ with $\Sing(X_\sgnsub\times_Y X).$ In the notation of the previous section, writing $X_1=X_3=X$ and $X_2=X_\sgnsub,$ we have
    \[
    \Lambda \simeq L_{12}*L_{23}\simeq (L_{12}\times L_{23})\times_{L_{12}\times L_{23}}L_{123}\simeq L_{123}\to L_{13}.
    \]
    The image of $L_{123}$ in $L_{13}^{cl}\simeq \Sing(X\times_YX)$ agrees with the set of points whose image under $F:\Sing(X\times_YX)\to \rT^*Y$ agrees with the image in $\rT^*Y$ of the conormal $\Conorm_{X_\sgnsub}Y.$
\end{proof}

\begin{definition}\label{defn:cat-simple-bside}
    Let $X=\bigsqcup_\sgn X^\sgn\to Y$ as above. Then we write $\cS^\sgn$ for the $\Coh(X\times_YX)$-module category $\Coh(X^\sgn\times_Y X).$ For $U\subset \overline{\Conorm_XY}$ open, we continue to write $\cS^\sgn$ for the image of $\cS^\sgn$ in $\mu\CohCat(U).$ 
\end{definition}

\subsection{The hypertoric 2-category}

We now specialize to the case of interest to us. 
%Once again we let $G\subset (\GG_m)^n$ as in \Cref{eq:basic-exact-sequence}. 
%Let $z_1,\ldots,z_n$ denote the standard coordinates on the vector space $\CC^n$ on which $G$ acts.
%
\begin{notation}
As in \Cref{defn:xsgn}, for a sign vector $\sgn,$ we write $X^\sgn_G:=\{\prod_{i\in\sgn} z_i=0\}/G$ for the $G$-equivariant intersection of coordinate hyperplanes in $\AA^n,$ equipped with its closed embedding $X^\sgn_G\hookrightarrow \AA^n/G.$ 
%    For $\sgn\in \sgvect,$ we will write $X_G^\sgn:= \{\prod_{i\in \sgn}z_i=0\}/G$ for the $G$-equivariant intersection of coordinate hyperplanes in $\AA^n,$ which we equip with its evident closed embedding $X_G^\sgn\hookrightarrow \AA^n/G.$ 
We write $X_G:=\bigsqcup_{\sgn\in\sgvect}X_G^\sgn,$ equipped with the map $X_G\to \AA^n/G$ restricting to the standard embedding on each component. More generally, for $\sgnsub\subset \sgvect$ a collection of sign vectors, we write $X_G^\sgnsub:=\bigsqcup_{\sgn\in\sgnsub}X_G^\sgn.$
\end{notation}
\begin{notation}\label{notation:b-endomorphism-cats}
    Let $\cA_G:=\Coh(X_G\times_{\AA^n/G} X_G),$ equipped with a monoidal structure by convolution. More generally, for $\sgnsub\subset \sgvect$ a collection of sign vectors, we write $\cA_G^\sgnsub:=\Coh(X_G^\sgnsub\times_{\AA^n/G} X_G^\sgnsub).$
\end{notation}

%By definition, the Lagrangian $\LL_G^\sgn$ is equal to the conormal $\rT^*_{X_G^\sgn}(\AA^n/G).$
The relevance of $\cA_G^\sgnsub$ is immediate from \Cref{defn:cohcats-general}:
\begin{corollary}
    %Treating $\LL_G^\sgnsub$ as the conormal to the map $X_G^\sgnsub\to \AA^n/G,$ we have 
    There is an equivalence of 2-categories
\begin{equation}
    \CohCat_{\LL_G^\sgnsub}(\AA^n/G) \simeq \Mod_{\cA_G^\sgnsub}(\St)
\end{equation}
between $\CohCat_{\LL_G^\sgnsub}(\AA^n/G)$ and the 2-category of module categories for $\cA_G^\sgnsub.$
\end{corollary}
%\begin{observation}
%    The 2-category $\CohCat_{\LL_G^\sgnsub}(\AA^n/G)$ comes equipped with $|I|$ canonical generators, corresponding to the $|I|$ disjoint components of the space $X_G^\sgnsub.$
%\end{observation}
\begin{example}
    Let $n=1$ and $G=\{1\}\subset (\GG_m)^n$ be the trivial torus. Then $\CohCat_{\LL_G}(\AA^1)$ is the 2-category of module categories over the monoidal category
    \begin{equation}\label{eq:moncat-example}
    \cA_G =
    \left(\begin{array}{cc}
        \Coh(\AA^1)&\Coh(0)\\
        \Coh(0) & \Coh(0\times_{\AA^1} 0)
    \end{array}\right).
    \end{equation}
    Each entry in the matrix \Cref{eq:moncat-example} admits a single categorical generator, with a commutative ring of endomorphisms, so that we may rewrite the matrix as
    \[
    \left(\begin{array}{cc}
        \Perf_{\kk[x]} & \Perf_{\kk}\\
        \Perf_{\kk} & \Perf_{\kk[\beta]}
    \end{array}\right),
    \]
    where $\beta$ is in cohomological degree $2.$
    %and we write $\kk=\CC$ for the coefficient ring. (As mentioned in \Cref{rem:coefficients}, all the constructions in the current section make sense for a general choice of coefficients $\kk.$) 
    In other words, the 2-category $\CohCat_{\LL_G}(\AA^1)$ has two objects, four generating 1-morphisms (namely, the rings $\kk[x],\kk,\kk,$ and $\kk[\beta]$ above) and two generating 2-morphisms (namely, $x$ and $\beta$). The relations among these can be made explicit, so that this 2-category thus admits a completely combinatorial description. Working in the same fashion, such a description can be given for all the 2-categories $\CohCat_{\LL_G}(\AA^n),$ for a general choice of $G$. (However, if $G$ is nontrivial, then the set of generating 1-morphisms in this description must be allowed to be infinite.)
\end{example}

\begin{example}\label{ex:b2cat-bounded}
For $m\in \frf_\ZZ$ a choice of attraction parameter, we may let $\sgnsub = m\bdd\subset \sgvect$ be the collection of sign vectors which are $m$-bounded. In this case, $\LL_G^{m\bdd}=\LL_G(0,m),$ and we recover the 2-category
\begin{equation}\label{eq:bounded-bside-2cat}
\CohCat_{\LL_G(0,m)}(\AA^n/G)\simeq \Mod_{\cA_G^{m\bdd}}(\St)
\end{equation}
associated to the $m$-bounded locus $\LL_G(0,m)=\LL_G^{m\bdd}$ in the Lagrangian $\LL_G.$
\end{example}

%The 2-category described in \Cref{eq:bounded-bside-2cat} is of the form defined in \Cref{defn:cohcats-general}, but the microlocal theory developed in \Cref{sec:bside-muloc} allows us to study the 2-category of microlocal coherent sheaves on $\LL_G(\FI,\mass)$ for a nonzero stability parameter $\FI.$ 
%By \Cref{defn:cohcats-microlocal}, this 2-category is presented as the module 2-category
%\begin{equation}\label{eq:bside-2cato-naive}
%\mu\CohCat(\LL_G(\FI,\mass)):=\Mod_{\Coh_{F^{-1}(\LL_G(\FI,\mass)))}(X^{\mass\bdd}\times_{\AA^n/G}X^{\mass\bdd})}(\St).
%\end{equation}
We are now ready to study the 2-category $\mu\CohCat(\LL_G(\FI,\mass))$ of microlocal coherent sheaves of categories on the category $\cO$ Lagrangian. To apply \Cref{defn:cohcats-microlocal}, we need to present $\LL_G(\FI,\mass))$ as an open subset of a conormal Lagrangian. There is an obvious choice, namely the open inclusion
\begin{equation}\label{eq:skel-openembedding}
\LL_G(\FI,\mass)\hookrightarrow \LL_G(0,\mass)=\LL_G^{\mass\bdd}.
\end{equation}
Applying definition \Cref{defn:cohcats-microlocal}, we may therefore write
\[
\mu\CohCat(\LL_G(\FI,\mass))\simeq\Mod_{
\Coh_{F^{-1}(\LL_G(\FI,\mass))}
(X^{m\bdd}\times_{\AA^n/G}X^{m\bdd})
}
(\St).
\]
\begin{lemma}\label{lem:closed-singsupp-is-cohcats}
Let 
$
\Lambda:=\LL_G(0,\mass)\setminus \LL_G(\FI,\mass)
$
be the closed complement of the embedding \Cref{eq:skel-openembedding}. Then
   % there is a Morita equivalence between the monoidal categories
   % \[
    $\Coh_{F^{-1}(\Lambda)}(X^{\mass\bdd}\times_{\AA^n/G}X^{\mass\bdd})\subset
    \Coh(X^{\mass\bdd}\times_{\AA^n/G}X^{\mass\bdd})$
%    \qquad\text{and}\qquad
    is the ideal generated by
    $\Coh(X^{\mass\bdd,\FI\unstab}\times_{\AA^n/G}X^{\mass\bdd,\FI\unstab}).$
%    \]
\end{lemma}
\begin{proof}
    By \Cref{prop:gen-by-idempotents}, $\Coh_{F^{-1}(\Lambda)}(X^{\mass\bdd}\times_{\AA^n/G}X^{\mass\bdd})$ is the smallest ideal in $\Coh(X^{\mass\bdd}\times_{\AA^n/G} X^{\mass\bdd})$ containing the matrix idempotents $E_\sgn$ for $\sgn\in \mass\bdd,\FI\unstab.$ But these matrix idempotents are monoidal generators of %subcategory generated by these idempotents is precisely 
    the monoidal subcategory
    $Coh(X^{\mass\bdd,\FI\unstab}\times_{\AA^n/G}X^{\mass\bdd,\FI\unstab}).$
\end{proof}

%Our main result of this section is an explicit 
%algebraic presentation of this 2-category in terms of the matrix idempotents $E_\sgn$ of \Cref{defn:matrix-idempotent}.
%\begin{proposition}
%    Let $\cI_G(\mass,\FI)\subset \cA_G^{\mass\bdd}$ denote the smallest ideal
%    in $\cA_G^{\mass\bdd}$ containing the matrix idempotents $\{E_\sgn\mid \sgn\in \FI\unstab\},$ and let $\cA_G(\FI,\mass)$ denote the quotient monoidal category
%    \[
%    \cA_G(\FI,\mass):=\cA_G^{\mass\bdd}/\cI_G(\FI,\mass).
%    \]
%    Then there is an equivalence
%    \[
%    \mu\CohCat(\LL_G(\FI,\mass))\simeq \Mod_{\cA_G(\FI,\mass)}(\St)
%    \]
 %   presenting the B-side 2-category $\cO$ $\mu\CohCat(\LL_G(\mass,\FI))$ as the 2-category of module categories for the monoidal category $\cA_G(\FI,\mass).$
%\end{proposition}
%\begin{proof}
%    {\color{red}Add here.}
%\end{proof}
\begin{corollary}\label{cor:mucohschob-as-quotient}
    There is an equivalence
    \[
    \mu\CohCat(\LL_G(\FI,\mass))\simeq \Mod_{\frac{\cA_G^{\mass\bdd}}{\cA_G^{\mass\bdd,\FI\unstab}}}(\St)
    \]
    presenting $\mu\CohCat(\LL_G(\FI,\mass))$ as the 2-category of module categories for the monoidal category obtained as the quotient of $\cA_G^{\mass\bdd}$ by the ideal generated by the subcategory $\cA_G^{\mass\bdd,\FI\unstab}.$
\end{corollary}
\begin{proof}
    This follows immediately from \Cref{defn:cohcats-microlocal} and \Cref{lem:closed-singsupp-is-cohcats}.
\end{proof}

\begin{example}\label{ex:2catO-B}
%We will begin on the spectral side, noting that $F^\vee=\GG_m\hookrightarrow (\GG_m)^2,$ dual to the multiplication map $m$, is the diagonal embedding, so that $F^\vee=\GG_m$ acts on $(\GG_m)^2$ by scalar multiplication. 
We return to the situation of \Cref{ex:2catso-intro}, where $F^\vee\simeq \GG_m\hookrightarrow (\GG_m)^2$ is the diagonal torus.

Before turning on the stability or attraction parameters, the skeleton $\LL_{F^\vee}$ is the image in $\rT^*(\AA^2/\GG_m)$ of the union of conormals to the zero section, $x$-axis, $y$-axis, and origin. The coherent 2-category admits a presentation
\[
%\mu\Coh\Cat(\LL_{F^\vee}) = 
\CohCat_{\LL_{F^\vee}}(\AA^2/F^\vee) = \Mod_{\cA_F^\vee}(\St),
\]

as modules over the monoidal category $\cA:=\Coh^{\GG_m}(X\times_Y X),$ where we write 
\[X= (\AA^2 \sqcup \AA^1_x \sqcup \AA^1_y \sqcup 0)\to \AA^2 = Y,\]
with the maps given by embeddings of closures of toric strata. In other words, presenting each component in the fiber product $X\times_YX$ as the appropriate shifted normal bundle, we have 
\begin{equation} \label{eq:monoidalcat-bside-example}
\cA \simeq 
\Coh^{\GG_m}\left(
\begin{array}{llll}
\AA^2 & \AA^1_x & \AA^1_y & 0\\
\AA^1_x & 
\rN_{\AA^1_x}[-1]\AA^2 & 0 & \rN_0[-1]\AA^1_y\\
\AA^1_y & 0 & \rN_{\AA^1_y}[-1]\AA^2 & \rN_0[-1]\AA^1_x\\
0 & \rN_0[-1]\AA^1_y & \rN_0[-1]\AA^1_x & \rN_0[-1]\AA^2
\end{array}
\right).
\end{equation}
%This matrix can be given a more ``symplectic'' presentation, by using Koszul duality to replace each $(-1)$-shifted normal bundle with a $2$-shifted conormal bundle. In general, this operation will be necessary to understand passage to the stable locus, which is a subset of the cotangent (rather than tangent) bundle of $\CC^n/F^\vee$, but we will postpone it till later.

If we pick the stability parameter $\mass\in \ffv_\ZZ=\ZZ$ to be positive, then the unstable locus in $\rT^*(\AA^2/\GG_m)$ is $\rT^*_0(\AA^2),$ so that passage to the $\mass$-semistable locus is restriction to the open subset $\cU := \rT^*\PP^1 = \rT^*((\AA^2\setminus \{0\})/\GG_m)$ inside $\rT^*(\AA^2/\GG_m).$ 

%\begin{remark}
%    The matrix \Cref{eq:monoidalcat-bside-example} can be given a ``symplectic'' presentation by using Koszul duality to replace each $(-1)$-shifted normal bundle with a 2-shifted normal bundle. In general, this procedure will be necessary to understand passage to the semistable locus, since in general the open subset $\cU$ of semistable points will not itself be a cotangent bundle. Note that this Koszul duality would have been necessary if we were to replace the stability parameter $\mass$ with $-\mass,$ since in that case the unstable locus would be not $\rT^*_0(\CC^2)$ but rather the zero-section $\rT^*_{\CC^2}\CC^2.$
%\end{remark}

Passing to the $\mass$-semistable locus is a localization of the monoidal category $\cA,$ so that we can write $\mu\CohCat(\LL_{F^\vee}(\mass,0))\simeq \Mod_{\cA|_\cU}(\St),$
where
\begin{equation} \label{eq:monoidalcat-bside-example-localized}
\cA|_\cU \simeq 
\Coh^{\GG_m}\left(
\begin{array}{llll}
\AA^2\setminus 0 & \AA^1_x\setminus 0 & \AA^1_y\setminus 0 & \emptyset\\
\AA^1_x\setminus 0 & \rN_{\AA^1_x\setminus 0}[-1](\AA^2\setminus 0) & \emptyset & \emptyset\\
\AA^1_y\setminus 0 & \emptyset & \rN_{\AA^1_y\setminus 0}[-1](\AA^2\setminus 0) & \emptyset\\
\emptyset & \emptyset & \emptyset & \emptyset
\end{array}
\right).
\end{equation}

Using the fact that the $\GG_m$-action on the spaces in \Cref{eq:monoidalcat-bside-example-localized} is free, and removing the final row and column (whose values are all the zero category $\Coh(\emptyset)\simeq 0$ --- not to be confused with the category $\Coh(0)\simeq \Mod_\kk$), we may rewrite \Cref{eq:monoidalcat-bside-example-localized} as

\begin{equation} \label{eq:monoidalcat-bside-tp1}
\cA|_\cU \simeq 
\Coh\left(
\begin{array}{lll}
\PP^1 & 0 & \infty \\
0 & T_0[-1]\PP^1 & \emptyset \\
\infty & \emptyset & T_{\infty}[-1](\PP^1)
\end{array}
\right),
\end{equation}
where we have identified the images of the $x$- and $y$-axes with the points $0,\infty \in \PP^1 = (\AA^2\setminus 0)/\GG_m$. This monoidal category evidently admits a convolution presentation as before, where now
\[
\cA|_\cU \simeq \Coh\left((\PP^1 \sqcup 0 \sqcup \infty)\times_{\PP^1}(\PP^1\sqcup 0 \sqcup \infty)\right).
\]

Finally, we impose a nonzero attraction parameter $\FI,$ passing from the ``TIE fighter'' Lagrangian
\[
\LL_{F^\vee}(\mass,0) = \PP^1\cup \rT^*_0\PP^1 \cup \rT^*_\infty\PP^1
\]
to the closed subspace
\[
\LL_{F^\vee}(\mass,\FI) = \PP^1\cup \rT^*_0\PP^1.
\]
At the level of categories, this has the effect of deleting the third row and column of \Cref{eq:monoidalcat-bside-tp1}, so that we have at last
\[
\mu\CohCat(\LL_{F^\vee}(\mass,\FI)) = \Mod_{\cA'}(\St)
\]
where
\begin{equation}\label{eq:2cato-tp1}
\cA' = \Coh\left((\PP^1\sqcup 0) \times_{\PP^1} (\PP^1\sqcup 0)\right) \simeq
\Coh\left(
\begin{array}{ll}
\PP^1 & 0\\
0 & T_0[-1]\PP^1
\end{array}
\right).
\end{equation}
\end{example}

\begin{remark}
    \Cref{ex:2catO-B} is misleading in some ways, due to the simplicity of the stable locus $\rT^*((\AA^2\setminus\{0\})/\GG_m).$ As a result, the theory of coherent singular support conditions was not really necessary, as the singular support condition could be understood just as a usual support condition over $\AA^2/\GG_m.$ 
    However, if we had picked the opposite sign for our stability parameter, then
    \Cref{eq:monoidalcat-bside-example-localized} would be replaced with a ``purely microlocal'' --- that is, disjoint from the zero-section --- localization. The easiest way to compute the result of this localization on \Cref{eq:monoidalcat-bside-example} would be to apply Koszul duality to replace the $(-1)$-shifted normal bundles with $2$-shifted conormal bundles, and then apply the fact that Koszul duality exchanges singular-support conditions (and microlocalizations) with usual support conditions (and usual localizations).
\end{remark}

\section{Proof of Theorems A \& B}
We now prove our first main theorem, an equivalence between the (Gale dual) pair of 2-categories $\cO$ we have defined:
\begin{proof}[Proof of \Cref{mainthm:2cats-O}]
    Our starting point is the equivalence of 2-categories \Cref{eq:gmh-general} proved as \cite{GMH}*{Theorem G}, namely
    \begin{equation}\label{eq:gmh-inproof}
    \PervCat_{\LL_G}(\CC^n/G)\simeq \CohCat_{\LL_{F^\vee}}(\CC^n/F^\vee),
    \end{equation}
    which sends the projective $\cP^\sgn_G$ in the left-hand side to the 
    $\Coh(X_{F^\vee}\times_{\CC^n/F^\vee} X_{F^\vee})$-module category 
    $\cS_{F^\vee}^\sgn:=\Coh(X_{F^\vee}\times_{\CC^n/F^\vee} X_{F^\vee}^\sgn).$
    Observe that for $\sgnsub\subset \sgnsub,$ the subcategory of $\CohCat_{\LL_{F^\vee}}(\CC^n/F^\vee)$ generated by the $\cS_{F^\vee}^\sgn$ for $\sgn\in \sgnsub$ is $\CohCat_{\LL_{F^\vee}^\sgnsub}(\CC^n/F^\vee),$ where we write $\LL_{F^\vee}^\sgnsub=\bigcup_{\sgn\in\sgnsub}\LL_{F^\vee}^\sgn$ for the union of components of $\LL_{F^\vee}$ corresponding to $\sgn\in\sgnsub.$
    
    Now recall that %\Cref{lem:mupervcat-subquotient} gives a description of 
    $\mu\PervCat(\LL_G(\FI,\mass))$ is realized as the subquotient 
    \begin{equation}\label{eq:thma-subquotient}
    \mu\PervCat(\LL_G(\FI,\mass))\simeq 
    \frac{\langle \cP_G^\alpha\mid\text{$\alpha$ $\FI$-semistable}\rangle}{\langle \cP_G^\beta\mid \text{$\beta$ $\FI$-semistable, $\mass$-unbounded}\rangle}
    \end{equation}
    of the left-hand side of \Cref{eq:gmh-inproof}.

    By the above observations, \Cref{eq:thma-subquotient} is equivalent to the quotient 2-category 
    \begin{equation}\label{eq:thma-quotientcoh}
    \CohCat_{\LL^{\FI\bdd}}(\CC^n/F^\vee)/\CohCat_{\LL^{\FI\bdd,\mass\unstab}}(\CC^n/F^\vee),
    \end{equation}
    and by \Cref{cor:mucohschob-as-quotient}, the quotient \Cref{eq:thma-quotientcoh} is equivalent to the 2-category $\mu\CohCat(\LL_{F^\vee}(\mass,\FI)).$
    By construction, this equivalence has identified $\cP^\sgn_G$ with $\cS^\sgn_{F^\vee}.$
\end{proof}

%\subsection{Proof of \Cref{mainthm:hp-decategorification}}
    We now proceed to the proof of \Cref{mainthm:hp-decategorification}. This theorem is really two theorems, with equivalences \Cref{eq:mainthm-decat-a} and \Cref{eq:mainthm-decat-b} identifying the decategorifications of the A- and B-side 2-categories, respectively. We will deal with each of these separately, beginning with the B-side.

    \begin{notation}
        From now on, $\kk$ is assumed to be a field of characteristic 0. For a $\kk$-algebra or $\kk$-linear category, we write $(-)_{\ZZ/2}:=(-)\otimes_\kk \kk((u))$ for the 2-periodization, where $u$ is a variable of degree 2.
    \end{notation}

    Our main tool will be an identification of the periodic cyclic homology of coherent-sheaf categories of stacks. 
    From \cite{Feigin-Tsygan}, when $X$ is an affine scheme there is an identification $\HP(\Perf(X))\simeq \sC^*_{\dR}(X)_{\ZZ/2}$ between the periodic cyclic homology of $\Perf(X)$ and the 2-periodized infinitesimal cohomology of $X$.
    In \cite{Preygel-loopspaces}, this result was generalized to the case where we replace $\Perf$ with $\Coh$ (allowing $X$ more generally to be a quasi-compact and separated algebraic space): in this case the statement remains true if we replace de Rham cochains $\sC^*_{\dR}(X)_{\ZZ/2}$ with Borel-Moore chains $\sC_*^{\BM,\dR}(X)_{\ZZ/2}.$
    
   Generalizing from algebraic spaces to stacks presents additional difficulties. This case was studied in \cite{Chen-Koszul}, which gave a periodic-cyclic version of the Atiyah-Segal completion theorem.
   %, due to the fact that if $X/G$ is a quotient stack, then $\HP(\Perf(X/G))$ is linear over the 2-periodic character ring
   %which have been addressed in \cite{Chen-localization}, which proved a version of the Atiyah-Sgeal
   \begin{notation}
       Let $G$ be a reductive group.
       %acting on a smooth quasi-projective variety $X.$ Then $\HP(\Perf(X/G))$ is a module over 
       If $A$ is an algebra linear over the 2-periodic representation ring
       $\HP(\Perf(BG))=\cR(G)_{\ZZ/2}$ of $G,$ we write $A_{\hat{e}}$ for the completion of $A$ at the augmentation ideal.
   \end{notation}

    \begin{theorem}[\cite{Chen-localization}*{Theorem 4.3.2}]\label{thm:hp-atiyahsegal}
        Let $G$ be a reductive group acting on a smooth quasi-projective variety $X.$ Then $\HP(\Perf(X/G))$ is linear over $\HP(\Perf(BG)),$ and
        there is an equivalence
        \begin{equation}\label{eq:perf-atiyahsegal}
        \HP(\Perf(X/G))_{\hat{e}}\simeq \sC^*_{\dR}(X/G)_{\ZZ/2}
        \end{equation}
        between the completion of $\HP(\Perf(X/G))$ and 2-periodic de Rham cochains on $X/G.$
    \end{theorem}
%    By tensoring with the dualizing sheaf, we obtain a coherent-sheaf version as well:
%    \begin{corollary}
%        Let $X,G$ as in \Cref{thm:hp-atiyahsegal}. Then there is an equivalence
%        \[
%        \HP(\Coh(X/G))_{\hat{e}}\simeq C_*^{\BM,\dR}(X/G)_{\ZZ/2}.
%        \]
%    \end{corollary}

    For a smooth variety, the categories of coherent sheaves and perfect complexes are equivalent. But we need a slight enhancement of this calculation beyond the smooth case.
    %Using a trick from \cite{Elmanto-Sosnilo}, we may extend this result slightly beyond the smooth case.
    \begin{corollary}\label{cor:atiyahsegal-quasismooth}
%        Let $G$ be a reductive group acting on an affine derived scheme $X$ such that that underlying classical scheme $X^{cl}$ is smooth. Then there is an equivalence
        Let $X=\bigsqcup X^\sgn\to \CC^n$ as in \Cref{defn:xsgn}, equipped with its action of $G\subset (\GG_m)^n$ as usual, and write $\fX$ for the fiber product $\fX:=X\times_{\CC^n}X$. Then there is an equivalence
        \[
        \HP(\Coh(\fX/G))_{\hat{e}}\simeq \sC_*^{\BM,\dR}(\fX/G)_{\ZZ/2}
        \]
        between the completion of $\HP(\Coh(\fX/G))$ and the 2-periodic Borel-Moore chains on $\fX.$
    \end{corollary}
    \begin{proof}
        First we observe that the statement is true when $G=D=(\GG_m)^n,$ since by combining Koszul duality with a shearing equivalence, we have
        \[
        \Coh(\AA^n[-1]/D)\simeq \Perf(\AA^n[2]/D)\simeq \Perf(\AA^n/D),
        \]
        so we are reduced to the case when $\fX$ is smooth. From here, we can reduce to the case of general $G$ by deequivariantization.
        %Since $\fX$ is as in \cite{HLP-Hodge}*{Amplification 1.25}, the results of \cite{HLP-Hodge} show that 
        %the invariant $\HP(\Coh(-))$ is insensitive to the derived structure of $\fX,$ 
        %so that the pushforward along the closed immersion $\fX^{cl}\to \fX$ induces an equivalence $\HP(\Coh(\fX^{cl}/G))\simeq\HP(\Coh(\fX/G)).$ 
        %As the stack $\fX^{cl}/G$ is is smooth, we may apply the inverse dualizing complex to reduce to the case of \Cref{eq:perf-atiyahsegal}, and then apply the dualizing complex to return from cochains to Borel-Moore chains.
        %***
        %As $\HP$ is a truncating invariant, there is an equivalence $\HP(\Perf(X/G))\simeq \HP(\Perf(X^{cl}/G)).$ Since de Rham cochains also depends only on the underlying classical scheme, we deduce the result by applying \Cref{thm:hp-atiyahsegal} to $X^{cl}/G.$
    \end{proof}
%\begin{corollary}\label{cor:hpcoh-cbm}
%    Let $G$ be a reductive group acting on on affine derived scheme $X$ wuch that the underlying classical scheme $X^{cl}$ is smooth. Then there is an equivalence
%    \[
%        \HP(\Coh(X/G))_{\hat{e}}\simeq C_*^{\BM,\dR}(X/G)_{\ZZ/2}.
%    \]
%\end{corollary}
%\begin{proof}
%
%\end{proof}
%    Pushforward along the closed embedding $i:X^{cl}\hookrightarrow X$ provides the vertical maps in the diagram
%    \[
%    \begin{tikzcd}
%        \HP(\Coh(X^{cl}/G))_{\hat{e}}\arrow[r, "\sim"] \arrow[d, "\sim" {anchor=north, rotate=90}] &  C^{\BM,\dR}_*(X^{cl}/G) \arrow[d,"\sim" {anchor=north, rotate=90}]\\
%        \HP(\Coh(X/G))_{\hat{e}} \arrow[r,dashed] &  C^{\BM,\dR}_*(X/G),
%    \end{tikzcd}
%    \]
%    which are equivalences since both $\HP(\Coh(-))$ and $C^{\BM,\dR}_*(-)$ are insensitive to derived structure, and the desired dashed equivalence is obtained as the composition of these.

    The only other fact we will need about periodic cylic homology is the following:
    \begin{lemma}
        $\HP(-)$ is a localizing invariant: it takes exact sequences of stable categories to fiber sequences.\footnote{We follow the convention of \cite{Land-Tamme-pullbacks}, which differs from that of \cite{BGT} in not requiring localizing invariants to commute with filtered colimits. In the conventions of \cite{BGT}, Hochschild and cyclic homology are localizing but periodic cyclic homology is not.}
    \end{lemma}
    \begin{proof}
        For Hochschild homology this is proved in \cite{Keller-cyclic} or \cite{Blumberg-Mandell-localization}*{Theorem 7.1}. The result for $\HP$ follows, as $\HP$ is obtained from $\HH$ by the exact functor $(-)^{tS^1}.$
    \end{proof}
    
\begin{proof}[Proof of \Cref{mainthm:hp-decategorification} (B-side)]
    %By \Cref{cor:mucohschob-as-quotient}, the 2-category $\mu\CohCat(\LL_{F^\vee}(\mass,\FI))$ is the 2-category of module categories for the monoidal category 
    The canonical generator $\cS$ of $\mu\CohCat(\LL_{F^\vee}(\mass,\FI))$
    has endomorphism category
    $\cA_{F^\vee}^{\FI\bdd}/\cJ$,
    %which fits into an exact sequence
    %\[
    %\cI\to \cA_{F^\vee}^{\FI\bdd} \to \cA_{F^\vee}^{\FI\bdd}/\cI\simeq \mu\CohCat(\LL_{F^\vee}(\mass,\FI)),
    %\]
    where $\cJ\hookrightarrow \cA_{F^\vee}^{\FI\bdd}$ is the ideal generated by 
    $\cA_{F^\vee}^{\FI\bdd,\mass\unstab}$
    (and the monoidal categories $\cA_{F^\vee}$ are as in \Cref{notation:b-endomorphism-cats}).
    We therefore conclude that 
    %$\tr^{tS^1}(\mu\CohCat(\LL_{F^\vee}(\mass,\FI)))$ 
    $\HP(\End(\cS))$
    is equivalent to the quotient of
    $\HP(\cA_{F^\vee}^{\FI\bdd})$ by the ideal $\HP(\cJ)$ generated by
    $\HP(\cA_{F^\vee}^{\FI\bdd,\mass\unstab})$.
    %By \Cref{lem:hp-is-kbm}, for $\sgnsub\subset \sgvect,$ we have $HP(\cA^{}$
    %By \Cref{lem:hp-is-kbm}, this is the quotient of
    Observe that the augmentation ideal for $\HP(BF^\vee)$ is contained inside of $\HP(\cJ),$ so that completion at the augmentation ideal does not affect the resulting quotient algebra.
    By \Cref{cor:atiyahsegal-quasismooth},
    the completion of this algebra at the augmentation ideal 
    %of $\HP(BF^\vee)$ is the quotient of
    %$\sC_*^{\BM,F^\vee}(X^{\FI\bdd}\times_{\AA^n}X^{\FI\bdd})$
    %by the ideal generated by
    %$\sC_*^{\BM,F^\vee}(X^{\FI\bdd,\mass\unstab}\times_{\AA^n}X^{\FI\bdd,\mass\unstab}).$
%    This monoidal category lives over 
%    \[\HP(BF^\vee)\simeq \sK_*^{\BM,F^\vee}(\pt)\simeq\Rep(F^\vee)\otimes\CC((u)),\]
%    and we may complete at the augmentation ideal. 
%    Since the augmentation ideal is contained inside of the $\HP(\cJ),$ this completion does not affect the monoidal category.
%    But by the Atiyah-Segal completion theorem, the monoidal category resulting from this completion 
is equivalent to the quotient of the algebra
    $\sC_*^{\BM,\dR}(X^{\FI\bdd}_{F^\vee}\times_{\AA^n/F^\vee}X^{\FI\bdd}_{F^\vee})_{\ZZ/2}$
    by the ideal generated by
    $\sC_*^{\BM,\dR}(X^{\FI\bdd,\mass\unstab}_{F^\vee}\times_{\AA^n/F^\vee}X^{\FI\bdd,\mass\unstab}_{F^\vee})_{\ZZ/2}.$
   %The Chern character now identifies these Borel-Moore K-theory groups with Borel-Moore homologies, and we recover 
   This agrees with the description of de Rham category $\cO$ from \Cref{thm:derham-o-quotient}, up to 2-periodization.
    %This almost agrees with the description of de Rham category $\cO$ given in \Cref{thm:derham-o-quotient}, except that the algebra and ideal there are equivariant Borel-Moore homologies, whereas \Cref{lem:hp-is-kbm} identifies the algebra and ideal coming from periodic cyclic homology as equivariant Borel-Moore K-theories. The Atiyah-Segal theorem therefore gives a surjective map
%    \begin{equation}\label{eq:kthry-to-homology}
%    \tr^{tS^1}(\mu\CohCat(\LL_{F^\vee}(\mass,\FI)))\simeq 
%    \HP(\cA_{F^\vee}^{\FI\bdd})/\HP(\cJ)\to 
%    (A/I)\otimes \CC((u)),
    %\cO^{\dR}_{\ZZ/2}(\htvar_{F^\vee}(\mass), \FI)
%    \end{equation}
%    where the right-hand side is the 2-periodization of the algebra described in \Cref{thm:derham-o-quotient}, whose kernel is the augmentation ideal of the representation ring of $F^\vee.$
%    Since this kernel is contained within $\HP(\cJ),$ we conclude that \Cref{eq:kthry-to-homology} is an equivalence.
\end{proof}

On the A-side, we begin with the non-microlocal computation.
%we will prove \Cref{mainthm:hp-decategorification} in stages, following the steps involved in the construction of $\mu\PervCat(\LL(\FI,\mass)).$
%the main theorem will follow directly once we have a complete understanding of the 1-dimensional case.
%In what follows, we write $\Perv_{\LL}(\CC)$ for the category of perverse sheaves 
%We begin with the non-equivariant calculation.
%

\begin{lemma}
    Let $\cP=\bigoplus \cP^\sgn$ be the sum of generating objects of $\PervCat_\LL(\CC^n),$ and similarly $P = \bigoplus P^\sgn$ the sum of the projectives in $\Perv_\LL(\CC^n).$ Then there is an equivalence
    \begin{equation}\label{eq:hp-nonmicrolocal}
    \HP(\End_{\Perv_\LL(\CC^n)}(\cP))\simeq \End_{\Perv_\LL(\CC)}(P)_{\ZZ/2}.
    \end{equation}
    Moreover, this equivalence intertwines the $G$-actions.
\end{lemma}
\begin{proof}
    It is sufficient to prove the case $n=1,$ since the general case is recovered from this one as a tensor product. Let $\cA=\End(\cP)$ and $A=\End(P).$ From the spectral decomposition
    \begin{equation}\label{eq:spectral-for-basic-inproof}
    \cA\simeq \Coh^{\GG_m}\left((\AA\sqcup 0)\times_{\AA}(\AA\sqcup 0)\right)\simeq \Coh^{\GG_m}\left(
    \begin{array}{ll}  
    \AA^1 & 0  \\
       0  &  \Conorm_0\AA^1
    \end{array}
    \right),
    \end{equation}
    we can see that there is an equivalence $\HP(\cA)\simeq \HH(\cA)\otimes_\kk \kk((u))$ (for which see for instance \cite{Chen-localization}*{Example 1.0.5}), so it is sufficient to describe an isomorphism $A\simeq \HH(\cA).$

    The map $A\to \HH(\cA)$ may be defined in the obvious way: observe the algebra $A$ is generated by the vertex idempotents $e_\Phi,e_\Psi$ and the maps $u,v$ with $1-uv,1-vu$ invertible, and define the map $A\to \HH(\cA)$ by sending the vertex idempotents to the classes of the identity maps $[\id_\Phi],[\id_\Psi],$ and the maps $u,v$ to the classes of the generating maps between $\Phi$ and $\Psi.$ The defining exact triangles for the twist and cotwist ensure that $1-uv,1-vu$ map to these invertible elements, so that the map is well-defined.

    To see that this map is an equivalence, observe that the Hochschild homology of the matrix category $\Cref{eq:spectral-for-basic-inproof}$ may be written as
    \[
    \left(
    \begin{array}{ll}
        \kk[T^\pm] & \kk[T^\pm] \\
        \kk[T^\pm] & \kk[T^\pm]
    \end{array}
    \right).
    \]
    The explicit description of $P_\Phi,P_\Psi$ given in \Cref{ex:projectives-formula} shows that $A$ has an identical expression, and it is easy to see that the obvious bases of $A$ and $\HH(\cA)$ are sent to each other by this map.

    For the statement about $G$-actions, it is necessary to show that the universal twist map $\cT:\kk[\pi_1 G]\to Z(\PervCat_\LL(\CC))$ induces on $\HP(\cA)\simeq A_{\ZZ/2}$ the monodromy automorphism. This follows immediately from the formula for the universal twist, and that our map $A\to \HH(\cA)$ was defined so that the monodromy automorphisms were sent to the classes of the twist and cotwist.
\end{proof}

\begin{corollary}\label{cor:hp-equivariant-nonmicrolocal-a}
    There is an equivalence
    \begin{equation}\label{eq:hp-nonmicrolocal-equivariant}
    \HP(\End_{\PervCat_{\LL_G}(\CC^n/G)}(\cP))\simeq \End_{\Perv_{\LL_G}(\CC^n/G)}(P)_{\ZZ/2}
    \end{equation}
\end{corollary}
\begin{proof}
 Each of the algebras in \Cref{eq:hp-nonmicrolocal-equivariant} is obtained from the corresponding algebra in \Cref{eq:hp-nonmicrolocal} by trivializing the respective action of $\kk[\pi_1G],$ and we saw that the equivalence \Cref{eq:hp-nonmicrolocal} intertwined these actions.
\end{proof}

\begin{proof}[Proof of \Cref{mainthm:hp-decategorification} (A-side)]
The theorem now follows immediately from 
\Cref{cor:hp-equivariant-nonmicrolocal-a}
by matching the subquotient presentations of \Cref{cor:betti-cato-as-subquotient} and \Cref{defn:betti-2cat-o} and applying the fact that $\HP$ is a localizing invariant.
\end{proof}

\bibliographystyle{plain}
\bibliography{refs}

\end{document}